\documentclass[final,1p,times,authoryear]{elsarticle}

\makeatletter
\def\ps@pprintTitle{%
 \let\@oddhead\@empty
 \let\@evenhead\@empty
 \def\@oddfoot{}%
 \let\@evenfoot\@oddfoot}
\makeatother

\usepackage[T1]{fontenc} 
\usepackage[utf8]{inputenc} 
\usepackage{amsmath,amsthm,amssymb,amsfonts} 
\usepackage[colorlinks,plainpages]{hyperref} 
\usepackage{mathtools}
\usepackage[all]{xy} 
\usepackage{stmaryrd}
\usepackage{cleveref}
\crefname{subsection}{subsection}{subsections}
\usepackage{paralist} 

\theoremstyle{plain}
\newtheorem{thm}{Theorem}[section]
\newtheorem*{thm*}{Theorem}
\newtheorem{lem}[thm]{Lemma}
\newtheorem{cor}[thm]{Corollary}
\newtheorem{prp}[thm]{Proposition}

\theoremstyle{definition}
\newtheorem{dfn}[thm]{Definition}
\newtheorem{ntn}[thm]{Notation}
\newtheorem{con}[thm]{Convention}
\newtheorem{rem}[thm]{Remark}

\newtheorem{exa}[thm]{Example}
\newtheorem{prb}[thm]{Problem} 
\newtheorem{ass}[thm]{Assumption}

\usepackage[noend]{algorithmic} 

\makeatletter
\newcounter{algorithmbis}
\renewcommand{\thealgorithmbis}{\thesection.\arabic{thm}}
\def\algorithmbis{\@ifnextchar[{\@algorithmbisa}{\@algorithmbisb}}
\def\@algorithmbisa[#1]{%
 \refstepcounter{thm}
 \trivlist
 \leftmargin\z@
 \itemindent\z@
 \labelsep\z@
 \item[\parbox{\textwidth}{%
 \hrule 
 \vspace{0.1cm}
 \noindent\strut\textbf{Algorithm \thealgorithmbis} #1
 \vspace{0.1cm}
 \hrule 
 }]\hfil\vskip0em%
} 
\def\@algorithmbisb{\@algorithmbisa[]}

\makeatother

\newcommand{\ideal}[1]{\left\langle#1\right\rangle}
\newcommand{\lideal}[2]{\prescript{}{#1}{\left\langle#2\right\rangle}}
\newcommand{\rideal}[2]{{\left\langle#2\right\rangle_{#1}}}
\newcommand{\lrideal}[2]{\text{}_{#1}{\left\langle#2\right\rangle\text{}_{#1}}}
\newcommand{\ul}[1]{\underline{#1}}
\newcommand{\ol}[1]{\overline{#1}}
\newcommand{\bracket}[1]{[\mathbf{#1}]}

\DeclareMathOperator{\gr}{Gr}

\DeclareMathOperator{\lE}{le}

\DeclareMathOperator{\lt}{lt} 
\DeclareMathOperator{\tail}{tail}

\DeclareMathOperator{\lcomp}{lcomp}
\DeclareMathOperator{\spoly}{spoly}

\DeclareMathOperator{\syz}{syz}

\DeclareMathOperator{\End}{End}

\DeclareMathOperator{\ZZ}{\mathbb{Z}}
\DeclareMathOperator{\RR}{\mathbb{R}}
\DeclareMathOperator{\CC}{\mathbb{C}}
\DeclareMathOperator{\NN}{\mathbb{N}}
\DeclareMathOperator{\KK}{\mathbb{K}}

\DeclareMathOperator{\QQ}{\mathbb{Q}}
\DeclareMathOperator{\K}{\mathbb{K}}

\DeclareMathOperator{\Mon}{{Mon}}
\DeclareMathOperator{\SMon}{{SMon}}
\DeclareMathOperator{\lm}{{lm}}
\DeclareMathOperator{\lc}{{lc}}
\DeclareMathOperator{\lec}{{ele}} 
\DeclareMathOperator{\mO}{\mathcal{O}}
\DeclareMathOperator{\mD}{\mathcal{D}}

\DeclareMathOperator{\irr}{irr}

\DeclareMathOperator{\mM}{\mathcal{M}}

\DeclareMathOperator{\Tn}{\K\langle \ul x\rangle}
\DeclareMathOperator{\Pn}{\K[\ul x]}

\begin{document}

\begin{frontmatter}

\title{Gröbner basics for mixed Hodge modules}

\author{Cornelia Rottner}
\address{TU Kaiserslautern\\
Department of Mathematics\\
67663 Kaiserslautern\\
Germany}
\ead{rottner@mathematik.uni-kl.de}
\ead[url]{http://www.mathematik.uni-kl.de/~rottner}

\author{Mathias Schulze}
\address{TU Kaiserslautern\\
Department of Mathematics\\
67663 Kaiserslautern\\
Germany}
\ead{mschulze@mathematik.uni-kl.de}
\ead[url]{http://www.mathematik.uni-kl.de/~mschulze}

\begin{abstract}
We develop a Gröbner basis theory for a class of algebras that
generalizes both PBW-algebras and rings of differential algebras on
smooth varieties.
Emphasis lies on methods to compute filtrations and graded structures
defined by weight vectors.
The approach is tailored for bifiltered $\mD$-modules satisfying properties
of mixed Hodge modules.
As a key ingredient in functors of such modules our theory applies to
compute the order filtration on pieces of a V-filtration.
\end{abstract}

\begin{keyword}
Gröbner basis, D-module, Hodge module, PBW-algebra
\MSC[2010] 13P10 \sep 14F10 \sep 16W70 \sep 32S35
\end{keyword}

\end{frontmatter}

\section*{Introduction}

Most algorithms in algebraic $\mD$-module theory are based on translating $\mD$-module theoretic constructions to computationally accessible operations over the ring of differential operators on the affine $n$-space which agrees with the Weyl algebra (see e.g. \cite{OT2001}). However, this approach limits the underlying varieties often to affine $n$-spaces. To deal with such constructions for general smooth varieties $X$, we work on a covering of $X$ by affine open sets $U$ on which the tangent sheaf is $\mO_U$-free and glue the results. Each such $U$ can be seen as a closed set in an affine $n$-space and we lift a basis of the tangent sheaf to elements $y_1,\dots,y_m$ of the Weyl algebra. 
Then $\mD_X(U)$ becomes the factor algebra of the free associative $\CC$-algebra $\CC\langle x_1,\dots,x_n,y_1,\dots,y_m\rangle$ by the two-sided ideal generated by the defining ideal $I(U)\subseteq \CC[x_1,\dots,x_n]$ of $U$ and the natural commutation relations of the variables. We refer to such an algebra as a coordinate system algebra.
As opposed to Weyl algebras, coordinate system algebras are in general not quotients of PBW-algebras. 
While there is a Gröbner basis theory for PBW-algebras, we are not aware of a well-developed generalization covering coordinate system rings.
We remark that algorithms by \cite{OakuAffine} for coordinate system algebras are not sufficiently general for our purpose (see \Cref{935}).

In this article we develop a Gröbner basis theory for so-called PBW-reduction-algebras which form a common generalization of PBW-algebras and coordinate system algebras (see \Cref{800}). Combined with gluing techniques this allows for $\mD$-module calculations on general smooth varieties \citep{Rottner}. 
Such an algebra is a certain quotient of a free associative $\K$-algebra of type $\K\langle x_1,\dots,x_n\rangle$ 
by a two-sided ideal containing commutation relations
with the property that a subset of the set of standard monomials $\{x_1^{\alpha_1}\cdots x_n^{\alpha_n}\mid \alpha\in \NN^n\}$ forms a $\K$-basis. Our Gröbner basis methods rely on so-called PBW-reduction data. While their existence is guaranteed, 
we know only in special cases how to compute them. We describe a method in case of quotients of PBW-algebras and coordinate system algebras. 
Generalizing the underlying notions of standard representation and $s$-polynomial, we give a Buchberger criterion for PBW-reduction-algebras for well-orderings. As a consequence we obtain Gröbner basics for PBW-reduction-algebras.

Our methods are designed for application to mixed Hodge modules as defined by \cite{Saito2}. These objects can be considered as bifiltered $\mD_X$-modules $\mM$ with an order and a $V$-filtration satisfying certain compatibilities. A technical ingredient in the construction of Hodge theoretic functors such as the direct image functor and the vanishing cycle functor is the induced order filtration on the $V_0\mD_X$-submodule $V_\alpha \mM$ for $\alpha\in \QQ$.
Our goal to compute this filtration gives the direction for the following sections.
On a suitable coordinate system algebra $D$ the two filtrations are induced by weight vectors $\mathbf{w}$ and $\mathbf{v}$ on the algebra generators $x_1,\dots,x_n$ and $\mM(U)$ can be represented as a quotient $D^E/Q$ with an induced order filtration. 

In more generality, we consider a PBW-reduction-algebra $A$ with two filtrations $F_\bullet^\mathbf{v} A$ and $F_\bullet^\mathbf{w} A$ given by weight vectors $\mathbf{v}$ and $\mathbf{w}$. 
Motivated by our problem in Hodge theory we impose among other conditions the inclusions of algebras $F_0^\mathbf{w} A\subseteq F_0^\mathbf{v} A\subseteq A$.
Given a finitely presented $A$-module $M$ with induced $F_\bullet ^\mathbf{w}A$-filtration $F^\mathbf{w}_\bullet M$ and an $F_0^\mathbf{v}A$-submodule $V$ of $M$, our main algorithm (see \Cref{431}) yields the $F_\bullet^\mathbf{w}F_0^\mathbf{v}A$-submodule filtration $F_\bullet^\mathbf{w} V$ on $V$ if it is a good filtration. 
It computes for increasing integers $k$ the intersection $F_k^\mathbf{w} V$ of the $F_0^\mathbf{v}A$-submodule $V$ and the $F_0^\mathbf{w}A$-submodule $F_k^\mathbf{w}M$ (see \Cref{311}). A stopping criterion (\Cref{524}) serves to check whether $F_k^\mathbf{w} V$ already generates $F_\bullet^\mathbf{w} V$.
Using a common bound for the $\mathbf{v}$-degree of $V$ and $F_k^\mathbf{w} M$, the computation of their intersection
is reduced to the case where $M$ is a free $A$-module (see \Cref{397}). 
Here we reformulate it as an intersection problem over the subalgebra $F_0^\mathbf{v} A$ of $A$. We show that $F_0^\mathbf{v} A$ is again a PBW-reduction-algebra and require $F^\mathbf{w}_\bullet F_0^\mathbf{v} A$ to be a weight filtration on it. We may thus assume that $F^\mathbf{v}_0 A=A$. Finally, we want to intersect an $A$-submodule and an $F^\mathbf{w}_0 A$-submodule of a free $A$-module. This is achieved by a syzygy computation combined with a Gröbner basis computation with respect to a $\mathbf{w}$-degree ordering. In general this might not be a well-ordering. To address this issue, we describe homogenization techniques for PBW-reduction-algebras (see \Cref{368}).

The results presented in this article originate from the Ph.D.~thesis of the first named author \citep{Rottner} which was supervised by the second named author.

\section{Gröbner basis framework for PBW-reduction-algebras}\label{800}

\subsection{PBW-reduction-algebras}\label{801}

We first fix some notation: By $E$ we usually denote a finite set of module generators on which we pick a total order $<^E$.
Given a ring $R$ and a left $R$-module $M$, we write $M^E$ for the direct sum $\bigoplus_{e\in E} M(e)$, where $(e)$ stands for the free generator corresponding to $e\in E$. We identify the $R$-modules $R$ and $R^E$ in case $E$ has cardinality one. For a subset $S\subseteq M$ we define similarly $S^E:=\{s(e)\mid s\in S,e\in E\}\subseteq M^E$. A map of $R$-modules $\psi:M\to N$ induces naturally a map $\psi^E:M^E\to N^E$. 
Given finite sets $E_1,\dots, E_s$ and $1\leq i_1<\dots <i_l\leq s$, we denote by $\pi_{E_{i_1},\dots,E_{i_l}}$ the projection of $R^{E_{1}}\oplus\dots\oplus R^{E_{s}}$ to $R^{E_{i_1}}\oplus \dots\oplus R^{E_{i_l}}$.
Given a left, right and two-sided $R$-module $M$ and a subset $N\subseteq M$, we denote by $\lideal{R}{N}$, $\rideal{R}{N}$ and $\lrideal{R}{N}=\ideal{N}$ the left, right and two-sided $R$-submodule of $M$ generated by $N$, respectively. Unless said otherwise, we mean by an $R$-module always a left $R$-module.

Throughout we denote by $\K$ a field and by $x_1,\dots,x_n$ variables. We consider the polynomial ring $\K[\ul x]:= \K[ x_1,\dots, x_n]$ as $\K$-submodule of the free associative $\K$-algebra $\Tn:= \K\ideal{ x_1,\dots, x_n}$.


\begin{dfn}\label{803}
Let $E$ be a finite set.
\begin{enumerate}
\item\label{803a} We denote by
\[
 \operatorname{Mon}(\K[\ul x]^E):=\{x_1^{\alpha_1}\cdots x_n^{\alpha_n} (e)\mid \alpha\in \NN^n, e\in E\}\subseteq \K[\ul x]^E
\]
and
\[
 \operatorname{Mon}(\Tn^E):=\{x_{i_1}\cdots x_{i_k}(e)\mid k\in \NN, 1\leq i_1,\dots, i_k\leq n, e\in E\}\subseteq \Tn^E
\]
the set of \emph{monomials} of $\K[\ul x]^E $ and $\Tn^E$, respectively. The set of \emph{standard monomials} $\SMon(\Tn^E)$ is defined as $\operatorname{Mon}(\K[\ul x]^E)$ considered as a subset of $ \operatorname{Mon}(\Tn^E)$.
Abbreviating $\ul x^\alpha (e):=x_1^{\alpha_1}\cdots x_n^{\alpha_n} (e)$ for $e\in E$ and $\alpha\in \NN^n$, we often write $p\in \Pn^E$ in multi-index notation $p=\sum_{e,\alpha} p_{e,\alpha}\ul x^\alpha (e)$ where $p_{e,\alpha}\in \KK$ denotes the coefficient of $\ul x^\alpha (e)$. 
\item\label{803d}
A total order $\prec$ on $\operatorname{Mon}(\Tn^E)$ is called a \emph{monomial ordering} if 
 \begin{itemize}
 \item[(a)]\label{803dii} $m(e)\prec m'(e')$ implies $pmq(e)\prec pm'q(e')$ for $m,m',p,q\in\operatorname{Mon}(\Tn)$ and $e,e'\in E$.
 \end{itemize}
Similarly, a total order $\prec$ on $\operatorname{SMon}(\Tn^E)$ is called a \emph{monomial ordering} if 
 \begin{itemize}
 \item[(a')]\label{803eii} $\ul x^\alpha (e)\prec \ul x^{\alpha'} (e')$ implies $\ul x^{\alpha+\gamma} (e)\prec \ul x^{\alpha'+ \gamma} (e')$ for all $\alpha,\alpha',\gamma \in \NN^n$ and $e,e'\in E$.
 \end{itemize}
A monomial ordering $\prec$ on $\operatorname{(S)Mon}(\Tn^E)$ satisfying additionally
 \begin{itemize}
 \item[(b)]\label{803di} $(e)\preceq m(e)$ for all $m\in \operatorname{(S)Mon}(\Tn)$ and $e\in E$ 
 \end{itemize}
 is called a \emph{(monomial) well-ordering}. Otherwise we say that it is a \emph{(monomial) non-well-ordering}.

 \item\label{803f} Let $\prec$ be a monomial ordering on $\operatorname{(S)Mon}(\Tn^E)$. If $0\neq t=\sum_{e\in E,m\in \operatorname{(S)Mon}(\Tn)} t_{e,m} m(e)\in \Tn^E$ with $t_{e,m}\in \KK$ and $m'(e'):=\max_\prec \{m(e)\mid t_{e,m}\neq 0\}$, then we define
 \begin{itemize}
 \item $\lm_{\prec}(t):=m'(e')$, the \emph{leading monomial} of $t$,
 \item $\lt_{\prec}(t):=t_{e',m'} m'(e')$, the \emph{leading term} of $t$,
 \item $\lc_{\prec}(t):=t_{e',m'}$, the \emph{leading coefficient} of $t$,
 \item $\lcomp_{\prec}(t):=e'$, the \emph{leading component} of $t$,
 \item $\tail_{\prec}(t):=t-\lt_\prec(t)$, the \emph{tail} of $t$,
 \item $\lE_\prec(t):=\sum_{1\leq j\leq k} e_{i_j}\in \NN^n$ if $m'=x_{i_1}\cdots x_{i_k}$, the \emph{leading exponent} of $t$,
 \item $\lec_{\prec}(t):=(\lE_\prec(t),e')$, the \emph{extended leading exponent} of $t$,
 \end{itemize}
 where $e_{i_j}\in \ZZ^n$ stands for the $i_j$th unit vector. 
For a subset $G\subseteq \Tn^E$ we consider the sets
\begin{itemize}
\item $
\operatorname{l}_{\prec}(G):=\{\lec_{\prec}(g)\mid g\in G\setminus \{0\}\}\subseteq \NN^n\times E,
$
\item $
\operatorname{L}_{\prec}(G):=\{\beta+\lec_{\prec}(g)\mid g\in G\setminus \{0\}, \beta\in \NN^n\}\subseteq \NN^n\times E,
$
\end{itemize}
where we define $\beta+ (\alpha,e):=(\alpha+\beta,e)$ for $\alpha,\beta\in \NN^n$ and $e\in E$.
 
 To simplify notation later on, we extend the ordering by setting $\lm_\prec (0)\prec \lm_\prec (t)$ and $\lm_\prec (0)\preceq \lm_\prec (t')$ 
 for all $t,t'\in \lideal{\K}{\operatorname{(S)Mon}(\Tn^E)}$ with $t\neq 0$. 
We denote by $\prec$ also the ordering induced by $\prec$ on $\NN^n\times E$ via the mapping $(\alpha,e)\mapsto \ul x^\alpha(e)$ and adopt an analogous convention for $\lE_{\prec}(0)$ and $\lec_{\prec}(0)$. Similarly, we define by abuse of notation $\alpha+\lE_{\prec}(0):=\lE_{\prec} (0)$ and $\alpha+\lec_{\prec}(0):=\lec_{\prec} (0)$ for any $\alpha\in \NN^n$. 

 We sometimes omit the index $\prec$ if it is clear from the context.
 \end{enumerate}
\end{dfn} 

 
\begin{rem}\label{369}
Monomial orderings on $\SMon(\Tn)=\Mon(\K[\ul x]) $ are the monomial orderings of commutative Gröbner basis theory.
\end{rem}


\begin{rem}\label{804}Let $E$ be a finite set.
\begin{enumerate}
\item\label{804a} 
Clearly the ordering defined by 
 \begin{align*}
x_{i_1}\cdots x_{i_k} \prec' x_{j_1}\cdots x_{j_l} \text{ if and only if } & k<l\\
 \text{or } & k=l \text{ and } (i_1,\dots, i_k)<_{\text{lex}} (j_1,\dots, j_k)
\end{align*}
is a monomial well-ordering on $\operatorname{Mon}(\Tn)$. Note that $\ul x^{\operatorname{le}(x_{i_1}\cdots x_{i_k})}\preceq x_{i_1}\cdots x_{i_k}$.
\item\label{804b}
Refine a monomial ordering $\prec$ on $\SMon(\Tn^E)$ by a monomial ordering $\prec'$ on $\operatorname{Mon}(\Tn^E)$ to a monomial ordering $(\prec,\prec')$ on $\operatorname{Mon}(\Tn^E)$ by setting 
 \begin{align*}
x_{i_1}\cdots x_{i_k}(e)\, (\prec,\prec')\, x_{j_1}\cdots x_{j_l}(e') \text{ if and only if } &\lec_{\prec'}(x_{i_1}\cdots x_{i_k}(e)) \prec \lec_{\prec'}(x_{j_1}\cdots x_{j_l}(e'))\\
\text{ or } &\lec_{\prec'}(x_{i_1}\cdots x_{i_k}(e))= \lec_{\prec'}(x_{j_1}\cdots x_{j_l}(e'))\\ &\text{ and } x_{i_1}\cdots x_{i_k} \prec' x_{j_1}\cdots x_{j_l}.
 \end{align*}
We denote also by $\prec$ the special ordering $ (\prec,\prec')$ with $\prec'$ from Part~\ref{804a}.
 
 \item\label{804c} Let $\prec$ be a monomial ordering on $\operatorname{(S)Mon}(\Tn^E)$. Then any $e\in E$ defines a monomial ordering $\prec_e$ on $\operatorname{(S)Mon}(\Tn)$ by
\[
 x_{i_1}\cdots x_{i_k} \prec_e x_{j_1}\cdots x_{j_l} \text{ if and only if } x_{i_1}\cdots x_{i_k}(e) \prec x_{j_1}\cdots x_{j_l} (e).
\]
\end{enumerate}
\end{rem}
 

Eventually, we will restrict ourselves to monomial orderings on $\SMon(\Tn^E)$ and refine them to orderings on $\Mon(\Tn^E)$ as outlined in \Cref{804}.\ref{804b} above if necessary. The following \namecref{814} lists some of the orderings on $\SMon(\Tn^E)$ which we will use frequently throughout this article:


\begin{rem}\label{814} Let $E_1,\dots,E_s$ and $E$ be finite sets. 
 \begin{enumerate}
 
 \item\label{814c}An ordering $\prec$ on $\SMon(\Tn)$ and a total order $<^E$ on $E$ induce the following orderings: 
 \begin{enumerate}
 \item The \emph{term over position ordering} (TOP-ordering) on $\SMon(\Tn^E)$:
 \begin{align*}
 \ul x^\alpha (e)\, ( \prec,<^E)\, \ul x^\beta (e') \text{ if and only if } & \ul x^\alpha \prec \ul x^\beta\\
 \text{or } & \ul x^\alpha= \ul x^\beta \text{ and } e<e',
 \end{align*}
 where $\alpha,\beta\in \NN^n$ and $e,e'\in E$.
 \item The \emph{position over term ordering} (POT-ordering) on $\SMon(\Tn^E)$:
 \begin{align*}
 \ul x^\alpha (e)\, (<^E,\prec)\, \ul x^\beta (e') \text{ if and only if } & e<e'\\
 \text{or } & e=e' \text{ and } \ul x^\alpha\prec \ul x^\beta,
 \end{align*}
 where $\alpha,\beta\in \NN^n$ and $e,e'\in E$.
 \end{enumerate}
 These orderings are well-orderings if and only if $\prec$ is a well-ordering. 
 
 \item\label{814d} 
Orderings $\prec_1^{E_1},\dots,\allowbreak \prec_s ^{E_s}$ on $\SMon(\Tn^{E_1}),\dots,\SMon(\Tn^{E_s})$, respectively, define the \emph{(module) block ordering} $(\prec_1^{E_1},\dots, \prec_s^{E_s})$ on $\SMon(\Tn^{E_1\sqcup\dots\sqcup E_S})$
 by
 \begin{align*}
 \ul x^\alpha (e)\, (\prec_1^{E_1},\dots, \prec_s^{E_s}) \, \ul x^{\beta} (e') \text{ if and only if, for $e\in E_i, e'\in E_j$, } & i>j\\
 \text{ or } & i=j \text{ and } \ul x^\alpha (e) \prec^{E_i}_{i} \ul x^{\beta} (e'),
 \end{align*}
where $\alpha,\beta\in \NN^n$.
 Notice that $(\prec_1^{E_1},\dots, \prec_s^{E_s}) $ is a well-ordering if and only if all $\prec_i^{E_i}$ are well-orderings. 
 \end{enumerate}
\end{rem}


\begin{con}\label{348} Let $E_1,\dots,E_s$ and $E$ be finite sets.
 If we write from now on $\prec^E$, we implicitly assume that $\prec^E$ is some ordering on $\SMon(\Tn^E)$. Similarly, $(\prec_1^{E_1},\allowbreak\dots, \prec_s^{E_s})$ always denotes a block ordering on $\SMon(\Tn^{E_1\sqcup\dots\sqcup E_s})$. 
 
We identify $\Tn^{E_1}\oplus \dots \oplus \Tn^{E_s}\cong \Tn^{E_1\sqcup\dots\sqcup E_s}$
to define the set of (standard) monomials of the former module as well as monomial orderings on them.
\end{con}


The following notions rely on the work by \cite{Bergman}:


\begin{dfn}\label{805}
Let $E$ be a finite set, $\prec$ a monomial ordering on $\Mon(\Tn^E)$. For $s\in \Tn^E\setminus \{0\}$ and $m,m'\in \Mon(\Tn)$ we define the $\KK$-linear \emph{reduction (map) (with respect to $(S,\prec)$)}
\[
\rho_{m,s,m'}:\Tn^E\to \Tn^E,\;
x_{i_1}\cdots x_{i_l}(e)\mapsto 
\begin{cases}
m(-\frac{1}{\lc_\prec(s)}\tail_\prec(s))m'& \text{if } x_{i_1}\cdots x_{i_l}(e)=m\lm_\prec(s)m',\\
x_{i_1}\cdots x_{i_l}(e)& \text{otherwise.}
\end{cases}
\]
Let $S\subseteq \Tn^E $ be a subset. A finite composition of such maps with $s\in S$ is called a \emph{reduction sequence} (with respect to $(S,\prec)$).
\begin{enumerate}
\item\label{805d} For a reduction sequence $\rho$ with respect to $(S,\prec)$
we say that $t\in \Tn^E$ reduces to $\rho(t)$, a \emph{reduction} of $t$ (with respect to $(S,\prec)$).
\item\label{805b} We call $t\in \Tn^E$ \emph{irreducible} (with respect to $(S,\prec)$) if $\rho(t)=t$ for all reductions $\rho$. 
Such elements generate the $\K$-submodule $(\Tn^E)^{\irr}_{S,\prec}\subseteq \Tn^E$. 
\item\label{805e} 
Suppose that $\prec$ is a special ordering as in \Cref{804}.\ref{804b}.Then we call $(S,\prec)$ a \emph{commutation system} if $S=\{s_{i,j,e}\mid 1\leq i<j\leq n,e\in E\}$ where
$$s_{i,j,e}:=x_jx_i(e)-c_{ij}x_ix_j(e)-d_{ij}\text{ with } \lm_\prec(d_{ij})\prec x_ix_j(e)\prec x_jx_i(e),$$ 
$c_{ij}\in \KK^*$ and $d_{ij}\in \Pn^E$, and if every element in $\Tn^E$ reduces to an element in $\Pn^E$.
We refer to the elements of $S$ as commutation relations and to $\rho_{m, s_{i,j,e} ,m'}$ as a \emph{commutation reduction}. 

We say that an ordering $\prec'$ is \emph{compatible} with $(S,\prec)$ 
if $(S,\prec')$ is a commutation system.

\end{enumerate}
\end{dfn}


\begin{rem}\label{807}
Let $E$ be a finite set, $S\subseteq \Tn^E$ and $\prec$ a monomial ordering.
\begin{enumerate}
 \item\label{807b} If $(S,\prec)$ is a commutation system and $S\subseteq S'\subseteq \Tn^E$, 
 then $(\Tn^E)^{\irr}_{S',\prec}\subseteq \Pn^E=(\Tn^E)^{\irr}_{S,\prec}$. 
 
 \item\label{807a} If $\prec$ is a well-ordering, every element of $\Tn^E$ reduces to an irreducible element with respect to $(S,\prec)$. If in addition $S$ is a two-sided submodule, then this reduction is unique. This identifies $(\Tn^E)^{\irr}_{S,\prec}=\Tn^E/S$ as $\K$-vector spaces and we write
 \[
 \rho_{S,\prec}:\Tn^E\to (\Tn^E)^{\irr}_{S,\prec}
 \]
for the map sending elements to their unique irreducible reduction. 

 \item\label{807c} If $(S,\prec)$ is a commutation system and $\prec$ a well-ordering, we can determine for an element $p\in \Tn$ a finite set $U\subseteq \Tn \times S\times \Tn$ such that
\[
 p=\rho_{S,\prec'}(p)+\sum_{(t,s,t')\in U} tst' \text{ with } \lm_{\prec'}(\rho_{S,\prec'}(p))=\lm_{\prec'}(p) \text{ and }\lm_{\prec'}(tst') \preceq' \lm_{\prec'}(p)
\]
for all orderings $\prec'$ compatible with $(S,\prec)$. In particular $\rho_{S,\prec}=\rho_{S,\prec'}$.
\end{enumerate}
\end{rem}


We are particularly interested in the following class of $\KK$-algebras:


\begin{dfn}\label{808} 
A \emph{PBW-reduction-datum} $(\Tn,S,I,\prec)$ consists of a commutation system $(S,\prec)$, where $\prec$ is a well-ordering 
and $I\subseteq \K[\ul x]$ a finite set such that 
 \begin{equation}\label{808b}
\operatorname{L}_\prec(I)=\operatorname{l}_\prec(\ideal{I\cup S}\cap \K[\ul x]).
\end{equation}
It defines a $\K$-algebra $\Tn/\ideal{I\cup S}$, 
which we call a \emph{PBW-reduction-algebra} and write by abuse of notation
\[
 \Tn/\ideal{I\cup S}=(\Tn,S,I,\prec).
\]
We say that a monomial (well-)ordering $\prec'$ is a (monomial) (well-)ordering on $ \Tn/\ideal{I\cup S}$ if $\prec'$ is compatible with $(S,\prec)$. 
\end{dfn}


\begin{rem}\label{837}
For a commutation system $(S,\prec)$ and $I\subseteq \Pn$ we have that 
\[
 \operatorname{l}_\prec(\ideal{I\cup S}\cap \K[\ul x])=\operatorname{L}_\prec(\ideal{I\cup S}\cap \K[\ul x]).
\]
Indeed, $\rho_{S,\prec}(\ul x^\alpha r)\in \ideal{I\cup S}\cap \KK[\ul x]$ with $\lm_\prec(\rho_{S,\prec}(\ul x^\alpha r))=\alpha+\lm_\prec(r)$ for $r\in \ideal{I\cup S}\cap \K[\ul x]$ and $\alpha\in \NN^n$.
In particular the inclusion $\subseteq$ in \ref{808}\eqref{808b} is always satisfied. This makes Condition~\ref{808}\eqref{808b} equivalent to
\[
 \operatorname{L}_\prec(I)=\operatorname{L}_\prec(\ideal{I\cup S}\cap \K[\ul x]).
\]
\end{rem}


\begin{lem}\label{n1}
For a PBW-reduction-datum $(\Tn,S,I,\prec)$ we can identify the $\K$-vector spaces 
\[
 \Tn/\ideal{I\cup S}=\Tn^{\irr}_{\ideal{I\cup S},\prec}=\bigoplus_{\alpha\notin \operatorname{L}_\prec(I)} \KK\ul x^\alpha
\]
\end{lem}

\begin{proof}
 The first equality is due to \Cref{807}.\ref{807a}. Since the set of irreducible standard monomials with respect to $(\ideal{I\cup S},\prec)$ forms a $\K$-basis of $\Tn^{\irr}_{\ideal{I\cup S},\prec}$ by \Cref{807}.\ref{807b}, it suffices to show that this set agrees with $ \{\ul x^\alpha\mid \alpha\notin \operatorname{L}_\prec(I)\}$. Clearly, the former set is contained in the latter. On the other hand, $\ul x^\alpha$ with $\alpha\notin \operatorname{L}_\prec(I)=\operatorname{l}_\prec(\ideal{S\cup I}\cap \Pn)$ is indeed irreducible.
\end{proof}


\begin{prp}\label{1000}
PBW-algebras are precisely the PBW-reduction-algebras with PBW-reduction datum of type $(\Tn,S,\{0\},\prec)$. In particular, polynomial rings and Weyl algebras are PBW-reduction-algebras.
\end{prp}

\begin{proof}
Given a PBW-algebra with commutation system $(S,\prec)$ with respect to the well-ordering $\prec$. Then $\ideal{S}\cap \Pn=\{0\}$ by definition and hence Equation \eqref{808b} is satisfied with $I=\{0\}$. The converse is due to \Cref{n1}. 
\end{proof}


\begin{lem}\label{888} 
 Consider the commutation system $(S,\prec)$, where $\prec$ is a well-ordering, the finite set $I\subseteq \Pn$ and the two-sided ideal $R\subseteq \Tn$ containing $S$ and $I$ and satisfying $\operatorname{L}_\prec(I)=\operatorname{l}_\prec(R\cap \Pn)$.
 For $p\in \Tn$ one can compute $a\in \Tn^I$ and a finite set $U\subseteq \Tn\times S\times \Tn$ such that 
 \[
 p=\rho_{R,\prec}(p)+\sum_{g\in I} a_g g+\sum_{(t,s,t')\in U} tst', 
 \]
$ \lE(a_g)+\lE(g)\preceq \lE(p)$ with equality for some $g\in I$ and $\lE(t)+\lE(s)+\lE(t')\preceq\lE(p)$.
In particular, $
 R=\ideal{S\cup I}$ and $(\Tn,S,I,\prec)$ is a PBW-reduction datum.
\end{lem}


\begin{proof}
 By \Cref{807}.\ref{807c} we can write $p\in \Tn$ as
 \begin{equation}\label{n2}
 p=\rho_{S,\prec} (p)+\sum_{(t,s,t')\in U_p} tst'
 \end{equation}
for some $U_p\subseteq \Tn\times S\times \Tn$ with $\lE(t)+\lE(s)+\lE(t')\preceq\lE(p)$. If $\lE(\rho_{S,\prec} (p))\notin \operatorname{l}_\prec(R\cap \Pn)$, then $\lm(\rho_{S,\prec} (p))\in \Tn^{\irr}_{R,\prec}$ and we continue with $\tail (\rho_{S,\prec} (p))$.
Otherwise pick $g\in I$ such that $\lE(p)=\lE(\rho_{S,\prec} (p))=\lE(g)+\alpha$. \Cref{807}.\ref{807c} also yields an Equation \eqref{n2} with $p$ replaced by $\ul x^\alpha g$.
Hence we obtain 
 \[
 p=c \ul x^\alpha g+ (\tail(\rho_{S,\prec} (p))-\tail(c\rho_{S,\prec} (\ul x^\alpha g)))+\sum_{(t,s,t')\in U_p} tst'-c\sum_{(t,s,t')\in U_{\ul x^\alpha g}} tst'
 \]
for $c=\lc(\rho_{S,\prec} (p))/\lc(\rho_{S,\prec}(\ul x^\alpha g))$. Induction on the well-ordering $\prec$ finishes the proof.
\end{proof}


\begin{prp}\label{841} 
For a commutation system $(S,\prec)$ with well-ordering $\prec$ and a two-sided ideal $S\subseteq R\subseteq \Tn$ exists a finite set $I\subseteq \Pn$ such that $\Tn/R=(\Tn,S,I,\prec)$ is a PBW-reduction-algebra. 
\end{prp}


\begin{proof}
Consider the set 
 \[
 L:=\operatorname{L}_\prec (R\cap \KK[\ul x])\subseteq \NN^n.
 \]
By Dickson's Lemma there is a finite subset $L'\subseteq L$ such that for every $\alpha\in L$ exists an $\alpha'\in L'$ with $\alpha\in \NN^n+\alpha'$. Choose for every $\alpha'\in L'$ an $r_{\alpha'}\in {R}\cap \Pn$ having leading exponent $\alpha'$. Setting
\[
 I:=\{r_{\alpha'}\mid \alpha'\in L'\},
\]
 $\Tn/R=(\Tn,S,I,\prec)$ is a PBW-reduction-algebra by \Cref{888}.
\end{proof}


In general it is unclear how to obtain the set $I$ of the PBW-reduction datum.
In the following special case this is possible.


\begin{lem}\label{812} 
Consider the $\K$-algebra $\KK\langle \ul x,\ul y\rangle:=\KK\langle x_1, \allowbreak \dots,\allowbreak x_n,\allowbreak y_1,\dots,y_m\rangle$, an ideal $I\subseteq \KK[\ul x]$ and a commutation system $(S,\prec)$ such that
\[
S=\{[x_j,x_i]\mid 1\leq i<j \leq n\}\cup\{[y_l,y_k]-d_{kl}\mid 1\leq k<l\leq m\}
\cup \{[y_k,x_i]-f_{ik}\mid 1\leq i\leq n, 1\leq k\leq m\}, 
\]
where $d_{kl},f_{ik}\in \Pn$.
Suppose the surjective $\KK$-linear homomorphism
\[
 \psi: \bigoplus_{\beta\in \NN^m} (\KK[\ul x]/J)\ul y^\beta \to \K\langle \ul x,\ul y\rangle/\ideal{J\cup S},\; \ol{\ul x^\alpha}\ul y^\beta\mapsto \ol{\ul x^\alpha\ul y^\beta} 
\]
is injective and that $I'$ is a Gröbner basis of $I\subseteq \KK[\ul x]$ with respect to the ordering induced by $\prec$. Then $(\KK\langle \ul x,\ul y\rangle,S,I',\prec)$ is a PBW-reduction datum.
\end{lem}

\begin{proof}\ 
Let $ 0\neq p=\sum_{\beta\in \NN^m}p_{\beta} \ul y^\beta\in \ideal{I\cup S}\cap \KK[ \ul x,\ul y]$ with $p_{\beta}\in \Pn$. Then $\ol 0=\ol p \in \K\langle \ul x,\ul y\rangle/\ideal{I\cup S}$, hence $\sum_{\beta\in \NN^m}\ol{p_{\beta}} \ul y^\beta=\psi^{-1}(\ol p)=0$ by hypothesis and $p_\beta\in I$ for all $\beta\in \NN^m$.
Since $I'$ is a Gröbner basis of $I\subseteq \KK[\ul x]$ with respect to the ordering induced by $\prec$, it follows that
\[
(\alpha',\beta'):=\lE_{\prec}(p)=(\lE_{\prec}(p_{\beta'}),\beta')\in L_\prec(I').\qedhere
\]
\end{proof}

\begin{rem}\label{n4}
For $S$ as in \Cref{812}, any special well-ordering $\prec$ as in \Cref{805}\eqref{805e} that satisfies
 \[
 \ul x^\alpha \ul y^\beta \prec\ul x^{\alpha'} \ul y^{\beta'}\text{ if } |\beta|<|\beta'|
\]
(with $\alpha,\alpha'\in \NN^n$ and $\beta,\beta'\in \NN^m$)
 makes $(S,\prec)$ a commutation system.
\end{rem}


\begin{dfn}\label{923} In the situation of \Cref{812}
 we call $A=(\KK\langle \ul x,\ul y\rangle,S, I',\prec)$ an \emph{elementary PBW-reduction datum / algebra}.
\end{dfn}


Generalizing \Cref{1000}, the following example describes differential operators on smooth, complex affine varieties as PBW-reduction-algebras. 


\begin{exa}\label{577} 
Let $X$ be a smooth irreducible complex affine variety of dimension $m$ defined by the prime ideal $I\subseteq \CC[\ul x]$. Its tangent sheaf $\Theta_X$ is a locally free $\mO_X$-module. Note that every element of $\Theta_X(X)= \operatorname{Der}_{\CC}(\CC[\ul x]/I)$ is of the form $\ol{\theta}$ for some $\theta\in \Theta_{\CC^n}(\CC^n)=\operatorname{Der}_{\CC}(\CC[\ul x])$ with $\theta(I)\subseteq I$. After shrinking $X$ if necessary there is
an $\mO_X$-basis $\ol{\theta_1},\dots,\ol{\theta_m}\in \Theta_X(X)=\operatorname{Der}_{\CC}(\CC[\ul x]/I)$ of $\Theta_X$ and there are
regular functions $\ol{f_1},\dots, \ol{f_m}\in \CC[\ul x]/I$ 
satisfying
$
[\ol{\theta_i},\ol{\theta_j}]=0$ and $[\ol{\theta_i},\ol{f_j}]=\delta_{ij}$ for $1\leq i,j\leq m$.
\begin{enumerate}

 \item \label{577a} The global sections of the sheaf of differential operators $\mD_X$ form an elementary PBW-reduction-algebra: 
We have a $\CC$-linear isomorphism (see \cite[Lemma 1.2.7]{Rottner})
\begin{align*}
\phi:\bigoplus_{\beta\in \NN^m} (\CC[\ul x]/I) \ul y^\beta \to &\mD_X(X)=\CC\langle \ol{x_1},\dots, \ol{x_n},\ol{\theta_1},\dots,\ol{\theta_m}\rangle \subseteq \End_{\CC}(\CC[\ul x]/I),\\
 \ol{\ul x^\alpha}\ul y^\beta \mapsto &\ol{x_1}^{\alpha_1}\cdots \ol{x_n}^{\alpha_n}\ol{\theta_1}^{\beta_1}\cdots \ol{\theta_m}^{\beta_m} 
\end{align*}
and the generators of the $\CC$-algebra $\mD_X(X)$ satisfy $[\ol{x_j},\ol{x_i}]=0$, $[\ol{\theta_p},\ol{\theta_k}]=0$ and $[\ol{\theta_k},\ol{x_i}]={\ol{{\theta}_k(x_i)}}$ for $1\leq i\leq j\leq n$ and $1\leq k\leq p\leq m$.
Consequently, $\psi$ factors through the algebra
\[ 
T_X:= \CC\langle \ul x, \ul y\rangle / \ideal{S\cup I},
\]
where 
\begin{align*}
 S:=&\{[x_j,x_i]\mid 1\leq i<j\leq n\}\cup \{[y_p,y_k]\mid 1\leq k<p\leq m\}\\
 \cup &\{ [y_k,x_i]-{\theta}_k(x_i) \mid 1\leq i\leq n, 1\leq k\leq m\}.
\end{align*}
This leads to isomorphisms
\begin{equation}\label{n3}
 \xymatrixrowsep{1pc}\xymatrix{
 \phi:\bigoplus_{\beta\in \NN^m} (\CC[\ul x]/I)\ul y^\beta \ar[r]^-{\psi}_-{\cong} &T_X\ar[r]_-{\cong} &\mD_X(X)\\
 \ol{\ul x^\alpha}\ul y^\beta \ar@{|->}[r]& \ol{\ul x^\alpha\ul y^\beta} \ar@{|->}[r]& \ol{x_1}^{\alpha_1}\cdots \ol{x_n}^{\alpha_n}\theta_1^{\beta_1}\cdots \theta_m^{\beta_m}.
 }
 \end{equation}
 identifying $\mD_X(X)$ with the elementary PBW-reduction-algebra $T_X$.

 \item \label{577c} 
By identifying $X$ with the closed subvariety $V(I,t-f_m)\subseteq \CC^n\times \CC_t$,
 we may assume $f_m$ agrees with $x_n$, and that $\theta_i(x_n)=\delta_{i,m}$. 
Consider the obvious algebra homomorphism $\CC\langle \ul x, y_1,\dots,y_{m-1},z\rangle \to \mD_X(X)$ sending $z$ to $\ol{x_n}\ol{\theta_m}$ with image $V$. It factors through the PBW-reduction-algebra $T_X^V=\CC\langle \ul x, y_1,\dots,y_{m-1},z\rangle /\ideal{S_V\cup I}$ where
\begin{align*}
S_V:=\{&[x_j,x_i], [y_l,y_k],[z,y_k],[y_k,x_i]-{\theta}^l_k(x_i), [z,x_i]-x_n{\theta}^l_m(x_i)\mid \\
& 1\leq i\leq j\leq n, 1\leq k\leq l\leq m-1 \}\setminus \{0\}.
\end{align*}
This extends \Cref{n3} to a commutative diagram of $\CC$-linear maps
\[
 \xymatrixrowsep{1pc}\xymatrix{
 \bigoplus_{\beta\in \NN^m} (\CC[\ul x]/I)\ul y^\beta \ar[r]^-{\psi}_-{\cong} &T_X\ar[r]_-{\cong} &\mD_X(X)\\
 \bigoplus_{\beta\in \NN^{m-1},\gamma\in \NN} (\CC[\ul x]/I)y_1^{\beta_1}\cdots y_{m-1}^{\beta_{m-1}} z^\gamma\ar[u] \ar@{->>}[r]\ar[u]& T_X^V\ar@{->>}[r]\ar[u]& V\ar[u].
 }
 \]
 where the right hand square consists of $\CC$-algebra homomorphisms. One can show that the left vertical map is injective. It follows that the bottom maps are isomorphisms and the vertical maps are injections. So we may identify $V$ with the elementary PBW-reduction-algebra $T_X^V$.

\item\label{577d} Let $\phi_{x_n}:\CC\langle \ul x,y_1,\dots,y_{m-1},z \rangle\to \CC\langle x_1,\dots,x_{n-1},y_1,\dots,y_{m-1},z \rangle $ be the $\CC$-algebra homomorphism that maps $x_n$ to $0$ and acts on all other variables as identity. In the situation of Part~\ref{577c}, $V/x_nV$
can be realized as the elementary PBW-reduction-algebra 
\[
T_X^{V/x_nV}:=(\CC\langle x_1,\dots, x_{n-1}, y_1,\dots,y_{m-1},z\rangle ,S_{V/x_nV},I_{V/x_nV},\prec^{V/x_nV})
\]
defined as follows: 
 For
\begin{align*}
S_{V/x_nV}:=\{&[x_j,x_i], [y_p,y_k],[z,y_k],[z,x_i], [y_k,x_i]-\phi_{x_n}({\theta}_k^l(x_i)) \mid \\
& 1\leq i\leq j\leq n-1, 1\leq k\leq p\leq m-1 \}\setminus \{0\},
\end{align*}
pick a well-ordering $\prec^{V/x_nV}$ as in \Cref{n4}. Now let 
 $I_{V/x_nV}\subseteq \CC[x_1,\dots,x_{m-1}]$ be a Gröbner basis of $\phi_{x_n}(I)$ with respect to the ordering induced by $\prec^{V/x_nV}$. Note that the canonical projection $V\to V/x_nV$ induces the same
 map $ T_X^{V}\to T_X^{V/x_nV}$ as $\phi_{x_n}$.

\item\label{577e}
With the assumption of Part~\ref{577c} assume that the subvariety $X_0:=V(x_n)\cap X\subseteq X$ is smooth. Then $(\ol{f_i}, \overline{\theta_i})_{1\leq i\leq m-1}$ is a global coordinate system on $X_0$.
By Part~\ref{577a} $\mD_{X_0}(X_0)$ identifies with the PBW-reduction-algebra $T_{X_0}$, whose commutation system is obtained by deleting all equations involving $y_m$ from $S$. The natural isomorphism $ V/x_nV\cong \mD_{X_0}(X_0)[\ol{x_n}\ol{\theta_m}]$ then identifies with the isomorphism $T_X^{V/x_nV}\cong T_{X_0}[z]$ that is induced by the identity of $\CC\langle x_1,\dots,x_{n-1},y_1,\dots,y_{m-1},z\rangle$.

\end{enumerate}
\end{exa}


\begin{rem}\label{935} 
There is an algorithmic approach to the sheaf of differential operators on smooth affine varieties by \cite{OakuAffine}.
Consider the setup of \Cref{577}.
Oaku suggests two methods: The first one is based on the statement that the $\CC$-subalgebra of the Weyl-algebra generated by $x_1,\dots,x_n$ and $\theta_1,\dots,\theta_m$ equals $\bigoplus_{\alpha\in \NN^n,\beta\in \NN^m} \CC\ul x^\alpha\theta_1^{\beta_1}\cdots \theta_m^{\beta_m}$. This does not hold in general: Indeed, there is always a local coordinate system such that $f_i=x_i$ and $\theta_i=\partial_i+\sum_{m+1\leq k\leq n} a_k^i(\ul x) \partial_k$ with $a_k^i(\ul x)\in \CC[\ul x]$. In this case $[\theta_j,\theta_i]\in \sum_{m+1\leq k\leq n} \CC[\ul x]\partial_k$ can only be contained in the above direct sum if it is $0$. However, $\theta_1,\dots,\theta_m$ do not commute in general.

Oaku's second method uses the Leibnitz rule to define a non-associative ``multiplication''. 
The resulting algorithm is essentially equivalent to \Cref{829}. However, Oaku's proof of correctness relies again on the above false statement.
\end{rem}


\begin{prp}\label{813}
PBW-reduction-algebras are left and right Noetherian rings.
\end{prp}

\begin{proof}
Let $A=(\Tn,S,I,\prec)$ be a PBW-reduction-algebra with $S=\{x_jx_i-c_{ij}x_ix_j-d_{ij}\mid 1\leq i<j\leq n\}$. 
This gives rise to a multi-filtration $F^\prec_\bullet$ on $A$ indexed by $\NN^n$ (see \cite{MultiFilt}) given
by
\[
 F^\prec_\alpha A:=\sum_{ \ul x^\beta\preceq \ul x^\alpha} \KK\ol{\ul x^\beta},\quad F^\prec_{\prec\alpha} A:=\bigcup_{\beta\in \NN^n:\ul x^\beta\prec\ul x^\alpha} F^\prec_{\beta} A=\sum_{ \ul x^\beta\prec \ul x^\alpha} \KK\ol{\ul x^\beta}
\] 
for $\alpha\in \NN^n$. Note that this filtration is exhaustive by \ref{n1}. By \cite[Lemma 1.2]{MultiFilt} it suffices to proof the claim for the
 associated multi-graded algebra
\[
 \gr^{F^\prec}\!A:=\bigoplus_{\alpha\in \NN^n} F^\prec_\alpha A/F^\prec_{\prec\alpha} A.
\]
It identifies with a factor algebra of a PBW-algebra by the isomorphism
 \begin{align*}
\varphi: \gr^{F^\prec}\!A\to & \left(\Tn/\ideal{ \{x_jx_i-c_{ij} x_ix_j\mid 1\leq i<j\leq n\}}\right)/\ideal{\{\ol{\lm_\prec(p)}\mid p\in I\}},\\ 
 \gr^{F^\prec}_{e_i} \!A \ni \overline{x_i}\mapsto & \ol{x_i}+\ideal{\{\ol{\lm_\prec(p)}\mid p\in I\}}.
 \end{align*}
The latter is left and right Noetherian (see e.g. \cite[Theorem 4.1]{Noeth}). 
\end{proof}

\subsection{Gröbner bases for PBW-reduction-algebras}\label{288}

\begin{dfn}\label{292}Let $A$ be a PBW-reduction-algebra and $E$ a finite set.
\begin{enumerate}

\item\label{292b}
If $(\Tn,S_e,I_e,\prec_e)$ is a PBW-reduction datum of $A$ for $e\in E$, we call $(\Tn,S_e,I_e,\prec_e)_{e\in E}$ a PBW-reduction datum for $A^E$ and write 
$A^E=(\Tn,S_e,I_e,\prec_e)_{e\in E}$. We say that the monomial (well-)ordering $\prec^E$ on $\SMon(\Tn^E)$ is a \emph{(well-)ordering} on $A^E$ if $\prec^E_e$ is an ordering on $e$th summand $(\Tn,S_e,I_e,\prec_e)$ of $A^E$. 
If $\prec^E_e=\prec_e$ for all $e\in E$, we call $(\Tn,S_e,I_e,\prec_e)_{e\in E}$ a PBW-reduction datum for $(A^E,\prec^E)$ and write $(A^E,\prec^E)=(\Tn,S_e,I_e,\prec_e)_{e\in E}$.

\item\label{292c} Abusing notation, we set for $(A^E,\prec^E)=(\Tn,S_e,I_e,\prec_e)_{e\in E}$ 
\[
\rho_{A^E,\prec^E}:=\bigoplus_{e\in E} \rho_{\ideal{I_e\cup S_e},\prec^E_e}: \Tn^E\to (\Tn^E)^{\irr}_{A^E,\prec^E}:=\bigoplus_{e\in E} \Tn^{\irr}_{\ideal{I_e\cup S_e},\prec^E_e}(e).
\]
We also define the map 
\[
 \tau_{A^E,\prec^E}: A^E\to (\Tn^E)^{\irr}_{A^E,\prec^E}\subseteq \Tn^E
\]
as the inverse of the composed map $(\Tn^E)^{\irr}_{A^E,\prec^E}\hookrightarrow \Tn^E\twoheadrightarrow A^E$ using \Cref{807}.\ref{807c}. We sometimes also use the notation $\rho_{\prec^E}$ and $\tau_{\prec^E}$ for the above maps if that does not cause any ambiguity. 

For $0\neq a\in A$, we define the data introduced in \Cref{803}\eqref{803f} by the corresponding data of $\tau_{(A^E,\prec^E)}(a)$ and adapt the convention for the leading exponents and monomials of $0$ accordingly. 

\end{enumerate}
\end{dfn}


\begin{rem}\label{833}
Let $A=(\Tn,S,I,\prec)$ be a PBW-reduction-algebra and $E$ and $E_1,\dots, E_s$ finite sets. Then we have:
\begin{enumerate}

\item\label{833a} Given a total order $<^E$ on $E$ PBW-reduction data for $(A^E,(\prec,<^E))$ and $(A^E,(<^E,\prec))$ are given by $(\Tn,S,I,\prec)_{e\in E}$. 

\item\label{833b} We identify $A^{E_1\sqcup\dots \sqcup E_s}=A^{E_1}\oplus \dots\oplus A^{E_s}$ extending the notions of \Cref{292} to the latter.
PBW-reduction data on $(A^{E_i},\prec^{E_i})$ define PBW-reduction data on $(A^{E_1}\oplus \dots\oplus A^{E_s}, \prec_{1,\dots,s}^{E_1,\dots,E_s})$.

\end{enumerate}
\end{rem}


\begin{dfn}\label{293} Let $A$ be a PBW-reduction-algebra, $E$ a finite set, $(A^E,\prec^E)=(\Tn,S_e,I_e,\allowbreak \prec_e)_{e\in E}$ and $M\subseteq A^E$ an $A$-submodule.
\begin{enumerate}

\item\label{293a} We call the finite set $G\subseteq M$ a \emph{Gröbner basis} of $M$ (with respect to $\prec^E$) if every $m\in M$ has a so-called \emph{standard representation}, i.e., there exists $a\in A^G$ such that 
 \[
 m=\sum_{g\in G} a_gg \text{ and } \lE_{\prec^E_{\lcomp(g)}}(a_g)+\lec_{\prec^E}(g)\preceq^E \lec_{\prec^E}(m) \text{ for all } g\in G.
 \]

\item\label{293b} If $G$ is a Gröbner basis of $M$, we say that $G$ is reduced if $0\notin G$, $\lc_{\prec^E}(g)=1$ for all $g\in G$, and if we have for all $g\in G$, $e\in E$ and $\alpha\in \NN^n$ 
\[
(\tau_{A^E,\prec^E}(g))_{e,\alpha}\neq 0 \text{ implies } (\alpha,e)\neq\lec_{\prec^E}(g')+\gamma \text{ for all } g\neq g'\in G, \gamma\in \NN^n.
\]

\end{enumerate}
\end{dfn}


\begin{rem}\label{846} 
Let $A$ be a PBW-reduction-algebra, $E$ a finite set, $\prec^E$ an ordering on $A^E=(\Tn,S_e,I_e,\prec_e)_{e\in E}$ and $M\subseteq A^E$ an $A$-submodule. To circumvent the problem that we do in general not have a well-defined notion of leading exponents of elements of $A^E$ with respect to $\prec^E$, we define Gröbner bases in this situation as follows: We say that a finite set $G\subseteq M$ is a \emph{Gröbner basis} of $M$ with respect to $\prec^E$ if there exists $h\in (\Pn^E)^G$ with $\ol{h_g}=g$ for $g\in G$
such that
for every $t\in \Pn^E$ with $\ol{t}\in M$ exists $a\in \Pn^G$ with 
 \[
 \ol{t}=\sum_{g\in G} \ol{a_g}g \text{ and } \lE_{\prec^E_{\lcomp(h_g)}}(a_g)+\lec_{\prec^E}(h_g)\preceq^E \lec_{\prec_E}(t) \text{ for all } g\in G.
 \]
 We say in this case that $\{h_g\mid g\in G\}$ induces a Gröbner basis of $M$ (with respect to $\prec^E$).
 
Notice that since there exists by definition of PBW-reduction-algebras a well-ordering on $A$, every $m\in M$ has a representative in $\Pn^E$. Moreover, this definition is compatible with \Cref{293}\eqref{293a}.
\end{rem}


\begin{dfn}\label{817}
Let $A$ be a PBW-reduction-algebra, $E$ a finite set, $(A^E,\prec^E)=(\Tn,S_e,I_e, \allowbreak\prec_e)_{e\in E}$ and let $a,a'\in A^E$ be nonzero.
\begin{enumerate}

\item\label{817a} Given a finite set $G\subseteq A^E$, we call $r\in A^E$ a \emph{(left) normal form} of $a$ with respect to $G$ if
 \begin{enumerate}
 \item there exists some $h\in A^G$ with 
\[
 a=\sum_{g\in G} h_g g+r
\]
such that $\lE_{\prec^E_{\lcomp(g)}}(h_g)+\lec_{\prec^E}(g)\preceq^E \lec_{\prec^E}(a)$ for all $g\in G$ and

\item $\lec_{\prec^E}(r)\notin L_{\prec^E}(G)$ if $r\neq 0$.
\end{enumerate}
 We call $r
 $ \emph{reduced} if $(\alpha,e)\notin L_{\prec^E}(G)$ given that $(\tau_{(A^E,\prec^E)}(r))_{e,\alpha}\neq 0$. We define the normal form of $0\in A^E$ with respect to $G$ to be $0$.

\item\label{817b} The \emph{$s$-polynomial} of $a$ and $a'$ with $e:=\lcomp(a)=\lcomp(a')$ is defined by
\[
 \spoly(a,a'):=\begin{cases}
 \frac{1}{\lc(\ul x^{c_{a,a'}}a)}\ul x^{c_{a,a'}} a-\frac{1}{\lc(\ul x^{c_{a',a}} a')}\ul x^{c_{a',a}} a' & \text{if } \ul x^{b_{a,a'}}(e)\in (\Tn^E)^{\irr}_{A^E,\prec^E},\\
 0&\text {otherwise,}
 \end{cases}
\]
where $b_{a,a'}
,c_{a,a'}\in \NN^n$ are given by $(b_{a,a'})_i:=\max\{\lE(a)_i,\lE(a')_i\}$ and $(c_{a,a'})_i:=(b_{a,a'})_i-\lE(a)_i$ for $1\leq i\leq n$. 
If $\lcomp(a)\neq \lcomp(a')$, we set $\spoly(a,a'):=0$. 
\item\label{817c} The \emph{$s$-polynomial} of $a$ and $p\in I_e$ is defined by
\[
 \spoly(a,p):=\begin{cases}
 \ul x^{c_{a,p}} a & \text{ if } e=\lcomp(a),\\
 0& \text{ otherwise, }
 \end{cases}
\]
where $c_{a,p}\in \NN^n$ is given by $(c_{a,p})_i:=\max\{\lE(a)_i,\lE(p)_i\} -\lE(a)_i$ for $1\leq i\leq n$.

\end{enumerate}
\end{dfn}


\begin{rem}\label{860} We keep the notation of \Cref{817}.
Assume that $a,a'\in A^E$ satisfy $e:=\lcomp(a)=\lcomp(a')$.
Then 
\[
\lec(\spoly(a,a'))\prec^E (b_{a,a'},e)=\lec(\ul x^{c_{a,a'}} a)=\lec(\ul x^{c_{a',a}} a').
\]
Similarly, we have for $p\in I_e$ 
\[
\lec(\spoly(a,p))\prec^E {c_{a,p}}+\lec( a).
\]
\end{rem}


The following algorithm clearly computes a normal form and terminates, hence showing the existence of normal forms:


\begin{algorithmbis}[Given a PBW-reduction-algebra $A$, a finite set $G\subseteq A^E$, a well-ordering $\prec^E$ on $A^E$ and $a\in A^E$, this algorithm computes a normal form of $a$ with respect to $G$ and $\prec^E$.]
\label[algorithm]{828}
\begin{algorithmic}[1]
\REQUIRE {A PBW-reduction-algebra $A$, a finite set $E$, $(A^E,\prec^E)=(\Tn,S_e,I_e,\prec_e)_{e\in E}$, $G\subseteq A^E$ finite and $a\in A^E$.}
\ENSURE A normal form $b\in A^E$ of $a$ with respect to $G$.
\WHILE {$a\neq 0$ and $\tilde{G}:=\{g\in G\mid \lec_{\prec^E}(a)\in L_{\prec^E}(\{g\})\}\neq \emptyset$}
\STATE Choose $g\in \tilde{G}
$.
\STATE Set $a:=\lc_{\prec^E}(a)\cdot \spoly(a,g)$.
\ENDWHILE
\RETURN $a$.
\end{algorithmic}
\end{algorithmbis}


The above \namecref{828} can be modified to return a reduced normal form using the same method as in the commutative setting (see e.g. \cite[Algorithm 1.6.11]{SingularBook}). 


\begin{rem}\label{818}
Let $A$ be a PBW-reduction-algebra, $E$ a finite set, $\prec^E$ a well-ordering on $A^E$ and $M\subseteq A^E$ an $A$-submodule. If $G$ is a Gröbner basis of $M$, then clearly $m\in A^E$ is an element of $M$ if and only if some / every normal form of $m$ with respect to $G$ is $0$.
\end{rem}


Our algorithm for computing Gröbner bases is based on a variant of the Buchberger criterion for polynomial rings that takes into account the additional relations:


\begin{prp}\label{816}[Buchberger criterion for PBW-re\-duc\-tion-algebras]
 Let $A$ be a PBW-re\-duc\-tion-algebra, E a finite set, $(A^E,\prec^E)=(\Tn,S_e,I_e,\prec_e)_{e\in E}$ and $G\subseteq A^E$ a finite set. Then $G$ is a (left) Gröbner basis (with respect to $\prec^E$) of the $A$-module $\lideal{A}{G}$ if and only if 
 \begin{enumerate}
 \item for all $g,g'\in G$ some / any normal form of $\spoly(g,g')$ with respect to $G$ is $0$ and
 \item for all $g\in G$ and $p\in I_{\lcomp(g)}$ some / any normal form of $\spoly(a,g)$ with respect to $G$ is $0$.
 \end{enumerate}
\end{prp}


For the proof we adapt a standard proof of the commutative Buchberger criterion 
to our setting. It relies on the following \namecref{815}, whose proof from the commutative setting carries over word by word:


\begin{lem}\label{815}
Let $A$ be a PBW-reduction-algebra, E a finite set, $(A^E,\prec^E)=(\Tn,S_e,I_e,\prec_e)_{e\in E}$. Let $G\subseteq A^E\setminus \{0\}$ be a finite set whose elements have the same leading monomial.
Let $m=\sum_{g\in G} a_g g$ with $a\in \KK^G$ be such that $\lm(m)\prec^E \lm(g)$ for $g\in G$.
Then there exists $d \in \KK^{G\times G}$ such that $m=\sum_{(g,g')\in G\times G}
d_{(g,g')}\spoly(g,g')$.
\end{lem}


The following remark lists some facts that are used throughout our proof of \Cref{816}:


\begin{rem}\label{856} 
Let $A$ be a PBW-reduction-algebra, $E$ a finite set, $\prec^E_o$ an ordering on $(A^E,\prec^E)=(\Tn,S_e,I_e,\prec_e)_{e\in E}$. Define for $l\in \NN$, $1\leq i_1,\dots,i_l\leq n$ the vector $\alpha:=\sum_{1\leq j\leq l} e_{i_j}\in \NN^n$ and let $e\in E$.
\begin{enumerate}

\item\label{856a} We have $\ul x^{\alpha}(e)\preceq^E_o x_{i_1}\cdots x_{i_l}(e)$. 

\item\label{856b} Independently of the choice of $\prec_o^E$, we can find $r_{i_1,\dots,i_l}\in \Pn$ and $f_{i_1,\dots,i_l}\in \K^*$ with $\lec_{\prec^E_o}(r_{i_1,\dots,i_l}(e))\prec^E_o (\alpha,e)$ such that
 \[
 x_{i_1}\cdots x_{i_l}-f_{i_1,\dots,i_l}\ul x^{\alpha}-r_{i_1,\dots,i_l} \in \lrideal{\Tn}{S_e}
 \]
 and hence
 \[
 \ol{x_{i_1}\cdots x_{i_l}(e)}=\ol{f_{i_1,\dots,i_l}\ul x^{\alpha}(e)+r_{i_1,\dots,i_l}(e) }\in A^E.
 \]
In particular, for any permutation $\sigma $ of the set $\{1,\dots,l\}$ 
 \[
\ol{\frac{1}{f_{i_1,\dots,i_l}}x_{i_1}\cdots x_{i_l}(e)-\frac{1}{f_{i_ {\sigma(1)},\dots,i_{\sigma(l)}}}x_{i_{\sigma(1)}}\cdots x_{i_{\sigma(l)}}(e)}= \ol t 
 \]
for some $t\in \Pn^E$ with $\lm_{\prec^E_o}(t) \prec^E_o (\alpha,e)$.
 Suppose now $\prec^E_o=\prec^E$ is fixed. If $\ul x^{\alpha}(e)\in (\Tn^E)^{\irr}_{A^E,\prec^E}$ then $f_{i_1,\dots,i_l}$ and $r_{i_1,\dots,i_l}$ can be additionally chosen such that
 $$\rho_{A^E,\prec^E}(x_{i_1}\cdots x_{i_l}(e))=f_{i_1,\dots,i_l}\ul x^{\alpha}(e)+ r_{i_1,\dots,i_l}(e).$$ 
 Otherwise $\lec_{\prec^E}(\rho_{A^E,\prec^E}(x_{i_1}\cdots x_{i_l}(e)))\prec^E (\alpha,e)$. 

\item\label{856c} Let $a\in A$ and $g\in A^E$. Then $\lec_{\prec^E}(ag)\preceq^E \lE_{\prec^E_{\lcomp(g)}}(a)+\lec_{\prec^E}(g)$ with equality if and only if the monomial with extended leading exponent $\lE_{\prec^E_{\lcomp(g)}}(a)+\lec_{\prec^E}(g)$ is irreducible.

\end{enumerate}
\end{rem}


\begin{proof}[Proof of \Cref{816}] 
By \Cref{818} it is clear that if $G$ is a Gröbner basis, then every normal form stated in the criterion is $0$.
Conversely, consider $0\neq m\in \lideal{A}{G}$ and choose $h\in A^G$ such that
 \begin{equation}\label{857}
 m=\sum_{g\in G} h_g g
 \end{equation}
satisfying additionally that
\[
 (\alpha,e):=\max\nolimits_{\prec^E} \{\lE_{\prec^E_{\lcomp(g)}}(h_g)+\lec_{\prec^E}(g)\mid g\in G\}
\]
is minimal with respect to $\prec^E$. If $(\alpha,e)\preceq^E \lec_{\prec^E}(m)$ then \Cref{857} is a standard representation and we are finished. Otherwise, set 
\[
 G':=\{g\in G\mid \lE_{\prec^E_{\lcomp(g)}}(h_g)+\lec_{\prec^E}(g)=(\alpha,e)\}
\]
and write
\begin{equation}\label{863}
 m=l+\sum_{g'\in G'} \tail_{\prec^E_e}(h_{g'}) g'+\sum_{g\in G\setminus G'} h_g g \quad \text{with}\quad l=\sum_{g'\in G'} \lt_{\prec^E_e}(h_{g'}) g'.
\end{equation}
By \Cref{856}.\ref{856c} and by choice of $G'$, we have for $g'\in G'$ 
\begin{align}\label{864}
 \lec_{\prec^E}( \tail_{\prec^E_e}(h_{g'}) g')\preceq^E& \lE_{\prec_e^E}( \tail_{\prec_e^E}(h_{g'}))+\lec_{\prec^E}(g')\\
 \prec^E & \lE_{\prec_e^E}(h_{g'})
 +\lec_{\prec^E}(g')=(\alpha,e),\nonumber
\end{align}
and for $g\in G\setminus G'$
\begin{equation}\label{865}
\lec_{\prec^E}(h_g g)\preceq^E \lE_{\prec_{\lcomp(g)}^E}( h_g)+\lec_{\prec^E}(g)\prec^E (\alpha,e).
\end{equation}
Hence
the leading monomial of $l$ is strictly smaller than $\ul x^\alpha (e)$.
We distinguish two cases: If $\ul x^\alpha(e)\in (\Tn^E)^{\irr}_{A^E,\prec^E}$ then all summands of $l$ have leading monomial $\ul x^\alpha(e)$ according to \Cref{856}.\ref{856c}. So \Cref{815} yields coefficients $d\in \KK^{G'\times G'}$ to write
\begin{equation}\label{825}
 l=\sum_{(g,g')\in G'\times G'} d_{(g,g')}s_{(g,g')},
\end{equation}
as a linear combination of $s$-polynomials 
\begin{align}\label{n7}
 s_{(g,g')}=&\spoly(\lm_{\prec_e^E}(h_g)g,\lm_{\prec_e^E}(h_{g'})g')\\
 =&\frac{1}{\lc_{\prec^E}(\lm_{\prec_e^E}(h_g)g)} \lm_{\prec_e^E}(h_g)g
-
\frac{1}{\lc_{\prec^E}(\lm_{\prec_e^E}(h_{g'})g') }\lm_{\prec_e^E}(h_{g'})g'.\nonumber
\end{align}
By definition of $c_{g,g'}$ and $c_{g',g}$ (see \Cref{817}.\ref{817b}) there exists $\beta_{(g,g')}\in \NN^n$ such that $c_{g,g'}+\beta_{(g,g')}=\lE_{\prec^E_e}(h_g)$ and $c_{g',g}+\beta_{(g,g')}=\lE_{\prec^E_e}(h_{g'})$. Applying \Cref{856}.\ref{856b}, we obtain
\begin{align*}\label{861}
s_{(g,g')}=& (d_g\ul x^{\beta_{(g,g')}}\ul x^{c_{g,g'}}+r^{(g,g')})g -(d_{g'}\ul x^{\beta_{(g,g')}} \ul x^{c_{g',g}}+{r}^{(g',g)})g' \\
=&\ul x^{\beta_{(g,g')}}\left( d_g\ul x^{c_{g,g'}}g -d_{g'} \ul x^{c_{g',g}}g' \right)+r^{(g,g')} g+r^{(g',g)}g'
\end{align*}
for suitably chosen $d_g,d_{g'}\in \KK^*$ and $r^{(g,g')},{r}^{(g',g)}\in A$ with 
\begin{equation*}
\lm_{\prec^E_e}(r^{(g,g')})\prec^E_e \lm_{\prec_e^E}(h_g)\text{ and }\lm_{\prec^E_e}({r}^{(g',g)})\prec^E_e \lm_{\prec_e^E}(h_{g'}).
\end{equation*}
Adding $\lec_{\prec^E}(g)$ and $\lec_{\prec^E}(g')$ respectively, we obtain 
\begin{equation}\label{866}
\lm_{\prec^E_e}(r^{(g,g')})+\lec_{\prec^E}(g),\ \lm_{\prec^E_e}({r}^{(g',g)})+ \lec_{\prec^E}(g') \prec^E_e (\alpha,e).
\end{equation}
As $\ul x^\alpha(e)$ is irreducible and ${c_{g,g'}}+\beta_{(g,g')}+\lm_{\prec^E}(g)=(\alpha,e)={c_{g',g}}+\beta_{(g,g')}+\lm_{\prec^E}(g')$, the monomial with extended leading coefficient $ {c_{g,g'}}+\lm_{\prec^E}(g)= {c_{g',g}}+\lm_{\prec^E}(g')$ is also irreducible.
By \Cref{n7} and \Cref{860} $\lm_{\prec^E}(s_{(g,g')})\prec \ul x^\alpha(e)$. Using \Cref{866} it follows 
that $\lt_{\prec^E}(d_g\ul x^{c_{g,g'}}g )=\lt_{\prec^E}(d_{g'} \ul x^{c_{g',g}}g')$.
By definition of $\spoly(g,g')$ it means that
$$ d_g\ul x^{c_{g,g'}}g -d_{g'} \ul x^{c_{g',g}}g'=f_{(g,g')} \spoly(g,g') $$
 for some $f_{(g,g')}\in \KK^*$.
Substituting into \Cref{n7} yields
\begin{equation}\label{862}
 s_{(g,g')}=f_{(g,g')}\ul x^{\beta_{(g,g')}}\spoly(g,g')+r^{(g,g')} g+{r'}^{(g,g')}g'
\end{equation}
and 
\begin{equation}\label{868}
 {\beta_{(g,g')}}+\lec_{\prec^E}(\spoly(g,g'))\prec^E (\alpha,e).
\end{equation}
By hypothesis we find an element $k^{(g,g')}\in A^G$ satisfying
\begin{equation}\label{826}
\spoly(g,g')=\sum_{g''\in G}k_{g''}^{(g,g')} g''
\end{equation}
and $\lE_{\prec^E_{\lcomp(g'')}}(k_{g''}^{(g,g')} )+\lec_{\prec^E}(g'')\preceq^E \lec_{\prec^E}(\spoly(g,g'))$. This yields together with \Cref{856}.\ref{856c} and \Cref{868} the estimate
 \begin{align}\label{867}
 \lE_{\prec^E_{\lcomp(g'')}}( \ul x^{\beta{(g,g'})}k_{g''}^{(g,g')} )+\lec_{\prec^E}(g'')&\preceq^E \beta_{(g,g')}+\lE_{\prec^E_{\lcomp(g'')}}(k_{g''}^{(g,g')} )+\lec_{\prec^E}(g'')\\\nonumber
 & \preceq^E \beta_{(g,g')}+\lec_{\prec^E}(\spoly(g,g'))\\\nonumber
 &\prec^E (\alpha,e).\nonumber
 \end{align}
Combining \Cref{825,862,826} we obtain
\[
 l=\sum_{(g,g')\in G'\times G'} d_{(g,g')} \left(f_{(g,g')}\sum_{g''\in G}\ul x^{\beta{(g,g'})} k_{g''}^{(g,g')} g''+r^{(g,g')} g+{r}^{(g',g)}g'\right)
\]
 and substituting into \Cref{863} contradicts the minimality of $(\alpha,e)$ by \Cref{864,865,866,867}.

In the other case, $\ul x^\alpha(e)$ is reducible, say $\alpha=\beta+\lm_{\prec^E_e}(p)$ for some $p\in I_e$ and $\beta\in \NN^n$. Then there exists by definition of $\spoly(g,p)$ for $g\in G'$ a vector $\gamma_{g}\in \NN^n$ such that 
\[
\lE_{\prec_e^E}(h_g)+\lec_{\prec^E}(g)=(\alpha,e)=\gamma_g+c_{g,p}+\lec_{\prec_E}(g)
\]
(see \Cref{817}.\ref{817c} for the definition $c_{g,p}$).
Therefore there is $q_g\in \KK^*$
\[
 \lm_{\prec_e^E}(h_g) g=(q_g\ul x^{\gamma_g} \cdot \ul x^{c_{g,p}}+t_g)g=q_g\ul x^{\gamma_g}\cdot \spoly(g,p)+ t_g g
\]
with $t_g\in A$ such that $\lE_{\prec_e^E}(t_g)\prec_e^E \lE_{\prec_e^E}(h_g)$ by \Cref{856}.\ref{856b}. Using that $$\gamma_g+\lec_{\prec^E}(\spoly(p,g))\prec^E \gamma_g+ c_{g,p}+\lec_{\prec_E}(g) =(\alpha,e)$$
by \Cref{860} and that $ \spoly(g,p)$ has a vanishing normal form with respect to $G$, we may argue as in the first case. This finishes our proof.
\end{proof}


The above \namecref{815} yields the following 
\namecref{829} for computing Gröbner bases:


\begin{algorithmbis}[Given a PBW-reduction-algebra $A$, a well-ordering $\prec^E$ and a finite set $G\subseteq A^E$, this algorithm computes a Gröbner basis of the module $\lideal{A}{G}$ with respect to $\prec^E$.]
\label[algorithm]{829}
\begin{algorithmic}[1]
\REQUIRE {A PBW-reduction-algebra $A$, a finite set $E$, $(A^E,\prec^E)=(\Tn,S_e,I_e,\prec_e)_{e\in E}$ and $G\subseteq A^E$ finite.}
\ENSURE A finite set $H\subseteq A^E$ such that $H$ is a Gröbner basis of $\lideal{A}{G}$ with respect to $\prec^E$.
\STATE Initialize $H:=G\setminus \{0\}:=\{g_1,\dots,g_s\}$.
\STATE Set $T:=\{(g_i,g_j)\mid 1\leq i<j\leq s \} \cup \{(g,p)\mid g\in H, p\in I_{\lcomp(g)}\}$.
\WHILE {$T\neq \emptyset$}
\STATE Choose $(t_1,t_2)\in T$ and delete it from $T$.
\STATE Compute a normal form $r$ of $\spoly(t_1,t_2)$ with respect to $H$ and $\prec^E$ by applying \Cref{828}.
\IF {$r\neq 0$}
\STATE Set $T:=T\cup\{(r,h)\mid h\in H\}\cup \{(r,p)\mid p\in I_{\lcomp(r)}\}$ and $H:=H\cup \{r\}$.
\ENDIF
\ENDWHILE
\RETURN $H$.
\end{algorithmic}
\end{algorithmbis}


\begin{lem}\label{819}
\Cref{829} is correct and terminates.
\end{lem}

\begin{proof} 
Correctness follows immediately from \Cref{816}. The $\operatorname{L}(H)$ form an increasing sequence of $\NN^n$-stable subsets of $\NN^n\times E$. By definition of a normal form it stabilizes exactly if $H$ does. Elements of $\NN^n\times E$ identify with monomials in $\Pn^E$. The latter is Noetherian and termination follows.
\end{proof} 


As in the commutative setting, the above algorithm can be modified to compute a reduced Gröbner basis.
An algorithm for computing left generators of a two-sided submodule of a free $A$-module carries over immediately from the setting of PBW-algebras (see e.g. \cite[Algorithm 6]{BuesoBook}). In our setting termination is a consequence of \Cref{813}.
Together with \Cref{819} this yields:


\begin{prp}\label{294}Let $A$ be a PBW-reduction-algebra, $E$ a finite set, $(A^E,\prec^E)=(\Tn,S_e,I_e,\prec_e)_{e\in E}$ and $G\subseteq A^E$ a finite subset. Then (reduced) Gröbner bases of the left $A$-modules
$\lideal{A}{G}$ and $\lrideal{A}{G}$ with respect to $\prec^E$ are computable. 
\end{prp}


\begin{dfn}\label{f1}
 We call $\Tn/\ideal{I'\cup S}=(\Tn,S,I',\prec)$ a \emph{factor PBW-reduction-algebra} of $A=(\Tn,S,I,\prec)$ if $I\subseteq I'$.
\end{dfn}


The following result explains how we consider factor algebras of PBW-reduction-algebras as PBW-reduction-algebras.


\begin{cor}\label{840}
Let $A=(\Tn,S,I,\prec)$ be a PBW-reduction-algebra and $M\subseteq A$ a two-sided $A$-ideal. Then
$
A/ {M}
$
is canonically isomorphic to the PBW-reduction-algebra
\[
\Tn/\ideal{S\cup I\cup \tau_{A,\prec}(G)}=(\Tn,S,I\cup \tau_{A,\prec}(G),\prec),
\]
 where $G$ is a left Gröbner basis of ${M}$ with respect to $\prec$.
\end{cor}

\begin{proof}
Clearly the map $\Tn\to A,\, t\mapsto \ol t$
induces the claimed isomorphism. For the second claim it is by \Cref{837} enough to show that
\[
\operatorname{L}_\prec(I\cup \tau_{(A,\prec)}(G))\supseteq \operatorname{L}_\prec(\lrideal{\Tn}{S\cup I\cup \tau_{A,\prec}(G)}\cap \Pn).
\]
So consider $0\neq t\in \lrideal{\Tn}{S\cup I\cup \tau_{A,\prec}(G)}\cap \Pn$. If $\lE(t)\in \operatorname{L}_\prec(I)$, we are finished. Otherwise we have according to \Cref{808b} that $\lm(t)$ is irreducible with respect to $(\lrideal{\Tn}{S\cup I},\prec)$ and hence $\lm(t)=\lm(\rho_{A,\prec}(t))=\lm(\ol{t})$. 
By construction $\ol t\in M$ and for suitable coefficients $a\in A^G$ there is a standard representation with respect to $G$ 
\[
\ol{t}=\sum_{g\in G} a_g g \text{ and } \lE(a_g)+\lE(g)\preceq \lE(\ol{t})=\lE(t) \text{ for all } g\in G
\]
with equality for some $g'\in G$. With $\lE(g')=\lE( \tau_{A,\prec}(g'))$ the claim follows.
\end{proof}
 

\begin{dfn}\label{296} Let $A$ be a ring, $E$ a finite set and $H_1,\dots, H_s\subseteq A^{E}$ finite subsets. The $A$-module
\[
 \syz_A(H_1,\dots,H_s):=\{(a_1,\dots,a_s)\in A^{H_1}\oplus\dots \oplus A^{H_s}\mid \sum_{1\leq i\leq s}\sum_{h_i\in H_i} (a_i)_{h_i} h_i=0\}
\]
is called the \emph{syzygy-module} of $H_1,\dots,H_s$ (in $A^{H_1}\oplus\dots \oplus A^{H_s}$).
Similarly, for $h_1,\dots,h_t\in A^E$ we define the \emph{syzygy-module} $
 \syz_{A}(h_1,\dots,h_t):=\syz_{A}(\{h_1\},\dots,\{h_t\}). 
$
\end{dfn}


Syzygies over PBW-reduction-algebras can be computed as in the commutative case:


\begin{lem}\label{834}
Let $A=(\Tn,S,I,\prec)$ be a PBW-reduction-algebra, $E$ a finite set and $H\subseteq A^E$ finite. Let $G$ be a Gröbner basis of $\lideal{A}{\{h+(h)\mid h\in H\} }\subseteq A^{E\sqcup H}$ with respect to $(<,\prec^{E\sqcup H})$, where $<$ is a total ordering on $E\sqcup H$ with $h<e$ for $e\in E$ and $h\in H$. Then
\[
\syz_A(H)=\lideal{A}{ \pi_H(G\cap A^H)}.
\]
\end{lem} 


\begin{rem}\label{1001}
Given a PBW-reduction algebra $A=(\Tn,S,I,\prec)$ and a finite set $E$ the following Gröbner basics can be performed as in the commutative setting: 
 \begin{enumerate}
 \item\label{1001a} 
 We can decide for submodules of $A^E$ whether one is included in the other using normal form computations with respect to any ordering $\prec^E$ with $\prec^E_e=\prec$ for all $e\in E$ (see \Cref{833}\eqref{833a}).
 \item\label{1001b}
 A non-commutative variant of \cite[Section 2.8.3]{SingularBook} allows to compute intersections of submodules of $A^E$,
 \end{enumerate}
\end{rem}


In the next \namecref{368}, we explain how to compute Gröbner bases with respect to non-well-orderings.

\section{Weight filtrations}\label{368}

The subject of investigation in this \namecref{368} are filtrations of type $F^\mathbf{u}_\bullet A$ induced by a so-called weight vector $\mathbf{u}$ on the PBW-reduction-algebra $A$. These filtrations have been studied theoretically and algorithmically for nonnegative weight vectors on PBW-algebras in \cite{BuesoBook}. 
Combining the methods of \cite{BuesoBook} and \cite{OT2001}, we develop a Gröbner basis algorithm 
for computing $F^\mathbf{u}_\bullet A$ for general weight vectors $\mathbf{u}$. 

\subsection{Weight filtrations on PBW-reduction-algebras}\label{297}

In this \namecref{297} let $A=(\Tn,S,I,\prec)$ be a PBW-reduction-algebra unless stated otherwise.


\begin{dfn}\label{298}
 Let $\mathbf{u}\in \ZZ^n$, $E$ a finite set, $L\subseteq A^E$ an $A$-submodule and $\mathbf{s}\in \ZZ^E$.
 \begin{enumerate}

\item\label{298c} Assigning weight $\mathbf{u}_{i}$ to $x_i$ and weight $\mathbf{s}_e$ to $(e)$ defines a $\ZZ$-grading on $\Tn^E$ with $l$th graded piece
$
 (\Tn^E)_l^{\mathbf{u}[\mathbf{s}]}=\bigoplus_{e\in E}\Tn_{l-\mathbf{s_e}}^\mathbf{u} (e)
$
with $$\Tn_{l}^\mathbf{u}:=\lideal{\K}{\{x_{i_1}\cdots x_{i_k}\mid k\in \NN, 1\leq i_1,\dots,i_k\leq n, \sum\nolimits_{1\leq j\leq k} \mathbf{u}_{i_j}=l\}}.$$
So every $0\neq r\in \Tn^E$ can be uniquely written as $r=\sum_{s_1\leq i\leq s_2} r_i$ with $r_i\in (\Tn^E)_i^{\mathbf{u}[\mathbf{s}]}$ and $r_{s_1},r_{s_2}\neq 0$. We call $s_2$ the \emph{${\mathbf{u}[\mathbf{s}]}$-degree} of $r$ and write $\deg_{\mathbf{u}[\mathbf{s}]}(r)=s_2$. We set $\deg_{\mathbf{u}[\mathbf{s}]}(0):=-\infty$. If $s_1=s_2$, we say that $r$ is \emph{${\mathbf{u}[\mathbf{s}]}$-homogeneous}. We define the \emph{${\mathbf{u}[\mathbf{s}]}$-leading terms} of $r$ by $\lt_{\mathbf{u}[\mathbf{s}]}(r):=r_{s_2}$. The elements $r_{s_1},\dots,r_{s_2}$ are called the \emph{${\mathbf{u}[\mathbf{s}]}$-homogeneous parts} of $r$.

\item\label{298a} The associated filtration on $F^\mathbf{u}[\mathbf{s}]_\bullet$ on $\Tn^E$ is defined by $$F^\mathbf{u}[\mathbf{s}]_k \Tn^E:=\{r\in \Tn^E\mid \deg_{\mathbf{u}[\mathbf{s}]}(r)\leq k\}$$ for $k\in \ZZ$. 
 It induces a quotient filtration $F^\mathbf{u}[\mathbf{s}]_\bullet A^E:=\bigoplus_{e\in E} F^\mathbf{u}_{\bullet-\mathbf{s}_e} A(e)$, where $F_\bullet^\mathbf{u} A:=({F_\bullet^\mathbf{u} \Tn}+\ideal{I\cup S})/\ideal{I\cup S}$.
We define the ${\mathbf{u}[\mathbf{s}]}$-degree for $a\in A^E$ by
 \[
 \deg_{\mathbf{u}[\mathbf{s}]}(a):=\deg_{ F^{\mathbf{u}}[\mathbf{s}]}(a):=\inf\{k\in \ZZ\mid a\in F^\mathbf{u}[\mathbf{s}]_k A^E\}
 \]
and extend it to subsets $T$ of $A^E$ or $\Tn^E$ is by setting $\deg_{\mathbf{u}[\mathbf{s}]} (T):=\max\{ \deg_{\mathbf{u}[\mathbf{s}]}(t)\mid t\in T\}$. 

\item\label{298e} Using the induced filtrations $F^{\mathbf{u}}[\mathbf{s}]_\bullet L:=F^{\mathbf{u}}[\mathbf{s}]_\bullet A^E\cap L$ and $F^{\mathbf{u}}[\mathbf{s}]_\bullet (A^E/L):=({F^{\mathbf{u}}[\mathbf{s}]_\bullet A^E}+L)/L$, we sets the ${\mathbf{u}[\mathbf{s}]}$-degree of elements and subsets of $L$ and $A^E/L$ as in Part~\ref{298a}.

\item We call $\gr^{\mathbf{u}}\!A:=\bigoplus_{k\in \ZZ} F^\mathbf{u}_{k} A/F^\mathbf{u}_{k-1} A$ the \emph{$\mathbf{u}$-graded algebra associated with $A$} and $\gr^{\mathbf{u}[\mathbf{s}]}\!L:=\bigoplus_{k\in \ZZ} F^\mathbf{u}[\mathbf{s}]_{k} L/F^\mathbf{u}[\mathbf{s}]_{k-1} L$ \emph{$\mathbf{u}$-graded module associated with $L$}.

\item\label{298b} We say that $\mathbf{u}$ is a \emph{weight vector} on $A$ if
 $\deg_\mathbf{u}(d_{ij})\leq \deg_\mathbf{u}(x_ix_j)$ for all $1\leq i<j\leq n$, where $S=\{x_jx_i-c_{ij}x_ix_j-d_{ij}\mid 1\leq i<j\leq n\}$. 
In this case we call $F^{\mathbf{u}}_\bullet A$ the \emph{weight filtration} associated to $\mathbf{u}$ on $A$ or the \emph{$\mathbf{u}$-weight filtration} on $A$. If $A\cong \gr^{\mathbf{u}}\!A$, then we say that $A$ is \emph{$\mathbf{u}$-graded} and speak of \emph{$\mathbf{u}$-homogeneous} elements of $A$. 
More generally, if $A$ is $\mathbf{u}$-graded, $E$ a finite set and the shift vector $\mathbf{s}\in \ZZ^E$ assigns degree $\mathbf{s}_e$ to $(e)$, then we call a homogeneous element of $A^E$ also $\mathbf{u}[\mathbf{s}]$-homogeneous. 

\end{enumerate}
We often suppress $\mathbf{s}$ in the above notations if it is the zero vector.
\end{dfn}


\begin{lem}\label{300}
 Let $\mathbf{u}\in \ZZ^n$ be a weight vector on $A$, $E$ a finite set, $\mathbf{s}\in \ZZ^E$ and $L\subseteq A^E$ an $A$-submodule. 
Then we have for all $a,a'\in A$
\[
 \deg_{\mathbf{u}}(a\cdot a') \leq \deg_{\mathbf{u}}(a)+\deg_{\mathbf{u}}(a').
\]
In particular, $F_\bullet^\mathbf{u}A=\lideal{\K}{\{\ol{\ul x^\alpha}\mid \langle \mathbf{u},\alpha\rangle \leq \bullet\}}$ is a filtered $\K$-algebra
and
$F^{\mathbf{u}}[\mathbf{s}]_\bullet A^E$, $F^{\mathbf{u}}[\mathbf{s}]_\bullet L$ and $F^{\mathbf{u}}[\mathbf{s}]_\bullet (A^E/L)$ are filtered $F^{\mathbf{u}}_\bullet A$-modules.
\end{lem}


 \begin{dfn}\label{349}
 Let $\mathbf{u}\in \ZZ^n$ be a weight vector on $A$, $E$ a finite set, $\prec^E$ an ordering on $A^E$ and $\mathbf{s}\in \ZZ^E$ a shift vector. We define the ordering $\prec_{\mathbf{u}[\mathbf{s}]}^E$ on $\SMon(\Tn^E)$ by
 \begin{align*}
 \ul x^\alpha (e) \prec_{\mathbf{u}[\mathbf{s}]}^E \ul x^{\alpha'} (e') \text{ if and only if } & \deg_{\mathbf{u}[\mathbf{s}]}(\ul x^\alpha (e))<\deg_{\mathbf{u}[\mathbf{s}]}(\ul x^{\alpha'} (e'))\\
 \text{ or }& \deg_{\mathbf{u}[\mathbf{s}]}(\ul x^\alpha (e))=\deg_{\mathbf{u}[\mathbf{s}]}(\ul x^{\alpha'} (e'))\text{ and } \ul x^\alpha (e) \prec^E \ul x^{\alpha'} (e')
 \end{align*}
 for $\alpha,\alpha'\in \NN^n$ and $e,e'\in E$.
 If $\mathbf{s}$ is the zero vector, we also write $\prec_{\mathbf{u}}^E$. We sometimes use the notation $\prec_{\mathbf{u}[\mathbf{s}]}^E$ without explicitly defining an ordering $\prec^E$ on $A^E$.
 \end{dfn}


\Cref{300} implies that $F^{\mathbf{u}}_0 A$ is a $\K$-subalgebra of $A$ if $\mathbf{u}$ is a weight vector on $A$.

\begin{lem}\label{220} Let $\mathbf{u}\in \ZZ^n$ be a weight vector on $A$. 
\begin{enumerate}

\item\label{220a} The $\K$-subalgebra $F^{\mathbf{u}}_0 A$ of $A$ is generated by residue classes of finitely many standard monomials.
 Moreover, such a generating set is computable. 

 \item \label{220e} The $\K$-subalgebra $F_0^\mathbf{u}A$ is isomorphic to a PBW-reduction-algebra.
 
 \item\label{220c} The $F^{\mathbf{u}}_0 A$-modules $F^{\mathbf{u}}_j A$ ($j\in \ZZ$) are 
 generated by residue classes of finitely many standard monomials.
 Moreover, such generating sets are computable. 

 \end{enumerate}
\end{lem}

\begin{proof}\
\begin{enumerate}

\item Taking exponents $\SMon(\Tn)\cap F^{\mathbf{u}}_0 \Tn $ identifies with
\begin{equation*}
U_0:=\{\alpha\in \RR^n\mid \langle \mathbf{u},\alpha\rangle \leq 0\}\cap \NN^n,
\end{equation*}
which is an intersection of a rational cone and the lattice $\ZZ^n$.
Therefore $U_0$ is a positive affine monoid by Gordan's lemma (see e.g. \cite[Lemma 2.9]{Bruns}), and has a computable minimal finite generating set \cite[Proposition 3.4.6]{Koch} \citep{BrunsComp}, say $\alpha_1,\dots,\alpha_s\in \ZZ^n$. This means that
$
U_0=\sum_{1\leq i\leq s} \NN\cdot \alpha_i,
$
and if $\alpha_i=\beta_1+\beta_2$ with $\beta_1,\beta_2\in U_0$, then $\beta_1=\alpha_i$ or $\beta_2=\alpha_i$ for $1\leq i\leq s$.

We claim that $F^{\mathbf{u}}_0 A=\K[\ol{\ul x^{\alpha_1}},\dots, \ol{\ul x^{\alpha_s}}]$.
By \Cref{300}, it suffices to show that ${F^{\mathbf{u}}_0 \Tn\cap {\SMon(\Tn)}}$ maps to $\K[\ol{\ul x^{\alpha_1}},\dots, \ol{\ul x^{\alpha_s}}]$. We proceed by induction on the well-ordering $\prec$. 
The base case is immediate since $1=\min_{\prec}\{ F^{\mathbf{u}}_0 \Tn\cap \SMon(\Tn) \}$. Consider now $\ul x^\alpha\in F^{\mathbf{u}}_0 \Tn\cap \SMon(\Tn)$. Then $\alpha\in U_0$, and hence $\alpha=\sum_{1\leq i\leq s}l_i\alpha_i$ for some $l\in \NN^s$. By \Cref{856}\eqref{856b} there exists $c\in \KK^*$ and $a\in \lideal{\K}{\SMon(\Tn)}$ with $\lm(a)\prec \ul x^\alpha$ such that
\[
\ol{\ul x^\alpha}=\ol{\ul x^{\sum_{1\leq i\leq s}l_i\alpha_i}}=c\ol{(\ul x^{\alpha_1})^{l_1}\cdots (\ul x^{\alpha_s})^{l_s}}+\ol{a}.
\]
The induction hypothesis applied to
 $\ol a\in F^{\mathbf{u}}_0 A$ yields the claim. 

\item We retain the notation of Part~\eqref{220a}. Consider the surjective $\K$-algebra map 
\[
 \pi: \KK\langle \ul y\rangle:=\KK\langle y_1,\dots,y_s\rangle \to F_0^\mathbf{u}A,\; y_i\mapsto \ol{\ul x^{\alpha_i}}.
\]
By \Cref{856}\eqref{856b} with $\prec_o^E=\prec_\mathbf{u}$ there exist $f_{ij}\in \K^*$ and $g_{ij}\in \Pn$ with $\lE_{\prec}(g_{ij})\prec \alpha_i+\alpha_j$ and $\deg_\mathbf{u}(g_{ij})\leq \deg_\mathbf{u}(\ul x^{\alpha_i}\ul x^{\alpha_j})\leq 0$ such that
\[
 \ul x^{\alpha_j}\ul x^{\alpha_i}-f_{ij}\ul x^{\alpha_i}\ul x^{\alpha_j}-g_{ij}\in \lrideal{\Tn}{S}\subseteq \lrideal{\Tn}{S,I}
\]
for $1\leq i<j\leq s$.
Then $g_{ij}\in F_0^\mathbf{u} A$ and Part~\eqref{220a} yield $g_{ij}'(y_1,\dots,y_s)\in \K[ \ul y]$ such that $g'_{ij}(\ol{\ul x^{\alpha_1}},\dots, \ol{\ul x^{\alpha_s}})=\ol{g_{ij}}\in A$.
It follows
\[
 S_0:=\{y_jy_i-f_{ij}y_iy_j-g_{ij}'\mid 1\leq i<j\leq s\}\subseteq \ker(\pi).
\]
Define the well-ordering $\prec_0$ on $\SMon(\K\langle \ul y\rangle)$ by
\begin{align*}
 \ul y^\beta\prec_0 \ul y^\gamma \text{ if and only if } & \sum_{1\leq k\leq s}\beta_i \alpha_i\prec \sum_{1\leq k\leq s} \gamma_i\alpha_i \\
 \text{ or } &\sum_{1\leq k\leq s}\beta_i \alpha_i=\sum_{1\leq k\leq s} \gamma_i\alpha_i \text{ and } \ul y^\beta\prec' \ul y^\gamma,
\end{align*}
where $\beta,\gamma\in \NN^s$ and $\prec'$ is some well-ordering on $\SMon(\K\langle \ul y\rangle)$. By construction, $(S_0,\prec_0)$ is a commutation system. We conclude that $\K\langle \ul y\rangle /\ker \pi$ is a PBW-reduction-algebra isomorphic to $F_0^\mathbf{u} A$ by \Cref{841}.
 
\item We keep the notation of Part~\eqref{220a} and consider first the case $j<0$. Let $\Delta_j:=\{\alpha_i\mid \langle \mathbf{u},\alpha_i\rangle \leq j\} $ and $\Delta'_j:=\{\alpha_i\mid j<\langle \mathbf{u}, \alpha_i\rangle<0\}$.
One easily checks that 
\begin{equation}\label{d1}
 U_j:=\{\alpha\in \NN^n\mid \langle \mathbf{u},\alpha\rangle\leq j\}=U_0+V_j:=\{\alpha+v\mid \alpha\in U_0,v\in V_j\},
 \end{equation}
 where
 $V_j:=\Delta_j\cup (U_j\cap
 \{\sum_{\delta\in \Delta'} l_\delta \delta\mid l\in \NN^{\Delta'}, |l|\leq j\})$.
We claim that 
\begin{equation}\label{d2}
F^\mathbf{u}_j A=\sum_{v\in V_j} F_0^\mathbf{u}A\ol{\ul x^v}.
\end{equation}
By \Cref{300} it suffices to show that $ F^\mathbf{u}_j \Tn\cap \SMon(\Tn)$ maps to the right hand side. This set has a minimal element $\ul x^\beta$ with respect to $\prec$ and we must have $\beta\in V_j$.
Proceeding by induction, the claim follows as in Part~\eqref{220a}.

The case $j=0$ being clear, we assume now $j>0$. Arguing as in the proof of Part~\eqref{220a}, we can compute a minimal finite set of generators $\Gamma$ of $\{\alpha\in \NN^n\mid \langle \mathbf{u},\alpha\rangle\geq 0\}$. As above, we obtain
\[
U_j:=\{\alpha\in \NN^n\mid \langle \mathbf{u},\alpha\rangle\leq j\}= U_0 +V_j,
\] 
where $V_j:=(\{\sum_{\gamma\in \Gamma_j} l_\gamma \gamma\mid l\in \NN^{\Gamma}, |l| \leq j\}\cap (U_j\setminus U_0))\cup \{0\}$ with 
 $\Gamma_j:=G\cap (U_j\setminus U_0)$. 
 With this notation \Cref{d2} follows as in case $j<0$.\qedhere
\end{enumerate}
\end{proof} 


\begin{ntn}\label{d3}
 Let $\mathbf{u}\in \ZZ^n$ be a weight vector on $A$.
 \begin{enumerate}

 \item\label{d3a} With notation from the proof of \Cref{220}.\ref{220a} we denote the set $G_A^{\mathbf{u}}:=\{{\ul x^{\alpha_1}},\dots,{\ul x^{\alpha_s}}\}$, whose residue classes generate $F_0^{\mathbf{u}}A$ as $\K$-algebra. 

 \item\label{d3b} With notation from the proof of \Cref{220}.\ref{220c} we denote for $j\in \ZZ$ the set $P_j^{A,\mathbf{u}}:=\{\ol{\ul x^v}\mid v\in V_j\}$, whose residue classes generate $F^{\mathbf{u}}_jA$ as $F^{\mathbf{u}}_0 A$-module.

 \end{enumerate}
\end{ntn}


\begin{rem}\label{345} Let $\mathbf{u}\in \ZZ^n$ be a weight vector on $A$.
The proof of \Cref{220}.\ref{220a} and \ref{220c} is constructive. It gives a method to represent an element of $F_0^{\mathbf{u}}A$ as a $\K$-linear combination of products of elements in $G_A^{\mathbf{u}}$ and elements $F_j^{\mathbf{u}}A$ as $F_0^{\mathbf{u}}A$-linear combinations of elements in $P_j^{A,\mathbf{u}}$.
\end{rem}


\begin{rem}\label{885}\
\begin{enumerate}
 \item\label{885a} Let $A=(\KK\langle \ul x,\ul y\rangle, S,I,\prec)$ (with $\KK\langle \ul x,\ul y\rangle:=\KK\langle x_1,\dots,x_n,y_1,\dots,y_m\rangle$) be an elementary PBW-reduction-algebra and $\mathbf{v}\in \ZZ^{n+m}$ be any weight vector on $A$.
Note that $\mathbf{w}=((0)_{1\leq i\leq n}, (1)_{1\leq i\leq m})$ is a weight vector on $A$. Then
\[
 F_k^\mathbf{v}A\cap F_l^\mathbf{w} A=\left(({F_k^\mathbf{v}\KK\langle \ul x,\ul y\rangle \cap F_l^\mathbf{w}\KK\langle \ul x,\ul y\rangle \cap \KK[ \ul x,\ul y] })+\ideal{I\cup S}\right)/\ideal{I\cup S}
\]
for all $k,l\in \ZZ$:
Indeed let $a\in F_k^\mathbf{v}A\cap F_l^\mathbf{w} A$. By \Cref{300} there exist representatives $a^{\mathbf{w}}\in F_l^\mathbf{w}\KK\langle \ul x,\ul y\rangle \cap \KK[ \ul x,\ul y]$ and $a^{\mathbf{v}}=\sum_{(\alpha,\beta)\in \NN^{n+m}}a^\mathbf{v}_{(\alpha,\beta)}\ul x^\alpha \ul y^\beta \in F_k^\mathbf{v}\KK\langle \ul x,\ul y\rangle$ of $a$.
By \Cref{923} $A= \bigoplus_{\beta\in \NN^m} (\KK[\ul x]/\ideal{I})\ul y^\beta$, 
and hence 
$\ol{\sum_{\alpha\in \NN^{n}} a^\mathbf{v}_{(\alpha,\beta)} \ul x^\alpha}=0\in A$ for all $\beta\in \NN^m$ with $|\beta|>l$. Thus $\sum_{(\alpha,\beta)\in \NN^{n+m}, |\beta|\leq l}a^\mathbf{v}_{(\alpha,\beta)}\ul x^\alpha \ul y^\beta\in F_k^\mathbf{v}\KK\langle \ul x,\ul y\rangle \cap F_l^\mathbf{w}\KK\langle \ul x,\ul y\rangle$ is also a representative of $a$. 

 \item\label{885b} Let $\mathbf{v}$ and $\mathbf{w}\in \ZZ^n$ be weight vectors on $A=(\Tn,S,I,\prec)$ such that
 \begin{equation}\label{886}
 F_k^\mathbf{v}A\cap F_l^\mathbf{w} A=\left( ({F_k^\mathbf{v}\Tn \cap F_l^\mathbf{w}\Tn \cap \Pn})+\ideal{I\cup S} \right)/\ideal{I\cup S}
\end{equation}
for $k,l\in \ZZ$.
By construction (see proof of \Cref{220}\eqref{220c})
\[
 F_\bullet^\mathbf{w} F_k^\mathbf{v} A=\sum_{p\in P_k^{A,\mathbf{v}}} (F_{\bullet-\langle \beta_i^k,\mathbf{w}\rangle}^\mathbf{w} F_0^\mathbf{v} A)\cdot \ol{p}.
\]

\end{enumerate}
\end{rem}


\begin{exa}\label{313}
In the situation of \Cref{577}.\ref{577c}, we have $T_X^V\cong F_0^\mathbf{v} T_X$, where $\mathbf{v}$ is the weight vector assigning weights $-1$ and $1$ to $x_n$ and $y_m$, respectively, and weight $0$ otherwise. 
By this example and \Cref{885} the weight vector $\mathbf{w}=((0)_{1\leq i\leq n},(1)_{1\leq i\leq m})$ on $T_X$ induces the weight vector $\mathbf{w}_\mathbf{v}=((0)_{1\leq i\leq n},(1)_{1\leq i\leq m})$ on $T_X^V$. 
Moreover, we may assume
\[
P_k^{T_X,\mathbf{v}}=\begin{cases} \{{x_{n}^k}\}&\text{if }k\leq 0,\\
\{{y_{m}^{l}}\mid 0\leq l\leq k\}& \text{otherwise}
\end{cases}
\]
and fix this choice.
\end{exa}

\subsection{Weight filtrations on submodules of free modules}\label{334}

In this \namecref{334} let $A=(\Tn,S,I,\prec)$ be a PBW-reduction-algebra unless otherwise specified and $\mathbf{u}\in \ZZ^n$ a weight vector on $A$.
For a given set $E$, an $A$-submodule $M\subseteq A^E$ and a shift vector $\mathbf{s}\in \ZZ^E$ we show how to compute a finite set of generators $M'$ of the filtration $F_\bullet^{\mathbf{u}}[\mathbf{s}] M$.
By definition this means that for every $m\in M$ there are coefficients $a\in A^{M'}$ such that
 \[
 m=\sum_{m'\in M'} a_{m'} m' \text{ and } \deg_\mathbf{u}(a_{m'})+\deg_{\mathbf{u}[\mathbf{s}]}(m')\leq \deg_{\mathbf{u}[\mathbf{s}]}(m) \text{ for all } m'\in M'.
\]

 
 \begin{lem}\label{844} 
Let $\mathbf{u}\in \ZZ^n$ be a weight vector on $A$, $E$ a finite set, $\mathbf{s}\in \ZZ^E$ a shift vector, $\prec^E$ an ordering and $M\subseteq A^E$ an $A$-submodule. If $G$ is a Gröbner basis of $M$ with respect to $\prec^E_{\mathbf{u}[\mathbf{s}]}$, then it generates $F^\mathbf{u}[\mathbf{s}]_\bullet M$ as $F_\bullet^\mathbf{u}A$-module.
 \end{lem}

 \begin{proof}
 Let $m\in F^\mathbf{u}[\mathbf{s}]_k M$ for some $k\in \ZZ$. Choose a representative $m'\in F^\mathbf{u}[\mathbf{s}]_k \Tn\cap \Pn$ of $m$ using \Cref{300}. By assumption there is $a\in \lideal{\K}{\SMon(\Tn)}^G$ and $h\in (\Pn^E)^G$ with $\ol{h_g}=g$ satisfying
 \[
 m=\sum_{g\in G} \ol{a_g }g \text{ and } \lE(a_g)+\lec(h_g)\preceq^E_{\mathbf{u}[\mathbf{s}]}\lec(m')
 \]
implying $\deg_{\mathbf{u}}(\ol{a_g})+\deg_{\mathbf{u}[\mathbf{s}]}( g)\leq \deg_{\mathbf{u}}({a_g})+\deg_{\mathbf{u}[\mathbf{s}]}(h_g)\leq \deg_{\mathbf{u}}(m')\leq k$. This means that $m\in \sum_{g\in G} F^\mathbf{u}_{k-\deg_{\mathbf{u}[\mathbf{s}]}( g)} A\cdot g$. 
 \end{proof}

 
Note that if $\prec^E$ is a well-ordering, then $\prec_{\mathbf{u}[\mathbf{s}]}^E$ is a well-ordering if and only if $\mathbf{u}\in \NN^n$. Since Gröbner bases with respect to well-orderings exist by \Cref{294}, a finite set of generators of $F^\mathbf{u}_\bullet M$ exists in this case.
If $\mathbf{u}\notin \NN^n$, we can still compute Gröbner bases with respect to $\prec_{\mathbf{u}[\mathbf{s}]}^E$.
To this end we homogenize $A$ with respect to a weight vector $\mathbf{w}$:


\begin{dfn}\label{319}
 Let $\mathbf{w}\in \NN^n$ be a weight vector on $A$, $E$ a finite set and $\mathbf{s}\in \ZZ^E$ a shift vector.
\begin{enumerate}

 \item\label{319e} We define the \emph{$\mathbf{w}[\mathbf{s}]$-homogenization} of $0\neq p=\sum_{m\in \Mon(\Tn^E)} p_{m} m\in \Tn^E$ (with $p_m\in \KK$) as
 \[
 h_{\mathbf{w}[\mathbf{s}]}(p):=\sum_{m\in \Mon(\Tn^E)} p_{m} h^{\deg_{\mathbf{w}[\mathbf{s}]}(p)-\deg_{\mathbf{w}[\mathbf{s}]}(m)} m \in \KK\langle h,\ul x\rangle^E
 \]
 and set $ h_{\mathbf{w}[\mathbf{s}]}(0)=0$.
 For $G\subseteq \Tn^E$, we set $h_{\mathbf{w}[\mathbf{s}]}(G):=\{h_{\mathbf{w}[\mathbf{s}]}(g)\mid g\in G\}$. We suppress $\mathbf{s}$ if it is the zero vector.

 \item\label{319a} The \emph{$\mathbf{w}$-homogenized} algebra associated with $A$ is the $(1,\mathbf{w})$-graded algebra
 \[
 A^{\mathbf{w}}=\KK\langle h,\ul x\rangle/\ideal{h_\mathbf{w}(\lrideal{\Tn}{S\cup I})\cup \{hx_i-x_ih\mid 1\leq i\leq n\}}.
 \]

 \item\label{319c} We define the ordering $(\prec^E)^\mathbf{w}$ on $\SMon(\KK\langle h,\ul x\rangle^E)$ for the ordering $\prec^E$ on $A^E$ by
 \begin{align*}
 h^\alpha \ul x^\beta (e)(\prec^E)^\mathbf{w} h^{\alpha'} \ul x^{\beta'} (e')\text{ if and only if }& \alpha+\langle \mathbf{w},\beta \rangle <{\alpha'}+\langle \mathbf{w},{\beta'} \rangle\\
 \text{ or }& \alpha+\langle \mathbf{w},\beta \rangle ={\alpha'}+\langle \mathbf{w},{\beta'} \rangle \text{ and }\ul x^\beta (e) \prec^E \ul x^{\beta'} (e')
 \end{align*}
 for $\alpha,\alpha'\in \NN$, $\beta,\beta'\in \NN^n$ and $e,e'\in E$.

\item\label{319d} We call the $\K$-algebra homomorphism given by
\[
 d_h:\KK\langle h,\ul x\rangle\to \Tn,\;h\mapsto 1, x_i\mapsto x_i
\]
\emph{dehomogenization map}. It induces a map $d_h: A^{\mathbf{w}}\to A$. We denote also the maps $ \bigoplus_{e\in E} d_h$ 
 by $d_h$.
 
 \end{enumerate}
\end{dfn}


Homogenized PBW-reduction-algebras are PBW-reduction-al\-ge\-bras:


\begin{lem}\label{845}
 Let $\mathbf{w}\in \NN^n$ be a weight vector on $A$ and set
 $S^{\mathbf{w}}:=h_\mathbf{w}(S)\cup \{hx_i-x_ih\mid 1\leq i\leq n\}$. Then $A^\mathbf{w}=(\K\langle h, \ul x\rangle, S^{\mathbf{w}}, I^\mathbf{w},\prec^{\mathbf{w}}) $ for some set $I^\mathbf{w}\subseteq \allowbreak\K[h,\ul x]$ with $(1,\mathbf{w})$-homogeneous elements. 
In particular, if $A$ is a PBW-algebra, then so is $A^{\mathbf{w}}$.
 
 Moreover, if $\prec'$ is any ordering on $A$, then $(\prec')^{\mathbf{w}}$ is an ordering on $A^{\mathbf{w}}$. If $\mathbf{w}$ is strictly positive, then there exists a finite set ${I'}^\mathbf{w}$ consisting of $(1,\mathbf{w})$-homogeneous elements such that $(\K\langle h, \ul x\rangle, S^{\mathbf{w}}, {I'}^\mathbf{w},(\prec')^{\mathbf{w}})$ is a PBW-reduction datum.
\end{lem}

\begin{proof} Write $S:=\{x_jx_i-c_{ij}x_ix_j-d_{ij}\mid 1\leq i<j\leq n\}$. Since $\mathbf{w}$ is a weight vector on $A$, we have $h_\mathbf{w}(x_jx_i-c_{ij}x_ix_j-d_{ij})=x_jx_i-c_{ij}x_ix_j-h^{\alpha_{ij}} h_\mathbf{w}(d_{ij})$ for $1\leq i<j\leq n $ and some $\alpha_{ij}\in \NN$.
Then $(S^{\mathbf{w}},\prec^{\mathbf{w}})$ is a commutation system by definition of $\prec^{\mathbf{w}}$ and $A^\mathbf{w}=(\K\langle h, \ul x\rangle, S^{\mathbf{w}}, I^\mathbf{w}, \prec^{\mathbf{w}})$ is a PBW-reduction-algebra by \Cref{841}.
Using that
$A^\mathbf{w}$ is $(1,\mathbf{w})$-graded, we replace $I^\mathbf{w}$ by the set of the $(1,\mathbf{w})$-homogeneous parts of its elements. The particular claim follows now from \Cref{1000} and \Cref{808}.

As above, $(S^{\mathbf{w}},(\prec')^{\mathbf{w}})$ is a commutation system. If $\mathbf{w}$ is strictly positive, then $(\prec')^{\mathbf{w}}$ is a well-ordering and \Cref{841} yields the corresponding PBW-reduction datum.
\end{proof}


Our approach is to homogenize $A$ by a strictly positive weight vector $\mathbf{w}\in \NN_{>0}^n$. This reduces Gröbner basis computations in $A^E$ with respect to the non-well-ordering $\prec^E$ to Gröbner basis computations in $(A^{\mathbf{w}})^E$ with respect to the well-ordering $(\prec^E)^{\mathbf{w}}$. The existence of such a weight vector is guaranteed by the following \namecref{394}:


\begin{lem}\label{394}
A weight vector $\mathbf{w}\in \NN_{>0}^n$ on $A$ exists and is effectively computable.
\end{lem}

\begin{proof}
 Consider the set $M$ of standard monomials appearing with nonzero coefficient in $S$. According to \cite[Lemma 1.2.11, Exercises 1.2.7 and 1.2.9]{SingularBook} there is a computable $\mathbf{w}\in \NN^n_{>0}$ such that
 \[
 \ul x^\alpha \prec \ul x^\beta\text{ if and only if } \langle \alpha,\mathbf{w}\rangle <\langle \beta,\mathbf{w}\rangle
 \]
for all $\ul x^\alpha,\ul x^\beta\in M$.
As $\prec$ is an ordering on $A$, $\mathbf{w}$ is a weight vector on $A$.
\end{proof}

If $A$ is an elementary PBW-reduction-algebra, we compute a PBW-reduction datum for the homogenized PBW-reduction-algebra $A^{\mathbf{w}}$ with respect to the weight vector $\mathbf{w}\in \NN_{>0}^n$ as follows:


\begin{lem}\label{875} 
Consider the $\K$-algebra $\KK\langle \ul x,\ul y\rangle:=\KK\langle x_1, \allowbreak \dots,\allowbreak x_n,\allowbreak y_1,\dots,y_m\rangle$, the elementary PBW-reduction-algebra $A=(\KK\langle \ul x,\ul y\rangle, S ,I,\prec)$ and the weight vector $\mathbf{w}\in \NN^{n+m}_{>0}$ on $A$.
Then $A^{\mathbf{w}}$ is an elementary PBW-reduction-algebra.
In addition, if $\prec'$ is an ordering on $A$ and $A= (\KK\langle \ul x,\ul y\rangle, S, I',\prec_{\mathbf{w}}') $, then
$$A^\mathbf{w}=(\K\langle h,\ul x,\ul y\rangle, S^{\mathbf{w}}, I^\mathbf{w}, (\prec')^{\mathbf{w}}),$$ 
where $I^\mathbf{w}$ is a Gröbner basis of $\ideal{h_\mathbf{w}(I')}\subseteq \KK[h,\ul x]$ with respect to the ordering induced by $(\prec')^{\mathbf{w}}$. So a PBW-reduction datum of $A^{\mathbf{w}}$ with respect to the ordering $(\prec')^{\mathbf{w}}$ is computable.
\end{lem}

\begin{proof}
By hypothesis there is a canonical isomorphism $\psi: \bigoplus_{\beta\in \NN^m} (\K[\ul x]/I)\ul y^\beta\to A$.
For the first claim we need to show that the $\K$-linear epimorphism
\[
 \psi^{\mathbf{w}}: \bigoplus_{\beta\in \NN^m} (\K[h,\ul x]/\ideal{h_\mathbf{w}(I)})\ul y^\beta\to A^{\mathbf{w}},\; \ol{h^c\ul x^\alpha} \ul y^\beta\mapsto \ol{h^c\ul x^\alpha \ul y^\beta}
\]
is injective: Consider $p=\sum_{c,\alpha,\beta}\ol{p_{c,\alpha,\beta}h^c\ul x^\alpha} \ul y^\beta\in \ker(\psi^{\mathbf{w}})$ (with $p_{c,\alpha,\beta}\in \K$). Because $\psi^\mathbf{w}$ is $(1,\mathbf{w})$-graded we may assume that $p_{c,\alpha,\beta}=0$ for $c+\langle (\alpha,\beta),\mathbf{w}\rangle\neq k$ for some fixed $k\in \ZZ$.
Define the $\K$-linear map
$$d_h': \bigoplus_{\beta\in \NN^m} (\K[h,\ul x]/\ideal{h_\mathbf{w}(I)})\ul y^\beta\to \bigoplus_{\beta\in \NN^m}(\K[\ul x]/I)\ul y^\beta,\quad \ol{h^c\ul x^\alpha}\ul y^\beta\mapsto \ol{\ul x^\alpha}\ul y^\beta.$$
 We see that
 $d_h\circ \psi^{\mathbf{w}}=\psi\circ d_h'$. So we obtain that $ \sum_{c,\alpha}p_{c,\alpha,\beta}\ul x^\alpha\in I$ for all $\beta\in \NN^m$. Since $\sum_{c,\alpha,\beta}{p_{c,\alpha,\beta}h^c\ul x^\alpha} \ul y^\beta$ and hence also $\sum_{c,\alpha}{p_{c,\alpha,\beta}h^c\ul x^\alpha}$ is $(1,\mathbf{w})$-homogeneous
\[
\sum_{c,\alpha}p_{c,\alpha,\beta}h^c\ul x^\alpha =h^zh_\mathbf{w}\left( \sum_{c,\alpha}p_{c,\alpha,\beta}\ul x^\alpha \right)\in \ideal{h_\mathbf{w}(I)}.
\]
for some $z\in \NN$.
 This implies $p=0$ and hence that $\psi^{\mathbf{w}}$ is injective as claimed According to \cite[Exercise 1.7.5]{SingularBook} we have $h_\mathbf{w}(\ideal{I})=\ideal{h_\mathbf{w}(I')}\subseteq \K[h,\ul x]$ since $I'$ is a Gröbner basis of $I$ with respect to $\prec_\mathbf{w}$. So the additional claim is immediate from \Cref{812}.
\end{proof}


We deduce from PBW-reduction data of $A^{\mathbf{w}}$ and $A$ a corresponding datum of the $(1,\mathbf{w})$-ho\-mo\-geniza\-tion of a given factor algebra of $A$ as explained below: 


\begin{lem}\label{876} 
Let $\mathbf{w}\in \NN^n_{>0}$ be a weight vector on $A$, 
$\prec'$ an ordering on $A$, $A^{\mathbf{w}}=(\K\langle h,\ul x\rangle ,S^{\mathbf{w}}, I^\mathbf{w},(\prec')^{\mathbf{w}})$ and $B=A/M$ a factor PBW-reduction-algebra. 
Suppose $G$ is a Gröbner basis of $M$ with respect to $\prec_\mathbf{w}$ and $G^\mathbf{w}$ is a Gröbner basis of the left $A^{\mathbf{w}}$-ideal generated by the residue classes of the elements in $h_\mathbf{w}(\tau_{A,\prec_\mathbf{w}}(G))$ with respect to $(\prec')^{\mathbf{w}}$.
Then $\mathbf{w}$ is a weight vector on $B$ and 
\[
B^{\mathbf{w}}=(\K\langle h,\ul x\rangle ,S^{\mathbf{w}},\tau_{A^{\mathbf{w}},(\prec')^{\mathbf{w}}}(G^\mathbf{w})\cup I^\mathbf{w},(\prec')^{\mathbf{w}}).
\]
In particular, PBW-reduction data for $\mathbf{w}$-homogenized factor algebras of PBW-algebras are computable. 
\end{lem}

\begin{proof}
Let $B=(\Tn,S,J,\prec)$ be a PBW-reduction datum. 
We first show that the $\KK$-linear morphism
\begin{align*}
 \psi: \K\langle h,\ul x\rangle/\ideal{h_\mathbf{w}(\lrideal{\Tn}{S\cup I} )\cup \{hx_i-x_ih\mid 1\leq i\leq n\} \cup h_\mathbf{w}(\tau_{A,\prec_\mathbf{w}}(G))}&\to B^{\mathbf{w}},\\
 \ol p&\mapsto {\ol{p}}
\end{align*}
is an isomorphism. Clearly, $\psi$ is well-defined and surjective. For the injectivity let $p\in \K\langle h,\ul x\rangle$ with $\psi(\ol p)=0$. Because $\psi$ is $(1,\mathbf{w})$-graded, we may assume that $p$ is $(1,\mathbf{w})$-homogeneous. 
Using the relations $hx_i-x_ih$, we may further assume that $p\in \sum_{k\geq 0} h^k \Tn$. By definition of $B^\mathbf{w}$ 
$$p\in \lrideal{\K\langle h,\ul x\rangle}{h_\mathbf{w}(\lrideal{\Tn}{S\cup J})\cup \{hx_i-x_ih\mid 1\leq i\leq n\}}$$
and hence $d_h(p)\in \lrideal{\Tn}{S\cup J}$.
Using the Gröbner basis $G$ we find coefficients $a\in A^G$ for a standard representation
\[
 \ol{d_h(p)}=\sum_{g\in G} a_g g \text{ with } \lE_{\prec_\mathbf{w}} (a_g)+\lE_{\prec_\mathbf{w}}(g)\preceq_{\mathbf{w}} \lE_{\prec_\mathbf{w}}(\ol{d_h(p)})\preceq_{\mathbf{w}} \lE_{\prec_\mathbf{w}}(d_h(p)).
\]
For a suitable $r\in \lrideal{\Tn}{S\cup I}$ we obtain
\[
 d_h(p)=\sum_{g\in G} \tau_{A,\prec_\mathbf{w}}(a_g)\tau_{A,\prec_\mathbf{w}}(g)+r \text{ and } \lE_{\prec_\mathbf{w}}(r)\preceq_{\mathbf{w}} \lE_{\prec_\mathbf{w}}(d_h(p)).
\]
Therefore 
\begin{equation*}
p= h^{c_p} h_\mathbf{w}(d_h(p))=\sum_{g\in G}h^{c_g'} h_\mathbf{w}(\tau_{A,\prec_\mathbf{w}}(a_g))h_\mathbf{w}(\tau_{A,\prec_\mathbf{w}}(g))+h^{c_r}h_\mathbf{w}(r) 
\end{equation*}
for suitable $c'\in \NN^{G\sqcup \{p\}\sqcup \{r\}}$ proving injectivity.

So $B^{\mathbf{w}}$ is canonically isomorphic to 
\[
 A^{\mathbf{w}}/\lrideal{A^\mathbf{w}}{\ol{h_\mathbf{w}(\tau_{A,\prec_\mathbf{w}}(G))}}
\]
and thus an application of \Cref{840} finishes the proof.
\end{proof}


We investigate now the relationship between $\prec^E$ and $(\prec^E)^{\mathbf{w}}$:
\begin{rem}\label{326} Let $\mathbf{w}\in \NN_{>0}^n$ be a weight vector on $A$, $E$ a finite set and $\prec^E$ an ordering on $A^E$. Then there exists for $e\in E$ a set $I^\mathbf{w}_e$ consisting of $(1,\mathbf{w})$-homogeneous elements such that $(\prec^E)^{\mathbf{w}}$ is a well-ordering on $(A^{\mathbf{w}})^E=(\K\langle h,\ul x\rangle,S^{\mathbf{w}},I_e^\mathbf{w},(\prec^E)^{\mathbf{w}}_e)_{e\in E}$ (see \Cref{845}). Furthermore it holds:
\begin{enumerate}

 \item\label{326a} If $\deg_{(1,\mathbf{w})}(h^\alpha \ul x^\beta (e))=\deg_{(1,\mathbf{w})}(h^{\alpha'} \ul x^{\beta'} (e'))$, then, by definition of $(\prec^E)^{\mathbf{w}}$,
 \[
 \ul x^\beta (e) \prec^E \ul x^{\beta'} (e')\text{ if and only if } h^\alpha \ul x^\beta (e) (\prec^E)^{\mathbf{w}} h^{\alpha'} \ul x^{\beta'} (e')
 \]
 for $\alpha,\alpha'\in \NN$, $\beta,\beta'\in \NN^n$ and $e,e'\in E$.
Thus, for any $(1,\mathbf{w})$-homogeneous $a\in \K[ h,\ul x]^E$, 
\[
 d_h(\lm_{(\prec^E)^{\mathbf{w}}} (a ))= \lm_{ \prec^E} ( d_h(a)).
\]

\item\label{326b} The map $\rho_{(A^{\mathbf{w}})^E,(\prec^E)^{\mathbf{w}}}$ preserves $(1,\mathbf{w})$-homogeneity since $I'_e$ for $e\in E$ and $S^{\mathbf{w}}$ are $(1,\mathbf{w})$-homogeneous.
Since the commutation relations as well as the $I_e^\mathbf{w}$ for $e\in E$ are $(1,\mathbf{w})$-homogeneous, \Cref{829} preserves $(1,\mathbf{w})$-homogeneity.

\end{enumerate}
\end{rem}


We explain now the computation of Gröbner bases with respect to non-well-orderings.


\begin{prp}\label{327}
 Let $\mathbf{w}\in \NN_{>0}^n$ be a weight vector on $A$, $E$ a finite set, $\prec^E$ an ordering on $A^E$, and $M=\lideal{A}{\ol{M'}}\subseteq A^E$ for $M'\subseteq \Pn^E$ finite.
 If the set $G\subseteq (A^{\mathbf{w}})^E$ is a Gröbner basis of $\lideal{A^{\mathbf{w}}}{\ol{h_{\mathbf{w}}(M')}}$ with respect to $(\prec^E)^{\mathbf{w}}$ consisting of $(1,\mathbf{w})$-homogeneous elements, then $ d_h (\tau_{(\prec^E)^{\mathbf{w}}}(G))$ induces a Gröbner basis of $M$ with respect to $\prec^E$.
 An analogous statement holds for two-sided modules.
\end{prp}

\begin{proof}
 We first show that $ d_h(G)\subseteq M$: For any element $g\in G\subseteq \lideal{A^{\mathbf{w}}}{ \ol{ h_{\mathbf{w}}(M')}}$ there are coefficients $a\in (A^{\mathbf{w}})^{M'}$ such that $g=\sum_{m'\in M'} a_{m'} \ol{h_{\mathbf{w}}(m')}$. Hence 
 $$ d_h(g)=
 \sum_{m'\in M'} d_h(a_{m'} ) d_h(\ol{h_{\mathbf{w}}(m')})=\sum_{m'\in M'} d_h(a_{m'} )\ol{m'}\in M.$$

 The second step is proving that $ d_h(G)$ is a Gröbner basis of $M$: For $t\in \Pn^E$ with $\ol t\in M$ there are coefficients $a\in \Tn^{M'}$ such that $\ol{t}=\sum_{m'\in M'} \ol{a_{m'}} \ol{m'}$. 
 This implies that there is $r\in \lrideal{\Tn}{S^E \cup I^E}$ such that $t=\sum_{m'\in M'} {a_{m'}} {m'}+r $ and hence we find
 $c\in \NN^{M'\sqcup \{t\}\sqcup \{r\}}$ such that 
 \[
 h^{c_t} h_{\mathbf{w}}(t)=\sum_{m'\in M'} h^{c_{m'}} h_{\mathbf{w}}(a_{m'})h_{\mathbf{w}}(m')+h^{c_{r}}h_{\mathbf{w}}(r)
 \]
showing that
\[
 \ol{ h^{c_t} h_{\mathbf{w}}(t)}\in \lideal{A^{\mathbf{w}}}{\ol{h_{\mathbf{w}}(M')}}.
\]
As $G$ is a $(1,\mathbf{w})$-homogeneous Gröbner basis and $\ol{h^{c_t} h_{\mathbf{w}}(t)}$ is $(1,\mathbf{w})$-homogeneous, we obtain a $(1,\mathbf{w})[(\deg_{(1,\mathbf{w})}(g))_{g\in G}]$-homogeneous $b\in (A^{\mathbf{w}})^G$ such that 
\[
 \ol{h^{c_t} h_{\mathbf{w}}(t)}=\sum_{g\in G} b_g g = \sum_{g\in G} \ol{ \tau_{(\prec^E)^{\mathbf{w}}_{\lcomp(g)}}(b_g)} \cdot \ol{ \tau_{(\prec^E)^{\mathbf{w}}}(g)}
\]
and 
\begin{equation}\label{894}
\lE_{(\prec^E)^{\mathbf{w}}_{\lcomp(g)}} (b_g)+\lec_{(\prec^E)^{\mathbf{w}}} (g)(\preceq^E)^{\mathbf{w}} \lec_{(\prec^E)^{\mathbf{w}}}(\ol{h^{c_t} h_{\mathbf{w}}(t)})
(\preceq^E)^{\mathbf{w}} \lec_{(\prec^E)^{\mathbf{w}}}(h^{c_t} h_{\mathbf{w}}(t)).
\end{equation}
Dehomogenizing we get
\begin{equation*}
\ol t= \ol{d_h(h^{c_t} h_{\mathbf{w}}(t)) } = \sum_{g\in G} \ol{ d_h(\tau_{(\prec^E)^{\mathbf{w}}_{\lcomp(g)}}( b_g))}\cdot \ol{ d_h(\tau_{(\prec^E)^{\mathbf{w}}}(g))}. 
\end{equation*}
By \Cref{894} and \Cref{326}.\ref{326a}, we have 
\begin{equation*}
\lE_{(\prec^E)^{\mathbf{w}}_{\lcomp(g)}} (d_h(\tau_{(\prec^E)^{\mathbf{w}}_{\lcomp(g)}}( b_g)))+\lec_{\prec^E} (d_h(\tau_{(\prec^E)^{\mathbf{w}}}(g)))
 \preceq^{E} \lec_{\prec^E}(d_h(h^{c_t} h_{\mathbf{w}}(t)))= \lec_{\prec^E}(t) 
\end{equation*}
concluding the proof.
\end{proof}
 

\begin{dfn}
Let $E$ a finite set and $\prec^E$.
\begin{enumerate}
 \item We call a well-ordering $\prec^E$ on $A^E$ \emph{computable} if we can compute $I_e$ for $e\in E$ such that $A^E=(\Tn,S, I_e,\prec^E_e)_{e\in E}$.
 \item We call the non-well-ordering $\prec^E$ on $A^E$ \emph{computable} if we can compute a weight vector $\mathbf{w}\in \NN^n_{>0}$ such that the ordering $(\prec^E)^{\mathbf{w}}$ on $(A^{\mathbf{w}})^E$ is computable.
\end{enumerate}
\end{dfn}


\Cref{394}, \Cref{327} and \Cref{326}.\ref{326b} imply

 
\begin{cor}\label{328}
 Let $E$ be a finite set. Gröbner bases with respect to any ordering on $A^E$ exist and are computable for computable orderings.
\end{cor}

 
The following algorithm summarizes the computation of such Gröbner bases. 


\begin{algorithmbis}[Given an $A$-submodule $M$ of a free $A$-module and an ordering on that free module, this algorithm computes a Gröbner basis of $M$ with respect to that ordering.]
\label[algorithm]{370}
\begin{algorithmic}[1]
\REQUIRE {A finite set $E$, an $A$-module $M=\lideal{A}{\ol{M'}}\subseteq A^E$ with $M'\subseteq \Tn^E$ finite and a computable ordering $\prec^E$ on $A^E$. }
\ENSURE A finite set $G\subseteq \Tn^E$ inducing a Gröbner basis of $M$ with respect to $\prec^E$.
\IF{ $\prec^E$ is a well-ordering}
\STATE Compute a Gröbner basis $G'$ of $M$ with respect to $\prec^E$ using \Cref{829}.
\RETURN $\tau_{A^E,\prec^E}(G)$.
\ENDIF
\STATE Compute a suitable weight vector $\mathbf{w}\in \NN_{>0}^n$ on $A$ and a PBW-reduction datum for $((A^{\mathbf{w}})^E,\allowbreak (\prec^E)^{\mathbf{w}})$.
\STATE Set $M':=h_{\mathbf{w}}(M')$.
\STATE Compute a $(1,\mathbf{w})$-homogeneous Gröbner basis $G'$ of $\lideal{A^{\mathbf{w}}}{\ol{M'}}$ over the ring $A^{{\mathbf{w}}}$ with respect to $(\prec^E)^{\mathbf{w}}$ using \Cref{829}.
\STATE Set $G:=d_h(\tau_{(\prec^E)^{\mathbf{w}}}(G))$.
\RETURN $G$.
\end{algorithmic}
\end{algorithmbis}


We can use Gröbner bases with respect to $\prec_{\mathbf{u}[\mathbf{s}]}^E$ to explicitly find generators of the filtration induced by $F^\mathbf{u}[\mathbf{s}]_\bullet A^E$ on submodules of $A^E$ if that ordering is computable (see \Cref{844}):


\begin{algorithmbis}[Given a weight vector $\mathbf{u}$ and an $A$-submodule $M $ of a free $A$-module, this algorithm computes $F^{\mathbf{u}}\bracket{s}_\bullet M$.]
\label[algorithm]{393}
\begin{algorithmic}[1]
\REQUIRE {A finite set $E $, an $A$-module $M=\lideal{A}{M'}\subseteq A^E$ with $M'$ finite, a weight vector $\mathbf{u}\in \ZZ^n$ on $A$, a shift vector $\mathbf{s}\in \ZZ^E$ and a computable ordering $\prec^E_{\mathbf{u}[\mathbf{s}]}$ on $A^E$.}
\ENSURE A finite set $G\subseteq \Tn^E$ such that $F^{\mathbf{u}}[\mathbf{s}]_\bullet M =\sum_{g \in G } F^\mathbf{u}_{\bullet-\deg_{\mathbf{u}[s]}(g )}A\cdot \ol{g}$.
\STATE Compute a set $G\subseteq \Tn^E$ inducing a Gröbner basis of $M$ with respect $\prec_{\mathbf{u}[\mathbf{s}]}^E$ by \Cref{370}. 
\RETURN $G$.
\end{algorithmic}
\end{algorithmbis}


\begin{algorithmbis}[Given a weight vector $\mathbf{u}$ and an $A$-submodule $M $ of a free $A$-module, this algorithm computes $F^{\mathbf{u}}\bracket{s}_k M$ for fixed $k\in \ZZ$.]
\label[algorithm]{340}
\begin{algorithmic}[1]
\REQUIRE {A finite set $E $, an $A$-module $M=\lideal{A}{M'}\subseteq A^E$ with $M'$ finite, a weight vector $\mathbf{u}\in \ZZ^n$ on $A$, a shift vector $\mathbf{s}\in \ZZ^E$, a computable ordering $\prec^E_{\mathbf{u}[\mathbf{s}]}$ on $A^E$, and $k\in \ZZ$.}
\ENSURE A finite set $G\subseteq \Tn^E$ such that $F^{\mathbf{u}}[\mathbf{s}]_k M =\lideal{F^{\mathbf{u}}_0A}{G}$.
\STATE Compute a set $G\subseteq \Tn^E$ inducing a Gröbner basis of $M$ with respect $\prec_{\mathbf{u}[\mathbf{s}]}^E$ by \Cref{370}. 
\STATE Set $G :=\{a\ol g\mid g\in G, a\in P^{A,\mathbf{u}}_{k-\deg_{\mathbf{u}[\mathbf{s}]}(g)}\}$. 
\RETURN $G$.
\end{algorithmic}
\end{algorithmbis}


\Cref{328} and \Cref{844} imply: 


\begin{cor}\label{395}
The filtration $F^\mathbf{u}[\mathbf{s}]_\bullet M$ is generated by a finite set for every $A$-module $M\subseteq A^E$ (with $E$ finite), every weight vector $\mathbf{s}\in \ZZ^n$ and every shift vector $\mathbf{s}\in \ZZ^E$.
\end{cor}


In the remainder of this \namecref{334} we aim for computing $\gr^{\mathbf{u}[\mathbf{s}]}\!M$ as an $\gr^{\mathbf{u}}\!A$ for an $A$-submodule $M\subseteq A^E$.


\begin{prp}\label{887}Let $\mathbf{u}\in \ZZ^n$ and $\mathbf{w}\in \NN^n_{>0}$ be weight vectors on $A$, and $A^{\mathbf{w}}=(\K\langle h,\ul x\rangle,S^{\mathbf{w}}, I^\mathbf{w},\allowbreak {(\prec_\mathbf{u})^{\mathbf{w}}})$ a PBW-reduction datum with $(1,\mathbf{w})$-homogeneous $I_\mathbf{w}$. 
\begin{enumerate}

 \item\label{887a} The natural $\K$-linear surjective map \[
\psi: \Tn\to \gr^\mathbf{u}\! A,\; x_{i_1}\cdots x_{i_k}\mapsto \ol{x_{i_1}\cdots x_{i_k}}+F^\mathbf{u}_{\deg_\mathbf{u}(x_{i_1}\cdots x_{i_k})-1}A\in \gr^{\mathbf{u}}_{\deg_\mathbf{u}(x_{i_1}\cdots x_{i_k})} A
\]
identifies the $\mathbf{u}$-graded algebra associated with $A$ with a PBW-reduction-algebra:
\[
\gr^\mathbf{u}\! A=\Tn/\ideal{\lt_{\mathbf{u}}(S)\cup \lt_{\mathbf{u}}(d_h(I^\mathbf{w}))}=(\Tn, \lt_{\mathbf{u}}(S),\allowbreak \lt_{\mathbf{u}}(d_h(I^\mathbf{w})), \prec).
\] 

\item\label{887b} If $\mathbf{u}\in \NN^n$ and $A=(\Tn,S,I_\mathbf{u},\prec_\mathbf{u})$, then $$\gr^\mathbf{u} \!A=(\Tn,\lt_{\mathbf{u}}(S),\lt_{\mathbf{u}}(I_\mathbf{u}), \prec).$$ 

\item \label{887c} Consider the finite set $E$, the ordering $\prec^E$ on $A^E$, the shift vector $\mathbf{s}\in \ZZ^E$ and the $A$-module $M\subseteq A^E$.
Using Part~\eqref{887a} we can identify
\[
 \gr^{\mathbf{u}[\mathbf{s}]}\! A^E= \Tn^E/\ideal{\lt_{\mathbf{u}}(S)^E\cup \lt_{\mathbf{u}}(d_h(I^\mathbf{w}))^E},
\]
where we put $\ol{(e)}$ in degree $\mathbf{s}_e$.

Let $G\subseteq \Pn^E$ induce a Gröbner basis of $M $ with respect to ${\prec}^E_{\mathbf{u}[\mathbf{s}]}$. 
 Under the above identification ${\lt_{\mathbf{u}[\mathbf{s}]}(G )}\subseteq \Tn^E$ then induces a Gröbner basis of the $\gr^{\mathbf{u} }\!A$-submodule $\gr^{\mathbf{u}[\mathbf{s}]}\! M$ of $ \gr^{\mathbf{u}[\mathbf{s}]}\! A^E$ with respect to $\prec^E$.

\item\label{887d} Let $M$ be a two-sided submodule of $A$ with Gröbner basis with respect to $\prec$ induced by $G\subseteq \Tn $. Then
\[
\gr^\mathbf{u} (A/M)= \gr^\mathbf{u} \!A/\gr^\mathbf{u} \!M=(\Tn, \lt_{\mathbf{u}}(S),\lt_{\mathbf{u}}(d_h(I^\mathbf{w}))\cup \rho_{\gr^\mathbf{w}\!A,\prec}(\lt_\mathbf{u}(G)),\allowbreak \prec)
\]
is a PBW-reduction-algebra.

\end{enumerate}
\end{prp}


\begin{proof}\
\begin{enumerate}

\item The map $\psi$ with kernel $\ideal{\lt_\mathbf{u}(\lrideal{\Tn}{S\cup I})}$ induces an isomorphism of $\K$-algebras $$\Tn/\ideal{\lt_\mathbf{u}(\lrideal{\Tn}{S\cup I})}\cong \gr^\mathbf{u}\! A.$$
As $d_h(I^\mathbf{w})\subseteq \lrideal{\Tn}{S\cup I}$, we have 
$\ideal{\lt_{\mathbf{u}}(S)\cup \lt_{\mathbf{u}}(d_h(I^\mathbf{w}))}\subseteq \ideal{\lt_\mathbf{u}(\lrideal{\Tn}{S\cup I})}$.
For the converse inclusion consider a $\mathbf{u}$-homogeneous $p\in \lrideal{\Tn}{\lt_\mathbf{u}(\lrideal{\Tn}{S\cup I})}$. Then there is a $p'\in \Tn$ such that $p+p' \in \lrideal{\Tn}{S\cup I} $ and $\deg_\mathbf{u}(p')<\deg_\mathbf{u}(p)$.
Using relations in $S$, we may assume $p,p'\in \Pn$.
Now we find $l,l',l''\in \NN$ such that $h^{l''}h_\mathbf{w}(p+p')=h^lh_\mathbf{w}(p)+h^{l'}h_\mathbf{w}(p')\in \lrideal{\K\langle h,\ul x\rangle}{S^{\mathbf{w}}\cup I^\mathbf{w}}$. 
By \Cref{888} we can write
\begin{equation}\label{889}
 h^{l''}h_\mathbf{w}(p+p')=\sum_{g\in I^\mathbf{w}} a_g g+\sum_{(t,s,t')\in U} tst'
 \end{equation}
for some $(1,\mathbf{w})[(\deg_{(1,\mathbf{w})}(g))_{g\in I^\mathbf{w}}]$-homogeneous $a\in \K[ h,\ul x]^{I^\mathbf{w}}$ and some finite set $U\subseteq \K\langle h,\ul x\rangle\setminus\{0\}\times S^{\mathbf{w}}\times \K\langle h,\ul x\rangle\setminus\{0\}$ such that
\[
 \lE(a_g)+\lE(g)\ (\preceq_{\mathbf{u}})^{\mathbf{w}} \lE(h^{l''}h_\mathbf{w}(p+p')) 
\]
and 
\[
 \lE(t)+\lE(s)+\lE(t')\ (\preceq_{\mathbf{u}})^{\mathbf{w}} \lE(h^{l''}h_\mathbf{w}(p+p'))
\]
with equality for some $g\in I^\mathbf{w}$.
We may assume that $t$ and $t'$ are $(1,\mathbf{w})$-homogeneous for all $(t,s,t')\in U$ and that all terms appearing in \Cref{889} are $(1,\mathbf{w})$-homogeneous of the same degree.
Dehomogenizing we obtain (see \Cref{326}\eqref{326a})
\[
p+p'=\sum_{g\in I^\mathbf{w}} d_h(a_g) d_h(g)+\sum_{(t,s,t')\in U} d_h(t)d_h(s)d_h(t')
\]
with
\begin{equation}\label{911}
 \lE(d_h(a_g))+\lE(d_h(g))\preceq_{\mathbf{u}} \lE(p+p')=\lE(p)
\end{equation}
and
\[
 \lE(d_h(t))+\lE(d_h(s))+\lE(d_h(t'))\preceq_{\mathbf{u}} \lE(p+p')=\lE(p)
\]
with equality for some $g\in I^\mathbf{w}$.
By definition of $\prec_\mathbf{u}$ there are corresponding inequalities for the $\mathbf{u}$-degree of the elements involved. By $\mathbf{u}$-homogeneity of $p$ there are ${I^\mathbf{w}}'\subseteq I^\mathbf{w}$ and $U'\subseteq U$ such that 
\[
p=\sum_{g\in {I^\mathbf{w}}'} \lt_\mathbf{u}(d_h(a_g)) \lt_\mathbf{u}(d_h(g))+\sum_{(t,s,t')\in U'}\lt_\mathbf{u}( d_h(t))\lt_\mathbf{u}(d_h(s))\lt_\mathbf{u}(d_h(t')).
\]
Hence $p\in \ideal{\Tn}{\lt_{\mathbf{u}}(S)\cup \lt_{\mathbf{u}}(d_h(I^\mathbf{w}))}$ proving the first equality.

By \Cref{837} and the $\mathbf{u}$-homogeneity of $\lrideal{\Tn}{\lt_\mathbf{u}(\lrideal{\Tn}{S\cup I})}$ it suffices to show that $\lE_{\prec}(p)\in L_{\prec}( \lt_{\mathbf{u}}(d_h(I^\mathbf{w})))$ to obtain the second equality.
To this end note that $\lE_{\prec_\mathbf{u}}(r)=\lE_{\prec_\mathbf{u}}(\lt_\mathbf{u}(r))=\lE_{\prec} (\lt_\mathbf{u}(r)) $ holds for $r\in \Pn$ and thus $\lE_{\prec}(p)=\lE_{\prec_\mathbf{u}}(p)$ by $\mathbf{u}$-homogeneity of $p$.
Choosing $g\in I^\mathbf{w}$ with equality in \Cref{911}, we obtain 
$$\lE_{\prec}(p)=\lE_{\prec}(\lt_\mathbf{u}(d_h(a_g)))+\lE_{\prec}(\lt_\mathbf{u}(d_h(g)))\in L_{\prec}( \lt_{\mathbf{u}}(d_h(I^\mathbf{w}))).$$

\item Follows by similar arguments as in Part~\eqref{887a}.

\item Consider $t\in \Pn^E$ with $0\neq \ol t\in \gr^{\mathbf{u}[\mathbf{s}]}\!M\subseteq \Tn^E/\ideal{\lt_{\mathbf{u}}(S)^E\cup \lt_{\mathbf{u}}(d_h(I^\mathbf{w}))^E}$. As that module is $\mathbf{u}[\mathbf{s}]$-graded and the ordering ${\prec}^E$ is transitive, 
we reduce to the case that $t$ is $\mathbf{u}[\mathbf{s}]$-homogeneous. Hence there exists $t'\in \Pn^E$ with $\deg_{\mathbf{u}[\mathbf{s}]}(t')<\deg_{\mathbf{u}[\mathbf{s}]}(t)$ such that $\ol{t+t'}\in M$. 
Using the Gröbner basis $G$, there are coefficients $a\in \Pn^G$ for a standard representation
\[
 \ol{t+t'}=\sum_{g\in G} \ol{a_g}\cdot \ol g\in M\quad
\text{with}\quad
 \lE_{({\prec}^E_{\mathbf{u}[\mathbf{s}]})_{\lcomp(g)}}(a_g)+\lE_{{\prec}^E_{\mathbf{u}[\mathbf{s}]}}(g)\preceq^E_{\mathbf{u}[\mathbf{s}]} \lE_{{\prec}^E_{\mathbf{u}[\mathbf{s}]}}(t+t')= \lE_{{\prec}^E_{\mathbf{u}[\mathbf{s}]}}(t).
\]
There is corresponding inequality of $\mathbf{u}[\mathbf{s}]$-degrees and we set $G':=\{g\in G\mid \deg_\mathbf{u}({a_g})+\deg_{\mathbf{u}[\mathbf{s}]}( g)=\deg_{\mathbf{u}[\mathbf{s}]}( t) \}$.
Then 
\[
 \ol{t}=\sum_{g\in G'} \ol{\lt_{\mathbf{u}}(a_g)}\cdot \ol{\lt_{\mathbf{u}[\mathbf{s}]}(g)} \in \gr^{\mathbf{u}[\mathbf{s}]}\! M
\]
and for $g\in G'$
\[
 \lE_{\prec^E_{\lcomp(g)}}(\lt_{\mathbf{u}}(a_g))+\lE_{\prec^E}(\lt_{\mathbf{u}[\mathbf{s}]}(g))\preceq^E \lE_{\prec^E}(\lt_{\mathbf{u}[\mathbf{s}]}(t+t'))= \lE_{\prec^E}(\lt_{\mathbf{u}[\mathbf{s}]}(t)).
\]

\item Using \Cref{840}, the claim follows from Parts~\ref{887a} and \ref{887c}.\qedhere

\end{enumerate}
\end{proof}


\begin{cor}\label{424}
If $A$ is a PBW-algebra and $\mathbf{u}\in \ZZ^n$ a weight vector on $A$, then $\gr^\mathbf{u} A$ is also a PBW-algebra. 
\end{cor}
 
\begin{proof}
By \Cref{394} the exists a weight vector $\mathbf w\in \NN^n_{>0}$ on $A$ and \Cref{845} implies that $A^{\mathbf{w}}$ is a PBW-algebra. Now the claim is due to \Cref{887}.\ref{887a}.
\end{proof}

 
\begin{cor}\label{902}
Consider the $\K$-algebra $\KK\langle \ul x,\ul y\rangle:=\KK\langle x_1, \allowbreak \dots,\allowbreak x_n,\allowbreak y_1,\dots,y_m\rangle$, the elementary PBW-reduction-algebra $A=(\KK\langle \ul x,\ul y\rangle, S ,I,\prec)$ and the weight vector $\mathbf{u}\in \ZZ^{n+m}$. Then $\gr^\mathbf{u} A$ is an elementary PBW-reduction-algebra.
In addition, if $I'$ is a Gröbner basis of $\lideal{\KK[\ul x]}{I}$ with respect to the ordering $\prec_\mathbf{u}$ on $A$, then 
\[
 \gr^\mathbf{u} A= (\KK\langle \ul x,\ul y\rangle,\lt_\mathbf{u}(S),\lt_\mathbf{u}( I' ),\prec).
\]
\end{cor}


\begin{proof}
Let $\mathbf{w}\in \NN^{n+m}$ be a weight vector on $A$ and $I^\mathbf{w}\subseteq \KK[h,\ul x]$ a $(1,\mathbf{w})$-homogeneous 
 Gröbner basis of $\ideal{h_\mathbf{w}(I)}\subseteq \K[h,\ul x]$ with respect to the ordering induced by $(\prec_\mathbf{u})^{\mathbf{w}}$ on $\K[h,\ul x]$.
Then \Cref{875} implies 
that 
\[
A^{\mathbf{w}}=(\K\langle h,\ul x,\ul y\rangle, \allowbreak S^{\mathbf{w}},I^\mathbf{w},(\prec_\mathbf{u})^{\mathbf{w}})
\]
According to \Cref{887}.\ref{887a} it follows that $\gr^\mathbf{u}\! A=(\K\langle \ul x,\ul y\rangle, \allowbreak \lt_\mathbf{u}(S),\lt_\mathbf{u}(d_h(I^\mathbf{w})),\prec)$ is an elementary PBW-reduction-algebra. By hypothesis and \Cref{327} both $I'$ and $d_h(I^\mathbf{w})$ are a Gröbner basis of 
 $\lideal{\K[\ul x]}{I}$ with respect to the ordering induced by $\prec_\mathbf{u}$. It follows that $$L_{\prec}(\lt_\mathbf{u}(I'))=L_{\prec}(\lt_\mathbf{u}(d_h(I^\mathbf{w})) )=\operatorname{l}_{\prec}(\ideal{ \lt_\mathbf{u}(d_h(I^\mathbf{w})) \cup \lt_\mathbf{u}(S) } \cap \K[\ul x,\ul y]), $$
 where the second equality follows from the above PBW-reduction datum of $\gr^\mathbf{u}\! A$, and 
 $$\lt_\mathbf{u}(I')\subseteq \lt_\mathbf{u}(\lideal{\K[\ul x]}{I})=\lideal{\K[\ul x]}{\lt_\mathbf{u}(d_h(I^\mathbf{w}))}\subseteq \lrideal{\K\langle \ul x,\ul y\rangle }{ \lt_\mathbf{u}(d_h(I^\mathbf{w})) \cup \lt_\mathbf{u}(S) }. $$
 The additional claim is now due to \Cref{888}. 
\end{proof}
 
 
\begin{exa}\label{903}\
\begin{enumerate}
 \item\label{903a} Consider the elementary PBW-reduction-algebra $T_X$ introduced in \Cref{577}.\ref{577a} and its weight vector $\mathbf{w}=((0)_{1\leq i\leq n},(1)_{1\leq i\leq m})$. Then 
$\gr^\mathbf{w} T_X=(\CC\langle \ul x,\ul y\rangle, \lt_\mathbf{w}(S),I_\mathbf{w} ,\prec)$ is also elementary with $\lt_\mathbf{w}(S)=\{[x_j,x_i],\allowbreak [y_l,y_k], [y_k,x_i]\mid 1\leq i\leq j\leq n, 1\leq k\leq l\leq m\}\setminus \{0\}$, where $I_\mathbf{w}$ is a Gröbner basis of $I $ with respect to the ordering induced by $\prec$ 
on $\CC[\ul x]$,
since $\lt_\mathbf{w}(I )=I $ (see \Cref{902}).
In particular, $\gr^\mathbf{w} T_X$ is a quotient algebra of the polynomial ring $\CC[\ul x,\ul y]$. and every ordering on it is computable.
\item\label{903b} We have an analogous result as in Part~\ref{903a} for the elementary PBW-reduction-algebra $T_X^V=(\CC\langle \ul x,y_1,\dots y_{m-1},z\rangle, S_V,J,\prec')$ and the weight vector $\mathbf{w}_\mathbf{v}=((0)_{1\leq i\leq n},(1)_{1\leq i\leq m})$
in the situation of \Cref{577}.\ref{577c}: Arguing as above, we have $\gr^{\mathbf{w}_\mathbf{v}} T_X^V=(\CC\langle \ul x,\ul y\rangle, \lt_\mathbf{w}(S_V),J_\mathbf{w} ,\prec')$, where $J_\mathbf{w}$ is a Gröbner basis of $\lideal{\K[\ul x]}{J}$ with respect to the ordering induced by $\prec'$ 
on $\CC[\ul x]$.
\end{enumerate}
\end{exa}
 

\textit{ In algorithms we use the symbol $\triangleright$ to mark comments.} 

 
\begin{algorithmbis}[Given a weight vector $\mathbf{u}$ on $A$ and an $A$-submodule $M$ of a free $A$-module, this algorithm computes $\gr^{\mathbf{u}}\! A$.]
\label[algorithm]{892A}
\begin{algorithmic}[1]
\REQUIRE {A weight vector $\mathbf{u}\in \ZZ^n$ on $A$ 
such that $\prec_{\mathbf{u}}$ is computable.}
\ENSURE A PBW-reduction datum $(\Tn,\lt_\mathbf{u}(S),I_\mathbf{u},\prec)$ of $\gr^{\mathbf{u}}\!A$ and a finite set $G\subseteq \Tn^E$ of $\mathbf{u}[\mathbf{s}]$-homogeneous elements whose residue classes form a set of $\gr^{\mathbf{u}}\!A$-generators of $\gr^{\mathbf{u}[\mathbf{s}]}\!M\subseteq \Tn^E/\ideal{\lt_\mathbf{u}(S)^E\cup I_\mathbf{u}^E}$.

\IF[Use \Cref{887}.\ref{887a}]{$\prec_\mathbf{u}$ is a non-well-ordering}
\STATE Find a weight vector $\mathbf{w}\in \NN^n_{>0}$ such that a PBW-reduction datum $A^{\mathbf{w}}=(\Tn,S^\mathbf{w},I^\mathbf{w}, {(\prec_{\mathbf{u}})^{\mathbf{w}}})$ is computable. 
\STATE Replace $I^\mathbf{w}$ by the set of the $(1,\mathbf{w})$-homogeneous parts of its elements.
\STATE Set $I':=d_h(I^\mathbf{w})$.
\ELSE
\STATE Compute a PBW-reduction datum $(\Tn,S,I',\prec_\mathbf{u})$ of $A$.
\ENDIF
\RETURN $(\Tn,\lt_\mathbf{u}(S),\lt_{\mathbf{u}[\mathbf{s}]}(I'),\prec)$.
\end{algorithmic}
\end{algorithmbis} 
 

\begin{algorithmbis}[Given a weight vector $\mathbf{u}$ on $A$ and an $A$-submodule $M$ of a free $A$-module, this algorithm computes $\gr^{\mathbf{u}\bracket{s}}\! M$.]
\label[algorithm]{892B}
\begin{algorithmic}[1]
\REQUIRE {A weight vector $\mathbf{u}\in \ZZ^n$ on $A$ such that $\prec_u$ is computable, a finite set $E$, an $A$-module $M=\lideal{A}{\ol{M'}}\subseteq A^E$ with $M'\subseteq \Tn^E$ finite and a shift vector $\mathbf{s}\in \ZZ^E$.
}
\ENSURE A finite set $G\subseteq \Tn^E$ of $\mathbf{u}[\mathbf{s}]$-homogeneous elements whose residue classes form a set of $\gr^{\mathbf{u}}\!A$-generators of $\gr^{\mathbf{u}[\mathbf{s}]}\!M\subseteq \Tn^E/\ideal{\lt_\mathbf{u}(\ideal{I\cup S})^E}$.
\STATE Compute a finite set $G\subseteq \Tn^E$ inducing a Gröbner basis of $M$ with respect to an ordering of type $(<^E,\prec)_{\mathbf{u}[\mathbf{s}]}$ by \Cref{370}.
\STATE Set $G:=\lt_{\mathbf{u}[\mathbf{s}]}(G)$.
\RETURN $(\Tn,\lt_\mathbf{u}(S),\lt_{\mathbf{u}[\mathbf{s}]}(I'),\prec)$ and $G$.
\end{algorithmic}
\end{algorithmbis}

\section{Interplay of weight filtrations and submodule structures of a free module over the PBW-reduction-algebra $A$}\label{397}
 
In this \namecref{397}, we consider two weight vectors $\mathbf{v}$ and $\mathbf{w}\in \ZZ^n$ on the PBW-reduction-algebra $A=(\Tn,S,I,\prec)$, which play the role of the $V$- and the order filtration. We impose certain assumptions that are motivated by Hodge theory. In particular, we assume that $\mathbf{v}$ is a \emph{$\mathbf{w}$-weight} on $A$, that is, $$F_0^\mathbf{w}A\subseteq F_0^\mathbf{v} A.$$ 
We study the interplay of the induced weight filtrations on free $A$-modules with $F^\mathbf{v}_0 A$- and $F^\mathbf{w}_0 A$-submodule structures: 
Given a finite set $E$ and $V',W'\subseteq \Tn^E$ finite subsets, the subjects of our investigation are the submodules 
\[
V:=\lideal{F_0^{\mathbf{v}}A}{\ol{V'}}\subseteq A^E \text{ and } W:=\lideal{F_0^{\mathbf{w}}A}
{\ol{W'}}\subseteq A^E.
\]
To simplify notation, we assume that $\ol v=\ol{v'}\in A^E$ for $v,v'\in V'$ implies $v=v'$ (and similarly for $W'$). 
For our algorithmic approach we need the following additional assumptions:


\begin{ass}\label{884}\
\begin{enumerate}
 \item\label{884a} We can determine a computable ordering of type $\prec'_\mathbf{v}$ on $A$. 
 
 \item\label{884b} We can compute a PBW-reduction-datum for $F_0^\mathbf{v} A$. More precisely, we can determine the kernel $K_\mathbf{v}$ of the surjective $\K$-algebra map (see \Cref{d3}.\ref{d3a})
 \[
 \phi_\mathbf{v}: A_\mathbf{v}:=\KK\langle \{y_g\mid g\in G_A^\mathbf{v}\}\rangle\to F_0^\mathbf{v}A, \; y_g\mapsto g
 \]
and a PBW-reduction datum for $A_\mathbf{v}/K_\mathbf{v}$ is computable. 

\item\label{884d} Under the assumptions of Part~\ref{884b}, assume in addition that 
 the filtration $F_\bullet^\mathbf{w}$ induced by $F_\bullet^\mathbf{w}F_0^\mathbf{v} A$ on $A_\mathbf{v}/K_\mathbf{v}$ is given by a weight vector $\mathbf{w}_\mathbf{v}$ on 
$A_\mathbf{v}/K_\mathbf{v}$ and that we can determine a computable ordering of type $\prec_{\mathbf{w}_\mathbf{v}}''$ on $A_\mathbf{v}/K_\mathbf{v}$.

\item\label{884c} For any integer $d\in \ZZ$ we can determine $\mathbf{t}_d\in \ZZ^{P_d^{A,\mathbf{v}}}$ such that $F_\bullet^\mathbf{w} F_d^\mathbf{v} A=\sum_{p\in P_d^{A,\mathbf{v}}}F^\mathbf{w}_{\bullet-(\mathbf{t}_d)_p} F_0^\mathbf{v} A\cdot \ol p$ (see \Cref{d3}.\ref{d3b}).

\item\label{884e} We have $F^\mathbf{v}_0F^\mathbf{w}_\bullet A=({F^\mathbf{v}_0 \Tn\cap F^\mathbf{w}_\bullet \Tn\cap \Pn})+\ideal{I\cup S})/\ideal{I\cup S}$.

\item\label{884f}We can determine a computable ordering of type $\prec_{\mathbf{w}}'''$ for some well-ordering $\prec'''$ on $A$. 
\end{enumerate}
\end{ass}


Note that \Cref{885}.\ref{885b} states a sufficient condition for \Cref{884}.\ref{884c}.


\begin{rem}\label{901}
Given a PBW-reduction datum of $A$ the following Gröbner basics for $A$-modules can be computed based on \Cref{829}: Gröbner bases with respect to $\prec$, module membership, intersections, and projections and syzygies (see \Cref{1001} and \Cref{834}). Moreover, using \Cref{884} we can solve the following problems:
 \begin{enumerate}
 
 \item\label{901a} \Cref{884}.\ref{884a} enables us to compute generators of the filtration $F_\bullet^\mathbf{v} M$ for an $A$-submodule $M$ of a free $A$-module. So in particular, we can determine $F_0^\mathbf{v} A$-generators of $F_k^\mathbf{v} M$ for $k\in \ZZ$.
 
 \item\label{901b} \Cref{884}.\ref{884b} ensures that we can perform the above listed Gröbner basics also over the ring $F_0^\mathbf{v} A$.
 
 \item\label{901c} A set of $F_\bullet^\mathbf{w}F_0^\mathbf{v}A$-generators of the filtration induced by $F_\bullet^\mathbf{w}A$ on $F_0^\mathbf{v}A$-submodules of free $F_0^\mathbf{v}A$-modules is computable by \Cref{884}\eqref{884d}. Similarly, we will see that \Cref{884}.\ref{884e} allows us to solve the corresponding problem for $F_0^\mathbf{v}A$-submodules of free $A$-modules. 
 
 \item\label{901d} 
 A computable ordering of type $\prec'''_\mathbf{w}$ on $A$ as in \Cref{884}.\ref{884f} enables us to realize the algebra $\gr^\mathbf{w} A$ as PBW-reduction-algebra by \Cref{892A}.
 
 \end{enumerate}
\end{rem}


The objective of this \namecref{397} is to treat the following problems:


\begin{prb}\label{399}\
\begin{enumerate}

 \item\label{399a} Module membership problem: Decide for $a\in A^E$ if $a\in V$ under \Cref{884}.\ref{884a} and \ref{884b}. 
 
 \item\label{399b} Find generators of the $F_0^\mathbf{w}A$-module $V\cap W$ under \Cref{884}.\ref{884a}-\ref{884d}. 
 
 \item\label{399c} Given that a set as in \Cref{884}.\ref{884c}, show that $V\cap F^\mathbf{w}[\mathbf{s}]_\bullet A^E$ is a well-filtered $F^\mathbf{w}_\bullet F^\mathbf{v}_0A$-module. Compute a corresponding generating set under \Cref{884}.\ref{884a}-\ref{884c}.

 \item\label{399d} Under \Cref{884} show that $\mathbf{v}$ is a weight vector on the PBW-reduction-algebra $\gr^{\mathbf{w}}A$ and represent $\gr^{\mathbf{w}[\mathbf{s}]} V$ as $F^\mathbf{v}_0 \gr^{\mathbf{w}}A$-module. 

\end{enumerate}
\end{prb}


\begin{rem}\label{439} 
In case $\mathbf{v}=(0)_{1\leq i\leq n}$, $F^{\mathbf{v}}_0A=A$ and \Cref{399}.\ref{399b} deals with the intersection of an $A$-submodule $M$ of $A^E$ with a finitely generated $F^\mathbf{u}_0 A$-submodule of $A^E$.
\end{rem}


\begin{exa}\label{14}
With regard to our applications to Hodge theory, we are particularly interested in the situation of \Cref{577} in the case \[
 \mathbf{v}=((-\delta_{n,i})_{1\leq i\leq n},(\delta_{m,i})_{1\leq i\leq m})\in \ZZ^{n+m} \text{ and } \mathbf{w}=((0)_{1\leq i\leq n},(1)_{1\leq i\leq m})\in \ZZ^{n+m} 
\]
under the condition that $x_n$ is a local coordinate (see \Cref{577}.\ref{577c}). In this case, $F^\mathbf{v}_\bullet T_X$ is the so-called $V$-filtration on $D_X(X)$ with respect to the divisor $\{x_n=0\}$ and $F^\mathbf{w}_\bullet A$ is the filtration with respect to the order of differential operators on $D_X(X)$. 

Note that we can indeed determine a PBW-reduction datum for $T_X$ by \Cref{577}.\ref{577a}. Moreover \Cref{884} is satisfied: Part~\ref{884a} follows by \Cref{394} and \Cref{875}. For Part~\ref{884b} recall that $F_0^\mathbf{v} T_X$ is isomorphic to the elementary PBW-reduction-algebra $T_X^V$ by \Cref{577}.\ref{577c}. By \Cref{812} a corresponding PBW-reduction datum can be computed.
By \Cref{313} we know that $\mathbf{w}$ induces the weight vector $\mathbf{w}_\mathbf{v}=((0)_{1\leq i\leq n},(1)_{1\leq i\leq m})$ on $T_X^V$. Again by \Cref{812} this show that Part~\ref{884d} is satisfied.
With $P_d^{T_X,\mathbf{v}}$ as in \namecref{313}, that \Cref{885}.\ref{885a} and \ref{885b} yields 
\[
 F_\bullet^\mathbf{w}F_d^\mathbf{v}T_X=\begin{cases}
 F^\mathbf{w}_\bullet F_0^\mathbf{v} T_X\cdot \ol{x_n^d}& \text{if } d\leq 0,\\
 \sum_{0\leq l\leq d}F^\mathbf{w}_{\bullet-l} F_0^\mathbf{v} T_X \cdot \ol{y_m^l}& \text{otherwise,}
 \end{cases}
\]
and Part~\ref{884c} is satisfied. \Cref{885}.\ref{885a} shows that also Part~\ref{884e} holds in this situation. Finally Part~\ref{884f} is an immediate consequence of \Cref{812}.
\end{exa}


Part of the difficulty of the above problems is due to module structures over different subrings in the chain of non-finite ring extensions
$F^\mathbf{w}_0A\subseteq F^\mathbf{v}_0A \subseteq A$.

\subsection{A one-to-one correspondence for $F^\mathbf{w}_0A$-submodules of bounded $\mathbf{v}$-degree of a free $A$-module}\label{217} 

Thus we first reduce to a problem involving only the PBW-reduction-algebra $F^\mathbf{v}_0A$ and its subalgebra $F^\mathbf{w}_0A$. To this end we consider $F^\mathbf{v}_0A$- and $F^\mathbf{w}_0A$-submodules of $A^E$ of $\mathbf{v}$-degree bounded by $d$ and lift them to a presentation of the $F^\mathbf{v}_0 A$-module $F^\mathbf{v}_d A^E$. 


\begin{rem}\label{344} 
The inclusion $F^\mathbf{w}_0 A\subseteq F^\mathbf{v}_0 A$ implies that for any finite set $N'\subseteq A^E$ 
\[
 \deg_{\mathbf{v}}(\lideal{F^\mathbf{v}_0A}{N'})= \deg_{\mathbf{v}}(\lideal{F^\mathbf{w}_0A}{N'})=\deg_{\mathbf{v}}(N')<\infty.
\]
\end{rem}


To construct the above presentation take $F^\mathbf{v}_0A$-generators $P^{A,\mathbf{v}}_d$ of $F^\mathbf{v}_dA$ (see \Cref{d3}\eqref{d3b}) and consider the $F^\mathbf{v}_0 A$-linear surjective map 
 \begin{equation}\label{62}
 \omega_{\mathbf{v},d}: F^\mathbf{v}_0A^{P^{A,\mathbf{v}}_d}\to F^\mathbf{v}_dA,\; q\mapsto \sum_{p\in P^{A,\mathbf{v}}_d} q_p\ol p.
\end{equation}
Consider $F^\mathbf{v}_0A$-generators $K_{\omega_{\mathbf{v},d}}$ of $\ker(\omega_{\mathbf{v},d})=\syz_A(\ol{P^{A,\mathbf{v}}_d})\cap F^\mathbf{v}_0A^{P^{A,\mathbf{v}}_d}$ as can be computed by \Cref{340} under \Cref{884}.\ref{884a}.
For every $a\in F^\mathbf{v}_dA$ fix a representation
\begin{equation}\label{218}
 a=\sum_{p\in P^{A,\mathbf{v}}_d} q^a_p \ol p \text{ with } q^a\in F^\mathbf{v}_0A^{P^{A,\mathbf{v}}_d},
\end{equation}
computable by \Cref{345},
to define a right inverse map of $ \omega_{\mathbf{v},d}$ 
\begin{equation}\label{61}
\upsilon_{\mathbf{v},d}: F^\mathbf{v}_dA\to F^\mathbf{v}_0A^{P^{A,\mathbf{v}}_d},\; a\mapsto q^a.
\end{equation}
 
The following one-to-one correspondence is now an immediate consequence of the homomorphism theorem:


\begin{lem}\label{343} Let $d\in \ZZ$.
There is an inclusion-, intersection- and sum-preserving one-to-one correspondence
\begin{align*}
\{F^\mathbf{w}_0A\text{-modules }K\subseteq (F^\mathbf{v}_0A^{P^{A,\mathbf{v}}_d})^E\mid \ker(\omega_{\mathbf{v},d}^E)\subseteq K\} &\leftrightarrow \{F^\mathbf{w}_0A\text{-modules }J\subseteq F^\mathbf{v}_dA^E \} \\
 \Omega_{\mathbf{v},d}^E:\;\;\;\; K & \mapsto \omega_{\mathbf{v},d}^E(K) \\
 \upsilon_{\mathbf{v},d}^E(J)+\ker(\omega_{\mathbf{v},d}^E) & \mapsfrom J \;\;\;\;: Y_{\mathbf{v},d}^E.
\end{align*}
It identifies $F^\mathbf{v}_0A$-modules on both sides. 
Moreover, if $K' \subseteq F^\mathbf{v}_dA^E$ and $\mathbf{u}\in\{\mathbf{v},\mathbf{w}\}$, then
\[
Y_{\mathbf{v},d}^E(\lideal{F^\mathbf{u}_0A}{K'})=\lideal{F^\mathbf{u}_0A}{\upsilon_{\mathbf{v},d}^E(K')}+\ker(\omega_{\mathbf{v},d}^E).
\] 
\end{lem}


The following algorithms compute images of $F^\mathbf{v}_0A$- and $F^\mathbf{w}_0A$-submodules under the above one-to-one correspondence.


\begin{algorithmbis}[Given a $\mathbf{w}$-weight $\mathbf{v}$ on $A$ and an $F^\mathbf{w}_0A$-submodule $M\subseteq A^E$, this algorithm computes $\upsilon_{\mathbf{v},d}^E(M)$ for some $d\geq \deg_\mathbf{v}(M)$.]
\label[algorithm]{225}
\begin{algorithmic}[1]
\REQUIRE {Two weight vectors $\mathbf{v},\mathbf{w}\in \ZZ^n$ on $A$ such that 
$\mathbf{v}$ is a $\mathbf{w}$-weight, a finite set $E$, a computable ordering of type $\prec'_\mathbf{v}$ on $A$, a finite set $M\subseteq \Tn^E$ and an optional natural number $d'$.}
\ENSURE Two finite subsets $M',K\subseteq (F^\mathbf{v}_0A^{P^{A,\mathbf{v}}_d})^E$, such that $Y_{\mathbf{v},d}^E(\lideal{F_0^\mathbf{u}A}{\ol M} )= \lideal{F^\mathbf{u}_0A}{M'}+\lideal{F^\mathbf{v}_0A}{K}$ for $\mathbf{u}\in \{\mathbf{v},\mathbf{w}\}$ and $\ker(\omega_{\mathbf{v},d}^E) =\lideal{F^\mathbf{v}_0A}{K}$, where $d:=\max\{\deg_\mathbf{v}(M)[,d']\}$. 
\STATE Set $d:=\max\{\deg_{\mathbf{v}}(M)[,d']\}$ and determine $P^{A,\mathbf{v}}_{d}$. 
\STATE $M':=\emptyset$.
\FOR{$m\in M$}
\STATE Find $q^m\in (F^\mathbf{v}_0A^{P^{A,\mathbf{v}}_{d}})^E$ such that $\ol m=\sum_{e\in E} \sum_{p\in P_{d}^{A,\mathbf{v}}}q^m_{e_p} \ol p (e)$ as explained in \Cref{345}.
\STATE $M':=M'\cup \{ q^m \}$.
\ENDFOR
\STATE Compute $F^\mathbf{v}_0A$-generators $K$ of $\syz_A(\ol{P^{A,\mathbf{v}}_{d}})\cap F^\mathbf{v}_0A^{P^{A,\mathbf{v}}_{d}}$ by \Cref{340} using the ordering $(\prec'_\mathbf{v},<^{P^{A,\mathbf{v}}_{d}})$ for some order $<^{P^{A,\mathbf{v}}_{d}}$ on ${P^{A,\mathbf{v}}_{d}}$.
\RETURN $M',K^E$.
\end{algorithmic}
\end{algorithmbis}


In the above algorithm, we mean by $\max\{\deg_\mathbf{v}(M)[,d']\}$ the value $\max\{\deg_\mathbf{v}(M),d'\}$ if $d'$ is defined and $\deg_\mathbf{v}(M)$ otherwise. 


\begin{algorithmbis}[Given a weight vector $\mathbf{v}$ on $A$ and a subset $M\subseteq (F^\mathbf{v}_0 A^{P_d^{A,\mathbf{v}}})^E$, this algorithm computes $\omega_{\mathbf{v},d}^E(M)$.]
\label[algorithm]{64}
\begin{algorithmic}[1]
\REQUIRE {A weight vector $\mathbf{v}\in \ZZ^n$ on $A$, an integer $d\in \ZZ$, a finite set $E$ and a finite subset $M\subseteq (F^\mathbf{v}_0 A^{P_d^{A,\mathbf{v}}})^E$.}
\ENSURE A set $M'\subseteq A^E$ such that $\omega_{\mathbf{v},d}^E(M)=M'$. 
\STATE Set $M':=\emptyset$. 
\FOR{$m\in M $}
\STATE $M':=M'\cup \{\sum_{e\in E}\sum_{p\in P_d^{A,\mathbf{v}}} m_{e_p} \ol p (e) \}$.
\ENDFOR
\RETURN $M'$.
\end{algorithmic}
\end{algorithmbis}

\subsection{Module membership problem for $F^\mathbf{v}_0A$-submodules of free $A$-modules}\label{409} 
 
In this \namecref{409}, we require that \Cref{884}.\ref{884a} and \ref{884b} is satisfied.
 Recall that $V=\lideal{F_0^\mathbf{v}A}{\ol{V'}}\subseteq A^E$ with $V'\subseteq \Tn^E$ finite and consider $a\in \Tn^E$. We explain how to check whether $
 \ol a\in V$,
which is equivalent to 
$
 \lideal{F^\mathbf{v}_0A}{\ol a}\subseteq V.
$
Since the $\mathbf{v}$-degree of the above ideals is bounded by $d:=\max\{ \deg_{\mathbf{v}}({V'}),\deg_\mathbf{v}(a)\}$ and the one-to-one correspondence in \Cref{343} is inclusion-preserving, our problem reduces to deciding whether
\[
 \lideal{F^\mathbf{v}_0A}{\upsilon_{\mathbf{v},d}^E(\ol a)}+\lideal{F^\mathbf{v}_0A}{K_{\omega_{\mathbf{v},d}}^E}\subseteq \lideal{F^\mathbf{v}_0A}{\upsilon_{\mathbf{v},d}^E(\ol{V'})}+\lideal{F^\mathbf{v}_0A}{K_{\omega_{\mathbf{v},d}}^E},
\]
which is in turn equivalent to 
\[
 \upsilon_{\mathbf{v},d}^E(\ol a)\in \lideal{F^\mathbf{v}_0A}{\upsilon_{\mathbf{v},d}^E(\ol{V'})\cup K_{\omega_{\mathbf{v},d}}^E}.
\]
The above module membership problem can be solved over the PBW-reduction-algebra $F^\mathbf{v}_0A$ by a normal form computation.
The following algorithm checks more generally whether $\lideal{F^\mathbf{v}_0A}{P}\subseteq V$ for $P\subseteq A^E$ finite.


\begin{algorithmbis}[Given a weight vector $\mathbf{v}$ on $A$ and two $F^\mathbf{v}_0A$-submodules $V,P$ of a free $A$-module, this algorithm checks if $P\subseteq V$.]
\label[algorithm]{410}
\begin{algorithmic}[1]
\REQUIRE {A weight vector $\mathbf{v}\in \ZZ^n$ on $A$, such that \Cref{884}.\ref{884a} and \ref{884b} is satisfied, a finite set $E$ and submodules $V:=\lideal{F_0^\mathbf{v}A}{\ol{V'}}$, $P:=\lideal{F_0^\mathbf{v}A}{\ol{P'}}\subseteq A^E$ with $V',P'\subseteq \Tn^E$ finite. }
\ENSURE \texttt{true} if $P\subseteq V$ and \texttt{false} otherwise.
\STATE Set $d:=\max\{\deg_{\mathbf{v}}(V'),\deg_{\mathbf{v}}(P')\}$.
\STATE Compute $P'':=\upsilon_{\mathbf{v},d}^E(\ol{P'})$, $V'':=\upsilon_{\mathbf{v},d}^E(\ol{V'})$ and $K:=K_{\omega_{\mathbf{v},d}}^E$ using \Cref{225}.
\STATE Set $J:=\lideal{F^\mathbf{v}_0A}{V''\cup K}$.
\FOR{$p''\in P''$}
\IF[Use \Cref{828} over the PBW-reduction-algebra $F^\mathbf{v}_0A $.]{$p''\notin J$}
\RETURN \texttt{false}.
\ENDIF
\ENDFOR
\RETURN \texttt{true}.
\end{algorithmic}
\end{algorithmbis}


\begin{rem}\label{571} 
 With a little extra bookkeeping the above \namecref{410} can be extended to represent $\ol{p'}$ for $p'\in P'$ as an $F_0^{\mathbf{v}} A$-linear combination of the elements of $\ol{V'}$ if $\ol{p'}\in V$.
\end{rem}

\subsection{Intersection of $F^\mathbf{v}_0A$- and $F^\mathbf{w}_0A$-submodules of a free $A$-module}\label{400} 

In this \namecref{400}, we require that \Cref{884}.\ref{884a}-\ref{884d} is satisfied. 
Based on the one-to-one correspondence of \Cref{217} we describe a method to compute generators the $F^\mathbf{w}_0A$-submodule
\[
 V\cap W\subseteq A^E,
\]
where $ V=\lideal{F^\mathbf{v}_0A}{\ol{V'}}$ and $W=\lideal{F^\mathbf{w}_0A}{ \ol{W'}}$.
Setting $d:=\max\{\deg_\mathbf{v}({V'}),\deg_\mathbf{v}(W')\}\in \ZZ$, we get by the one-to-one correspondence in \Cref{343} and \Cref{344}
\begin{equation*}
 V\cap W=\omega_{\mathbf{v},d}^E(J_{W}\cap J_{V}),
\end{equation*}
where
\begin{equation}\label{380}
 J_{W}=\lideal{F^{\mathbf{w}_\mathbf{v}}_0F^\mathbf{v}_0A}{\upsilon_{\mathbf{v},d}^E(\ol{W'})} + \lideal{F^\mathbf{v}_0A}{K_{\omega_{\mathbf{v},d}}^E}
\end{equation}
and 
\begin{equation}\label{403}
J_{V}=\lideal{F^\mathbf{v}_0A}{\upsilon_{\mathbf{v},d}^E(\ol{V'})} + \lideal{F^\mathbf{v}_0A}{K_{\omega_{\mathbf{v},d}}^E}.
\end{equation}
To ease notation we identify $F^\mathbf{v}_0A^{W'}=F^\mathbf{v}_0A^{\upsilon_{\mathbf{v},d}^E(\ol{W'})}$ and $F^\mathbf{v}_0A^{V'}=F^\mathbf{v}_0A^{\upsilon_{\mathbf{v},d}^E(\ol{V'})}$ and set $K:=K_{\omega_{\mathbf{v},d}}^E$.
Now consider the syzygy module 
\[
R:=\syz_{F^\mathbf{v}_0A}\left(\upsilon_{\mathbf{v},d}^E(\ol{W'}),\upsilon_{\mathbf{v},d}^E(\ol{V'}),K \right)\subseteq F_0^\mathbf{v} A^{W'}\oplus F_0^\mathbf{v} A^{V'}\oplus F_0^\mathbf{v} A^{K}\xrightarrow{\pi_{W'}} F_0^\mathbf{v} A^{W'}
\]
and set
\[
R':=\pi_{W'}(R) \cap F^{\mathbf{w}_\mathbf{v}}_0 F^\mathbf{v}_0A^{W'}.
\]
A set of $F^\mathbf{v}_0A$-generators of $R$ can be obtained from a Gröbner basis calculation over the PBW-reduction-algebra $F^\mathbf{v}_0A$ (see \Cref{834}). By \Cref{340} we can determine a finite set $G$ such that
$R'=\lideal{F^{\mathbf{w}_\mathbf{v}}_0F_0^\mathbf{v}A}{G}$. The intersection $V\cap W$ is obtained from $G$ as follows.


\begin{lem}\label{376}
 We have 
 \begin{equation}\label{377}
J_{W}\cap J_{V}=\lideal{F^{\mathbf{w}_\mathbf{v}}_0F_0^\mathbf{v}A}{ \{\sum_{w'\in W'}g_{w'}\upsilon_{\mathbf{v},d}^E(\ol{w'}) \mid g\in G \} } +\lideal{F^\mathbf{v}_0A}{K}
 \end{equation}
 and hence
 $V\cap W=\lideal{F^\mathbf{w}_0A}{ \{\sum_{w'\in W'}g_{w'} \ol{w'} \mid g\in G \} }$.
\end{lem}

\begin{proof}
For the non-trivial inclusion of \Cref{377} pick $q\in J_{W}\cap J_{V}$. Then there exist $a\in F^{\mathbf{w}_\mathbf{v}}_0F_0^\mathbf{v}A^{W'}$, $b\in F^\mathbf{v}_0 A^{V'}$ and $c,c'\in F^\mathbf{v}_0 A^{K}$ such that
 \[
 q=\sum_{w'\in W'} a_{w'}\upsilon_{\mathbf{v},d}^E(\ol{w'})+ \sum_{k\in K} c_k k =\sum_{v'\in V'} b_{v'}\upsilon_{\mathbf{v},d}^E(\ol{v'})+\sum_{k\in K} c'_k k. 
 \]
This implies that $(a,-b,c-c')\in R$. By the choice of $G$, there is $f\in F^{\mathbf{w}_\mathbf{v}}_0F_0^\mathbf{v}A^{G}$ such that $a=\sum_{g\in G} f_g g$ and hence $\sum_{w'\in W'} a_{w'}\upsilon_{\mathbf{v},d}^E(\ol{w'})=\sum_{g\in G}f_g\sum_{w'\in W'} g_{w'}\upsilon_{\mathbf{v},d}^E(\ol{w'})$, which is in the right hand side of \Cref{377}. The second equality follows immediately.
\end{proof}


\begin{algorithmbis}[Given a $\mathbf{w}$-weight $\mathbf{v}$ on $A$, an $F^\mathbf{v}_0A$-submodule $V$ and an $F^\mathbf{w}_0A$-submodule $W$ of a free $A$-module, this algorithm computes the intersection $V\cap W$.]
\label[algorithm]{401}
\begin{algorithmic}[1]
\REQUIRE {Two weight vectors $\mathbf{v},\mathbf{w}\in \ZZ^n$ on $A$ such that $\mathbf{v}$ is a $\mathbf{w}$-weight and such that \Cref{884}.\ref{884a}-\ref{884d} is satisfied, a finite set $E$, submodules $V:=\lideal{F^\mathbf{v}_0A}{\ol{V'}},W:=\lideal{F^\mathbf{w}_0A}{{\ol{W'}}}\subseteq A^E$ with $V',W'\subseteq \Tn^E$ finite.}
\ENSURE A finite set $G\subseteq A^E$ such that
$\lideal{F_0^{\mathbf{w}}A}{G}=V\cap W$.
\STATE Set $d:=\max\{\deg_{\mathbf{v}}(V'),\deg_{\mathbf{v}}(W')\}$.
\STATE Compute $V'':=\upsilon_{\mathbf{v},d}^E(\ol{V'})$, $W'':=\upsilon_{\mathbf{v},d}^E(\ol{W'})$ and $K:=K_{\omega_{\mathbf{v},d}}^E$ by \Cref{225}.
\STATE Compute $R:= \syz_{F^\mathbf{v}_0A}(W'',V'',K)\subseteq F_0^\mathbf{v} A^{W'}\oplus F_0^\mathbf{v} A^{V'}\oplus F_0^\mathbf{v} A^{K}$ using \Cref{829} with the setup of \Cref{834}
over the PBW-reduction-algebra $F^\mathbf{v}_0A$. 
\STATE Determine $G'$ such that $\lideal{F^{\mathbf{w}_\mathbf{v}}_0F_0^\mathbf{v}A}{G'}=\pi_{W'}(R)\cap F^{\mathbf{w}_\mathbf{v}}_0F_0^\mathbf{v} A^{W'}$
using \Cref{340} over $F^\mathbf{v}_0A$. 
\STATE Set $G:=\{\sum_{w'\in W'}g_{w'}' \ol{w'}\mid g'\in G'\}$.
\RETURN $G$.
\end{algorithmic}
\end{algorithmbis}


\begin{rem}\label{449}
By setting $\mathbf{w}:=\mathbf{v}$, \Cref{401} enables us to determine the intersection of finitely generated $F^\mathbf{v}_0 A$-modules. In this case, we do not need to apply \Cref{340}.
\end{rem}

\subsection{Induced $\mathbf{w}$-weight filtration on $F^\mathbf{v}_0A$-submodules of free $A$-modules}\label{405}

In this \namecref{405}, we require that \Cref{884}.\ref{884a}-\ref{884c} is satisfied.
We explain how to compute $F^\mathbf{w}_\bullet F^\mathbf{v}_0 A$-generators of the module
\[
 F^\mathbf{w}[\mathbf{s}]_\bullet V=V\cap F^\mathbf{w}[\mathbf{s}]_\bullet A^E,
\]
where $V=\lideal{F^\mathbf{v}_0 A}{\ol{V'}}$ with $V'\subseteq \Tn^E$ finite. 
Bounding the $\mathbf{v}$-degree by $d:=\deg_\mathbf{v} (V')$, we proceed as in \Cref{400}.
By \Cref{884}.\ref{884c} there are $P^{A,\mathbf{v}}_d$ and $\mathbf{t}_d\in \ZZ^{P^{A,\mathbf{v}}_d}$ such that 
$F_\bullet^\mathbf{w} F_d^\mathbf{v} A=\sum_{p\in P^{A,\mathbf{v}}_d} F^\mathbf{w}_{\bullet-(\mathbf{t}_d)_p} F_0^\mathbf{v}A\cdot \ol p.$
We get by \Cref{343} and \Cref{344}
\begin{equation*}
 F^\mathbf{w}[\mathbf{s}]_\bullet V=V\cap F^\mathbf{w}[\mathbf{s}]_\bullet F^\mathbf{v}_dA^E=\omega_{\mathbf{v},d}^E(J_V \cap J_{F^\mathbf{w}[\mathbf{s}]_\bullet}),
\end{equation*}
where
\begin{equation*}
J_{V}=\lideal{F^\mathbf{v}_0A}{\upsilon_{\mathbf{v},d}^E(\ol{V'})} + \lideal{F^\mathbf{v}_0A}{K_{\omega_{\mathbf{v},d}}^E}
\quad \text{and}\quad
 J_{F^\mathbf{w}[\mathbf{s}]_\bullet}= F^{\mathbf{w}_\mathbf{v}}[\mathbf{t}]_\bullet (F^\mathbf{v}_0A^{P^{A,\mathbf{v}}_d })^E + \lideal{F^\mathbf{v}_0A}{K_{\omega_{\mathbf{v},d}}^E}
\end{equation*}
with $\mathbf{t}_{e_p}=\mathbf{s}_e+(\mathbf{t}_d)_p$ for $e\in E$, $p\in P^{A,\mathbf{v}}_d$.
It follows that
\[
 J_{V}\cap J_{F^\mathbf{w}[\mathbf{s}]_\bullet}=\left( J_V
 \cap F^{\mathbf{w}_\mathbf{v}}[\mathbf{t}]_\bullet (F^\mathbf{v}_0A^{P^{A,\mathbf{v}}_d })^E\right)+\lideal{F^\mathbf{v}_0A}{K_{\omega_{\mathbf{v},d}}^E}.
\]
Applying \Cref{393} over $ F^\mathbf{v}_0 A$, we determine a finite set $G\subseteq (\Tn^{P^{A,\mathbf{v}}_d})^E$ such that
\[
J_V
 \cap F^{\mathbf{w}_\mathbf{v}}[\mathbf{t}]_\bullet (F^\mathbf{v}_0A^{P^{A,\mathbf{v}}_d })^E=\sum_{g\in G} F^{\mathbf{w}_\mathbf{v}}_{\bullet-\deg_{\mathbf{w}_\mathbf{v}[\mathbf{t}]}(g)}F^\mathbf{v}_0 A\cdot\ol g.
\]
Since $ \deg_{\mathbf{w}[\mathbf{s}]}(\omega_{\mathbf{v},d}^E(\ol g))\leq \deg_{\mathbf{w}_\mathbf{v}[\mathbf{t}]}(\ol g) \leq \deg_{\mathbf{w}_\mathbf{v}[\mathbf{t}]}(g)$ by \Cref{884}.\ref{884c}, this implies that 
\[
 F^\mathbf{w}[\mathbf{s}]_\bullet V=
 \sum_{g\in G} F^\mathbf{w}_{\bullet-\deg_{\mathbf{w}_\mathbf{v}[\mathbf{t}]}(g)}F^\mathbf{v}_0 A\cdot \omega_{\mathbf{v},d}^E(\ol g)=
 \sum_{g\in G} F^\mathbf{w}_{\bullet-\deg_{\mathbf{w}[\mathbf{s}]}(\omega_{\mathbf{v},d}^E(\ol g))}F^\mathbf{v}_0 A\cdot \omega_{\mathbf{v},d}^E(\ol g).
\]


\begin{algorithmbis}[Given a $\mathbf{w}$-weight $\mathbf{v}$ on $A$ and an $F^\mathbf{v}_0A$-submodule $V$ of a free $A$-module with shift vector $\mathbf{s}$, this algorithm computes $F^\mathbf{w}\bracket{s}_\bullet V$.]
\label[algorithm]{406}
\begin{algorithmic}[1]
\REQUIRE {Two weight vectors $\mathbf{v},\mathbf{w}\in \ZZ^n$ on $A$ such that $\mathbf{v}$ is a $\mathbf{w}$-weight and such that \Cref{884}.\ref{884a}-\ref{884c} is satisfied, a finite set $E$, a submodule $V:=\lideal{F^\mathbf{v}_0A}{\overline{V'}}
 \subseteq A^E$ with $V'\subseteq \Tn^E$ finite and a shift vector $\mathbf{s}\in \ZZ^E$.}
\ENSURE A finite set $G\subseteq A^E$ and $\mathbf{t}\in \ZZ^G$ which satisfy $F^\mathbf{w}[\mathbf{s}]_\bullet V=\sum_{g\in G} F^\mathbf{w}_{\bullet-\mathbf{t}_g}F^\mathbf{v}_0 A\cdot g=\sum_{g\in G} F^\mathbf{w}_{\bullet-\deg_{\mathbf{w}[\mathbf{s}]}(g)}F^\mathbf{v}_0 A\cdot g.$
\STATE Set $d:=\deg_{\mathbf{v}}(V')$.
\STATE Choose $P^{A,\mathbf{v}}_d$ and $\mathbf{t}_d\in \ZZ^{P^{A,\mathbf{v}}_d}$ such that $F_\bullet^\mathbf{w} F_d^\mathbf{v} A=\sum_{p\in P^{A,\mathbf{v}}_d} F^\mathbf{w}_{\bullet-(\mathbf{t}_d)_p} F_0^\mathbf{v}A\cdot \ol p$.
\STATE Compute $V'':=\upsilon_{\mathbf{v},d}^E(\overline{V'})$ and $K:=K_{\omega_{\mathbf{v},d}}^E$ using \Cref{225}.
\STATE Define the shift vector $\mathbf{t}\in (\ZZ^{ P^{A,\mathbf{v}}_d } )^E$ by $\mathbf{t}_{e_p}=\mathbf{s}_e+(\mathbf{t}_d)_p$ for $e\in E$ and $p\in P^{A,\mathbf{v}}_d$.
\STATE Find $G'\subseteq (\Tn^{ P^{A,\mathbf{v}}_d } )^E$ such that 
$\sum_{g'\in G'} F^{\mathbf{w}_\mathbf{v}}_{\bullet-\deg_{\mathbf{w}_\mathbf{v}[\mathbf{t}]}(g')}F^\mathbf{v}_0 A\cdot \ol{g'}=
\lideal{F^\mathbf{v}_0A}{V''\cup K}\cap F^{\mathbf{w}_\mathbf{v}}[\mathbf{t}]_\bullet (F^\mathbf{v}_0A^{P^{A,\mathbf{v}}_d })^E$ using \Cref{393} over $F_0^\mathbf{v}A$.
\STATE Define $\mathbf{t}'\in \ZZ^{G'}$ by $\mathbf{t}_g':=\deg_{\mathbf{w}_\mathbf{v}[\mathbf{t}]}(g')$ for $g\in G'$.
\STATE Compute $G:=\omega_{\mathbf{v},d}^E(\ol{G'})$ by applying \Cref{64} and define $\mathbf{t}''\in \ZZ^G$ by $\mathbf{t}''_g:=\min\{\mathbf{t}_{g'}'\mid g'\in G' \text{ with } \omega_{\mathbf{v},d}^E(\ol{g'})=g \}$.
\RETURN $G, \mathbf{t''}$.
\end{algorithmic}
\end{algorithmbis}


\begin{rem}\label{898}
The above algorithm implicitly computes a representative $g'\in \Tn^E$ of $g\in G$ with $\deg_{\mathbf{w}[\mathbf{s}]}(g')\leq \mathbf{t}_g$
\end{rem}

\subsection{Associated $F^\mathbf{w}_\bullet$-graded modules to $F_0^\mathbf{v}A$-submodules of free $A$-modules}\label{417} 

In this \namecref{417} we require \Cref{884}.
We explain how to express $\gr^{\mathbf{w}[\mathbf{s}]}V$ for $V=\lideal{F_0^\mathbf{v}A}{V'}$ as a finitely generated $F^\mathbf{v}_0\gr^{\mathbf{w}}\!A$-module.


\begin{prp}\label{893}
Let $\mathbf{s}\in \ZZ^E$ be a shift vector and 
 $\gr^{\mathbf{w}} \!A=(\Tn,\lt_\mathbf{w}(S), I_\mathbf{w},\prec)$ under the identification made in \Cref{887}.\ref{887a}.
 \begin{enumerate}
 \item\label{893a} The vector $\mathbf{v}$ is a weight vector on $(\Tn,\lt_\mathbf{w}(S), I_\mathbf{w},\prec)$. With $F_\bullet^\mathbf{v}\gr^{\mathbf{w}}\! A$ induced by $F_\bullet^\mathbf{v} A$ $$F_0^\mathbf{v}\gr^{\mathbf{w}}\! A= F_0^\mathbf{v}(\Tn/\ideal{\lt_\mathbf{w}(S)\cup I_\mathbf{w}}).$$
 \item\label{893b} We may consider $\gr^{\mathbf{w}[\mathbf{s}]} V$ as an $F_0^\mathbf{v}\gr^\mathbf{w} A$-submodule of $(\gr^{\mathbf w}A)^E=\Tn^E/\ideal{\lt_\mathbf{w}(S)^E\cup I_\mathbf{w}^E}$, where we put $\ol{(e)}$ in degree $\mathbf{s}_e$.
 If
 $G\subseteq \Tn^E$ is finite with $F^\mathbf{w}[\mathbf{s}]_\bullet V=\sum_{g\in G} F_{\bullet-\deg_{\mathbf{w}[\mathbf{s}]}(g)}^\mathbf{w} F_0^\mathbf{v} A\cdot \ol{g}$, then
 $\gr^{\mathbf{w}[\mathbf{s}]}V$ is $F_0^\mathbf{v}\gr^\mathbf{w} \!A$-generated by
 $
\ol{ \lt_{\mathbf{w}[\mathbf{s}]}(G)}
$ under the above identification.
 \end{enumerate}
\end{prp}


\begin{proof}\
\begin{enumerate}

 \item 
Since $\mathbf{v}$ and $\mathbf w$ are weight vectors on $A$,
$\mathbf{v}$ is a one on $\gr^{\mathbf{w}} \!A$.
The map $\psi$ in \Cref{887}.\ref{887a} induces the claimed equality by \Cref{884}.\ref{884e}.

\item 
The identification follows from $\gr_k^\mathbf{w} F_0^\mathbf{v} A=F_0^\mathbf{v} \gr_k^\mathbf{v} A$ for $k\in \ZZ$ combined with Part~\ref{893a}.
To prove the statement on generation let $\ol v\in F^\mathbf{w}[\mathbf s]_k V$ with $v\in F^\mathbf{w}[\mathbf s]_k \Tn^E$ represent a non-zero element $v'$ in $\gr^{\mathbf{w}[\mathbf s]}_k V$. By hypothesis there exits $f\in F_0^\mathbf{v} \Tn^G$ with $f_g\in F^\mathbf{w}_{k-\deg_{\mathbf{w}[\mathbf s]}(g)}\Tn$ for all $g\in G$ such that $\ol v=\sum_{g\in G} \ol{f_g}\ol g$
and hence
\[
 v=\sum_{g\in G} f_g g+p
\]
for some $p\in F^\mathbf{w}[\mathbf s]_k \ideal{I\cup S}$. Taking $\mathbf{w}[\mathbf s]$-leading terms $v'$ identifies with 
\[
\ol{\lt_{\mathbf{w}[\mathbf{s}]}(v)}= \sum_{g\in G'} \ol{\lt_{\mathbf{w}}(f_g)}\cdot \ol{\lt_{\mathbf{w}[\mathbf{s}]}( g)} \in \Tn^E/\ideal{I_\mathbf{w}^E\cup \lt_\mathbf{w}(S)^E }
\]
for a suitable subset $G'\subseteq G$.\qedhere

\end{enumerate}
\end{proof}


Note that \Cref{884}.\ref{884a}-\ref{884c} enables us to find $G$ as in the above \namecref{893} yielding the following algorithm:
 
 
\begin{algorithmbis}[Given a $\mathbf{w}$-weight $\mathbf{v}$ on $A$ and an $F^\mathbf{v}_0A$-submodule $V$ of a free $A$-module with shift vector $\mathbf{s}$, this algorithm computes $\gr^{\mathbf{w}\bracket{s}} V$.]
\label[algorithm]{423}
\begin{algorithmic}[1]
\REQUIRE {Two weight vectors $\mathbf{v},\mathbf{w}\in \ZZ^n$ on $A$ such that $\mathbf{v}$ is a $\mathbf{w}$-weight and such that \Cref{884} is satisfied, a finite set $E$,
an $F^\mathbf{v}_0A$-module $V=\lideal{F^\mathbf{v}_0A}{\ol{V'}}\subseteq A^E$ with $V'\subseteq \Tn^E$ finite and a shift vector $\mathbf{s}\in \ZZ^E$.}
\ENSURE A finite $\mathbf{w}[\mathbf{s}]$-ho\-mo\-gen\-eous set $G\subseteq \Tn^E$ inducing $F^\mathbf{v}_0\gr^{\mathbf{w}}A$-generators of $\gr^{\mathbf{w}[\mathbf{s}]}V\subseteq \Tn^E/\ideal{\lt_\mathbf{w}(\ideal{I\cup S})^E}$.
\STATE Determine a finite set $G\subseteq \Tn^E$ satisfying $F^\mathbf{w}[\mathbf{s}]_\bullet V=\sum_{g\in G} F_{\bullet-\deg_{\mathbf{w}[\mathbf{s}]}(g)}^\mathbf{w} F_0^\mathbf{v} A\cdot \ol{g}$ by \Cref{406} and \Cref{898}. 
\STATE Set $G:=\lt_{\mathbf{w}[\mathbf{s}]}(G)$.
\RETURN $G$.
\end{algorithmic}
\end{algorithmbis}

 
\begin{exa}\label{904} 
In the situation of \Cref{577}.\ref{577c} note that $F^\mathbf{v}_0 T_X\cong T_X^V $ for the weight vector $\mathbf{v}=((-\delta_{in})_{1\leq i\leq n},\allowbreak(\delta_{im})_{1\leq i\leq m})$ on $T_X$. Hence $F^\mathbf{v}_0 \gr^\mathbf{w} T_X= \gr^\mathbf{w} F^\mathbf{v}_0 T_X=\gr^\mathbf{w_\mathbf{v}} T_X^V $ for the weight vectors $\mathbf{w}=((0)_{1\leq i\leq n},(1)_{1\leq i\leq m})$ on $T_X$ and $\mathbf w_\mathbf{v}=\mathbf w$ on $T_X^V$ by \Cref{313}. In particular, a PBW-reduction datum is computable for $F^\mathbf{v}_0 \gr^\mathbf{w} T_X$ by \Cref{903}.\ref{903b}.
\end{exa}

\section{Interplay of weight filtrations on a module over the PBW-reduction-algebra $A$}\label{311} 

The purpose of this \namecref{311} is to extend the methods from the previous \namecref{397} for free $A$-modules to quotients $A^E/L$. 
In general the $\mathbf{v}$-degree is unbounded and \Cref{343} does not apply.
In many cases this problem can be solved by passing to
$F_d^\mathbf{v} L$ for a suitable integer $d$.
 
Let $A=(\Tn,S,I,\prec)$ be a PBW-reduction-algebra and $\mathbf{v}$, $\mathbf{w}\in \ZZ^n$ two weight vectors on $A$ such that $\mathbf{v}$ is a $\mathbf{w}$-weight. Given a
 finite set $E$ and $L',V',W'\subseteq A^E$ finite subsets, $L:=\lideal{A}{L'}$ and $M:=A^E/L$, we consider the $F_\bullet^\mathbf{v}A$- and $F_\bullet^\mathbf{w}A$-submodules 
\[
 V:=\lideal{F_0^{\mathbf{v}}A}{\ol{V'}}\subseteq M\quad \text{and}\quad W:=\lideal{F_0^{\mathbf{w}}A}{\ol{W'}}\subseteq M,
\]
respectively.
For any finite set $N\subseteq A^E$ or element $a\in A^E$, we denote by 
$\tilde N$ or $\tilde a$ a (set of) representatives in $\Tn^E$.

We extend our list of assumptions from \Cref{884} as follows:


\begin{ass}\label{905}
\Cref{884}.\ref{884a} and \ref{884b} holds if we replace $A$ by $\gr^\mathbf{w} A$.
\end{ass}


\begin{exa}\label{906}
In the setting of \Cref{14} \Cref{905} holds by \Cref{903}.\ref{903a}, \Cref{394}, \Cref{875} and \Cref{904}.
\end{exa}

\subsection{$F_0^\mathbf{v} A$-presentations of $F^\mathbf{v}_0A$-submodules of $A$-modules}\label{263}

In this \namecref{263}, we only require that $\mathbf{v}$ is a weight vector on $A$ and that \Cref{884}\ref{884a} holds.
To represent $V$ as a quotient of a free $F^\mathbf{v}_0A$-module we use the surjective $F^\mathbf{v}_0A$-linear map
\[
 \varphi: F^\mathbf{v}_0A^{V'}\to V,\; (v')\mapsto \overline{v'}.
\]
 It induces an isomorphism of $F^\mathbf{v}_0A$-modules $V\cong F^\mathbf{v}_0A^{V'}/\ker(\varphi)$, where 
\[
 \ker(\varphi)=\pi_{V'}(\syz_A(V',L'))\cap F^\mathbf{v}_0A^{V'}.
\]
The preceding intersection is computable by \Cref{340}. Hence we obtain:


\begin{algorithmbis}[Given a weight vector $\mathbf{v}$ on $A$ and an $F^\mathbf{v}_0 A$-submodule $V$ of a finitely presented $A$-module, this algorithm represents $V$ as a quotient of a free $F^\mathbf{v}_0 A$-module.]
\label[algorithm]{79}
\begin{algorithmic}[1]
\REQUIRE {A weight vector $\mathbf{v}\in \ZZ^n$ on $A$ such that \Cref{884}.\ref{884a} holds, a finite set $E$, an $A$-module $M:=A^E/\lideal{A}{L'}$ and a submodule $V:=\lideal{F^\mathbf{v}_0 A}{\overline{V'}}\subseteq M$ with $L', V'\subseteq A^E$ finite.}
\ENSURE A finite set $Q\subseteq F^\mathbf{v}_0 A^{V'}$ such that $ F^\mathbf{v}_0A^{V'}/\lideal{F^\mathbf{v}_0A}{Q}\cong V$ via $\overline{a}\mapsto \overline{\sum_{v'\in V'} a_{v'}v'}$.
\STATE Compute an $A$-generating set $S$ of $\syz_A(V',L')$ using \Cref{829} with the setup of \Cref{834}.
\STATE Compute an $F^\mathbf{v}_0A$-generating set $Q$ of $\lideal{A}{\pi_{V'}(S)}\cap F^\mathbf{v}_0A^{V'}$ by \Cref{340}.
\RETURN $Q$.
\vskip0.2cm
\end{algorithmic}
\end{algorithmbis}

\subsection{Module membership problem for $F^\mathbf{v}_0A$-submodules of $A$-modules}\label{70} 
 
In this \namecref{70}, we require that \Cref{884}.\ref{884a} and \ref{884b} is satisfied. 
We explain how to check for $a\in A^E$ whether
$
 \overline{a}\in V=\lideal{F_0^\mathbf{v}A}{\overline{V'}}\subseteq M=A^E/L,
$
which is equivalent to 
\[
 a\in \lideal{F^\mathbf{v}_0A}{V'}+L.
\]
Since 
 $\deg_\mathbf{v}(a),\deg_\mathbf{v}(V)\leq d:=\max\{ \deg_{\mathbf{v}}(\tilde{V'}),\deg_\mathbf{v}(\tilde{a})\}$ this, in turn, is equivalent to
\begin{equation*}
 {a}\in \lideal{F^\mathbf{v}_0A}{V'}+(L\cap F^\mathbf{v}_d A^E).
\end{equation*}
An $F^\mathbf{v}_0 A$-generating set $L''$ of the above intersection can be determined by \Cref{340}. It remains to decide whether 
\begin{equation*}
 {a}\in \lideal{F^\mathbf{v}_0A}{V'\cup L''}.
\end{equation*}
This problem is solvable by \Cref{410}.


\begin{algorithmbis}[Given a weight vector $\mathbf{v}$ on $A$ and two $F^\mathbf{v}_0 A$-submodules $V$ and $P$ of a finitely presented $A$-module, this algorithm checks if $P\subseteq V$.]
\label[algorithm]{38}
\begin{algorithmic}[1]
\REQUIRE {A weight vector $\mathbf{v}\in \ZZ^n$ on $A$ such that \Cref{884}.\ref{884a} and \ref{884b} holds, a finite set $E$, a module $M=A^E/\lideal{A}{L'}$ and submodules $V:=\lideal{F_0^\mathbf{v}A}{\ol{V'}},P:=\lideal{F_0^\mathbf{v}A}{\ol{P'}}\subseteq M$ with $L',V',P'\subseteq A^E$ finite.}
\ENSURE \texttt{true} if $P\subseteq V$ and \texttt{false} otherwise.
\STATE Set $d:=\max\{\deg_{\mathbf{v}}(\tilde{V'}),\deg_{\mathbf{v}}(\tilde{P'})\}$.
\STATE Compute a set $L''$ of $F^\mathbf{v}_0A$-generators of $\lideal{A}{L'}\cap F_d^\mathbf{v} A^E$ using \Cref{340}.
\IF[Decide by \Cref{410}.]{$P'\subseteq \lideal{F^\mathbf{v}_0A}{V'\cup L''}$} 
\RETURN \texttt{true}.
\ENDIF
\RETURN \texttt{false}.
\end{algorithmic}
\end{algorithmbis}


\begin{rem}\label{572} 
 By \Cref{571} the above \namecref{38} can be extended to represent $\ol{p'}\in P'$ as an $F_0^{\mathbf{v}} A$-linear combination of the elements of $\ol{V'}$ if $p\in V$.
\end{rem}

\subsection{Intersection of $F^\mathbf{v}_0A$- and $F^\mathbf{w}_0A$-submodules of an $A$-module}\label{219}

In this \namecref{219}, we require that \Cref{884}.\ref{884a}-\ref{884d} is satisfied.
Consider the $A$-module $M=A^E/L$ and its submodules $V=\lideal{F^\mathbf{v}_0A}{\overline{V'}}$ and $W=\lideal{F^\mathbf{w}_0A}{\overline{W'}}$. We explain how to compute the $F^\mathbf{w}_0A$-submodule
\[
 W\cap V\subseteq M.
\]
 Setting $I:=\lideal{F^\mathbf{w}_0A}{W'} \cap \left(\lideal{F^\mathbf{v}_0A}{V'}+L\right)$, we can rewrite 
\begin{equation}\label{383}
 W\cap V=(I+L)/L\subseteq M.
\end{equation}
Since $\deg_{\mathbf{v}}(\lideal{F^\mathbf{w}_0A}{W'} )\leq \deg_{\mathbf{v}}(\tilde{W'})\leq d:=\max\{\deg_{\mathbf{v}}(\tilde{V'}),\deg_{\mathbf{v}}(\tilde{W'})\}$ by \Cref{344}
\[
 I=\lideal{F^\mathbf{w}_0A}{W'} \cap \left(\lideal{F^\mathbf{v}_0A}{V'}+(L\cap F^\mathbf{v}_d A^E)\right) 
\]
is an intersection of a finitely generated $F^\mathbf{w}_0 A$-module with a finitely generated $F^\mathbf{v}_0 A$-module.
\Cref{340} yields a finite set of $F^\mathbf{v}_0A$-generators $L''$ of $L\cap F^\mathbf{v}_dA$. 
Finally
\[
 I=\lideal{F^\mathbf{w}_0A}{W'} \cap \lideal{F^\mathbf{v}_0A}{V'\cup L''}
\]
can be computed as in \Cref{400}.


\begin{algorithmbis}[Given a $\mathbf{w}$-weight $\mathbf{v}$ on $A$, an $F^\mathbf{v}_0A$-submodule $V$ and an $F^\mathbf{w}_0A$-submodule $W$ of a finitely presented $A$-module, this algorithm computes $V\cap W$.]
\label[algorithm]{378}
\begin{algorithmic}[1]
\REQUIRE {Two weight vectors $\mathbf{v},\mathbf{w}\in \ZZ^n$ on $A$ such that $\mathbf{v}$ is a $\mathbf{w}$-weight and such that \Cref{884}.\ref{884a}-.\ref{884d} is satisfied, a finite set $E$, an $A$-module $M:=A^E/\lideal{A}{L'}$, submodules $V:=\lideal{F^\mathbf{v}_0A}{\overline{V'}}$, $W:=\lideal{F^\mathbf{w}_0A}{\overline{W'}}\subseteq M$ with $L',V',W'\subseteq A^E$ finite.}
\ENSURE A finite set $G\subseteq A^E$ such that 
$V\cap W=\lideal{F_0^{\mathbf{w}}A}{\overline{G}}$.
\STATE Set $d:=\max\{\deg_{\mathbf{v}}(\tilde{V'}),\deg_{\mathbf{v}}(\tilde{W'})\}$.
\STATE Determine $F^\mathbf{v}_0A$-generators $L''$ of $\lideal{A}{L'}\cap F^\mathbf{v}_dA^E$ using \Cref{340}.
\STATE Compute a set of $F_0^{\mathbf{w}}A$-generators $G$ of $\lideal{F^\mathbf{w}_0A}{{W'}}\cap \lideal{F^\mathbf{v}_0A}{{V'\cup L''}}$ by \Cref{401}.
\RETURN $G$.
\vskip0.2cm
\end{algorithmic}
\end{algorithmbis}


The preceding approach does not allow us to reduce the computation of $F^\mathbf{w}[\mathbf{s}]_\bullet V= F^\mathbf{w}[\mathbf s]_\bullet M \cap V$ to the situation of \Cref{406}, because the $\mathbf{v}$ is in general unbounded.
Up to a fixed index $k\in \ZZ$ this is possible.
Based on \Cref{406}, one can compute a finite set $G\subseteq A^E$ and $\mathbf{t}\in \ZZ^G$ such that
\begin{equation}\label{m1}
F^\mathbf{w}[\mathbf{s}]_{k'}M \cap V=\sum_{g\in G} F^\mathbf{w}_{k'-\mathbf{t}_g} F^\mathbf{v}_0 A\cdot \ol g=\sum_{g\in G} F^\mathbf{w}_{k'-\deg_{\mathbf{w}[\mathbf{s}]}(g)} F^\mathbf{v}_0 A\cdot \ol g \text{ for } k'\leq k
\end{equation}
 and
\begin{equation}\label{m2}
F^\mathbf{w}[\mathbf{s}]_\bullet \lideal{F^\mathbf{v}_0A}{G}=\sum_{g\in G}
F^\mathbf{w}_{\bullet- \mathbf{t}_g}F^\mathbf{v}_0 A\cdot g=\sum_{g\in G}
F^\mathbf{w}_{\bullet- \deg_{\mathbf{w}[\mathbf{s}]}(g)}F^\mathbf{v}_0 A\cdot g.
\end{equation}


\begin{algorithmbis}[Given a $\mathbf{w}$-weight $\mathbf{v}$ on $A$ and an $F^\mathbf{v}_0A$-submodule $V$ of a finitely presented $A$-module with shift vector $\mathbf{s}$, this algorithm computes $F^\mathbf{w}\bracket{s}_\bullet V$ up to index $k$.] 
\label[algorithm]{444}
\begin{algorithmic}[1]
\REQUIRE {Two weight vectors $\mathbf{v},\mathbf{w}\in \ZZ^n$ on $A$ such that $\mathbf{v}$ is a $\mathbf{w}$-weight and such that \Cref{884}.\ref{884a}-\ref{884c} holds, a finite set $E$, an $A$-module $M:=A^E/\lideal{A}{L'}$, a submodule $V:=\lideal{F^\mathbf{v}_0A}{\overline{V'}}
 \subseteq M$ with $L',V'\subseteq A^E$ finite, a shift vector $\mathbf{s}\in \ZZ^E$ and $k\in \ZZ$.}
\ENSURE A finite set $G\subseteq A^E$ and $\mathbf{t}\in \ZZ^G$ satisfying
\Cref{m1,m2}.
\STATE Set $d':=\max\{ \deg_\mathbf{v}(P^{A,\mathbf{w}}_{k-\mathbf{s}_e})\mid e\in E\}$.\COMMENT{$\deg_{\mathbf{v}}(F^\mathbf{w}[\mathbf{s}]_k A^E)\leq d'$.}
\STATE Set $d:=\max\{d',\deg_{\mathbf{v}}(\tilde{V'})\}$.
\STATE Determine a set of $F^\mathbf{v}_0A$-generators $L''$ of $\lideal{A}{L'}\cap F^\mathbf{v}_dA^E$ using \Cref{340}.
\STATE Compute a finite set $G\subseteq A^E$ and $\mathbf{t}\in \ZZ^G$ satisfying $F^\mathbf{w}[\mathbf{s}]_\bullet \lideal{F^\mathbf{v}_0A}{{V'\cup L''}}=\sum_{g\in G}
F^\mathbf{w}_{\bullet- \mathbf{t}_g}F^\mathbf{v}_0 A\cdot g $ by \Cref{406}.
\RETURN $G, \mathbf{t}$.
\vskip0.2cm
\end{algorithmic}
\end{algorithmbis}


\begin{rem}\label{m3}
Due to \Cref{898} the above algorithm implicitly computes a representative $\tilde g\in \Tn^E$ of $g\in G$ with $\deg_{\mathbf{w}[\mathbf{s}]}(\tilde g)\leq \mathbf{t}_g$.
\end{rem}

\subsection{Induced $\mathbf{w}$-weight filtrations on $F^\mathbf{v}_0A$-submodules of $A$-modules}\label{71} 

In this \namecref{71}, we require that \Cref{884} and \Cref{905} are satisfied.
Recall that $V=\lideal{F_0^\mathbf{v}A}{\overline{V'}}$ is an $F^\mathbf{v}_0A$-submodule of $M=A^E/L$ with $L=\lideal{A}{L'}$ and $\mathbf{s}\in \ZZ^E$ a shift vector. As an additional hypothesis satisfied in applications to mixed Hodge modules we require that 
$F^\mathbf{w}[\mathbf{s}]_\bullet V$ is $F^\mathbf{w}_\bullet F^\mathbf{v}_0 A$-finitely generated.

Our approach is based on a general result on induced filtrations:
Let $F_\bullet R$ be a filtered $\K$-algebra and $T\subseteq R$ a subalgebra with induced filtration. Consider an $ F_\bullet R$-module $F_\bullet N$, an $R$-submodule $P\subseteq N$ and an $T$-submodule $U\subseteq N$. 
 The filtration $F_\bullet N$ induces two $F_\bullet T$-filtrations on $Q:=(U+P)/P$ as follows:
\[
\xymatrix{
& F_\bullet N \ar[dr]^-{\text{quot}}_-{\text{filt}}\ar[dl]_-{\text{subm}}^-{\text{filt}} &\\
 F_\bullet U \ar[d]_-{\text{quot}}^-{\text{filt}}& &F_\bullet (N/P)\ar[d]_-{\text{subm}}^-{\text{filt}}\\
 \llap{$F_\bullet^{q(U)}Q:=$}(F_\bullet U + P)/P\ar@{^{(}->}[rr]& & F_\bullet^s Q:=F_\bullet(N/P)\cap Q.
 }
\]
One easily sees that $F^{q(U)}_\bullet Q\subseteq F^s_\bullet Q$ and that $F^{q(U)}_\bullet Q$ depends on $U$, while $F^s_\bullet Q$ does not. 
Equality of the two filtrations can be described in of associated graded modules:


\begin{prp}\label{442}
We have $F^{q(U)}_\bullet Q=F_\bullet^s Q$ if and only if
\[
 \gr^{F}(U\cap P)= \gr^{F}U\cap \gr^{F}P
\]
under the natural identification of the above modules with submodules of $\gr^F N$. 
\end{prp}

\begin{proof}
For both equalities the inclusion of the left in the right hand side holds trivially.

 Assume that $F^{q(U)}_\bullet Q=F_\bullet^s Q$ and let $0\neq n\in \gr^{F}_k\ U\cap \gr^{F}_k P$ for $k\in \ZZ$. Then there exist $u\in F_k U$ and $p\in F_k P$ such that $n=u+F_{k-1} N=p+F_{k-1} N$. This implies $u-p\in F_{k-1}N$ and thus $\overline{u}\in Q\cap F_{k-1}( N/P)=F_{k-1}^s Q=F_{k-1}^{q(U)}Q$. Hence there is some $u'\in F_{k-1} U$ and $p'\in P$ such that $u=u'+p'$. We conclude that $p'\in P\cap U$ and $p'+F_{k-1} N=u-u'+F_{k-1} N=n$ showing the first implication.
 
 Conversely, assume $\gr^{F}(U\cap P)= \gr^{F}U\cap \gr^{F}P$ and consider $q\in U+P$ with $0\neq \overline{q}\in F^s_k Q$ for $k\in \ZZ$. By construction of $F^s_\bullet Q$, there exists $u\in U, p\in P$ such that $\overline{q}=\overline{u} $ and $u+p\in F_k N$. If $u\in F_k N$, we are done. Otherwise $p\notin F_k N$ and there is some $j>k$ such that $u+F_{j-1} N=-p+F_{j-1} N\in \gr^{F}_j U\cap \gr^{F}_j P=\gr^{F}_j(U\cap P)$. Hence there exist $n\in U\cap P$, $u'\in F_{j-1}U$ and $p'\in F_{j-1} P$ such that $u=n+u'$ and $p=-n+p'$. Then $u'+p'=u+p\in F_{k} N$, $\ol q =\ol{u'}$ and $u'\in F_{j-1} N$. Iterating this argument finishes the proof.
\end{proof}


In the following we construct an increasing sequence of finitely generated $F^\mathbf{v}_0 A$-modules $V_k\subseteq \lideal{F_0^\mathbf{v}A}{{V'}} +L=V_k+L$ such that
\[
(F^\mathbf{w}[\mathbf s]_\bullet V_k+L)/L =F^\mathbf{w}[\mathbf{s}]_\bullet^{q(V_k)} V\subseteq 
F^\mathbf{w}[\mathbf{s}]_\bullet^s V=F^\mathbf{w}[\mathbf{s}]_\bullet V
\]
becomes an equality for large $k$. By assumption $F^\mathbf{w}[\mathbf{s}]_k V$ contains $F^\mathbf{w}_\bullet F^\mathbf{v}_0 A$-generators of 
$F^\mathbf{w}[\mathbf{s}]_\bullet V$ for large $k$.
For fixed $k\in \ZZ$ \Cref{444} computes a set $V_k'\subseteq A^E$ such that
\begin{equation}\label{429}
F^\mathbf{w}[\mathbf{s}]_{k'}V=\sum_{v\in V_k'} F^\mathbf{w}_{k'-\deg_{\mathbf{w}[\mathbf{s}]}(v)} F^\mathbf{v}_0 A\cdot \ol v
\end{equation}
for $k'\leq k$ 
and 
\begin{equation}\label{445}
F^\mathbf{w}[\mathbf{s}]_\bullet \lideal{F^\mathbf{v}_0A}{V_k'}=\sum_{v\in V_k'}
F^\mathbf{w}_{\bullet- \deg_{\mathbf{w}[\mathbf{s}]}(v)}F^\mathbf{v}_0 A\cdot v.
\end{equation}
We consider only $k$ such that $F^\mathbf{w}[\mathbf{s}]_k V$ is a set of $F^\mathbf{v}_0 A$-generators of $V$. It suffices to take $k\geq \deg_{\mathbf{w}[\mathbf{s}]}(\tilde{V'})$.
For any such $k\in \ZZ$ set $V_k=\lideal{F^\mathbf{v}_0A}{V_k'}$. Then
$$F^\mathbf{w}[\mathbf{s}]_\bullet^{q({V_k})} V=\sum_{v\in V_k'}
F^\mathbf{w}_{\bullet- \deg_{\mathbf{w}[\mathbf{s}]}(v)}F^\mathbf{v}_0 A\cdot \ol v$$ is an exhaustive filtration.
 
In this situation \Cref{442} reads:


\begin{cor}\label{524}We have 
\begin{equation}\label{446}
F^\mathbf{w}[\mathbf{s}]_{\bullet}V=\sum_{v\in V_k'} F^\mathbf{w}_{\bullet-\deg_{\mathbf{w}[\mathbf{s}]}(v)} F^\mathbf{v}_0 A\cdot \ol v
\end{equation}
if and only if
\begin{equation}\label{447}
\gr^{\mathbf{w}[\mathbf{s}]}({V_k})\cap \gr^{\mathbf{w}[\mathbf{s}]}(L)=\gr^{\mathbf{w}[\mathbf{s}]}({V_k}\cap L).
\end{equation}
\end{cor} 


A PBW-reduction datum of $\gr^\mathbf{w} A$ is computable by \Cref{892A} due to \Cref{884}.\ref{884f}, and \Cref{905}.
\Cref{423} and \Cref{892B} compute $F^\mathbf{v}_0\gr^\mathbf{w}\! A$-generators of 
$\gr^{\mathbf{w}[\mathbf{s}]}(V_k)\subseteq (\gr^\mathbf{w} \!A)^E$ and $\gr^\mathbf{w}\! A$-generators of $\gr^{\mathbf{w}[\mathbf{s}]}(L)\subseteq (\gr^\mathbf{w}\! A)^E$, respectively. 
The two modules can be intersected by \Cref{401} (see \Cref{439}). 
In the same way we compute $F^\mathbf{v}_0A$-generators of $V_k\cap L$ and \Cref{423} yields $\gr^{\mathbf{w}[\mathbf{s}]}(V_k\cap L)$.
Finally \Cref{447} can be verified using \Cref{410}.

This leads to the following algorithm:


\begin{algorithmbis}[Given a $\mathbf{w}$-weight $\mathbf{v}$ on $A$, an $A$-submodule $L$ and an $F^\mathbf{v}_0 A$-submodule $V$ of a free $A$-module with shift vector $\mathbf{s}$, this algorithm checks whether the quotient and the submodule filtration induced by $F^\mathbf{w}\bracket{s}_\bullet$ on $(V+L)/L$ agree.]
\label[algorithm]{73}
\begin{algorithmic}[1]
\REQUIRE {
Two weight vectors $\mathbf{v},\mathbf{w}\in\ZZ^n $ such that $\mathbf{v}$ is a $\mathbf{w}$-weight and such that \Cref{884} and \Cref{905} are satisfied, a finite set $E$, submodules $L=\lideal{A}{L'}$ and $V= \lideal{F^\mathbf{v}_0 A}{{V'}}\subseteq A^E$ with $L',V'\subseteq A^E$ finite and a shift vector $\mathbf{s}\in \ZZ^E$.
}
\ENSURE \texttt{true} if $F^{\mathbf{w}}[\mathbf{s}]_\bullet^s (V+L/L)=F^{\mathbf{w}}[\mathbf{s}]_\bullet^{q(V)} (V+L/L)$
 and \texttt{false} otherwise.
 \STATE Compute a PBW-reduction datum of $\gr^\mathbf{w} A$ using \Cref{892A}.
 \STATE Find $\gr^\mathbf{w}\!A$-generators $L''$ of $\gr^{\mathbf{w}[\mathbf{s}]} (L)$ by \Cref{892B}. 
\STATE Compute $F^\mathbf{v}_0\gr^\mathbf{w}\!A$-generators $V''$ of $\gr^{\mathbf{w}[\mathbf{s}]} (V)$ using \Cref{423}.
\STATE Find $F^\mathbf{v}_0\gr^\mathbf{w}\!A$-generators $J$ of the intersection $\lideal{F^\mathbf{v}_0\gr^\mathbf{w}A}{V''}\cap \lideal{\gr^\mathbf{w}A}{L''} $ using \Cref{401} and \Cref{439}.
\STATE Compute $F^\mathbf{v}_0 A$-generators $K$ of $L\cap V$ by \Cref{401} and \Cref{439}.
\STATE Determine $F^\mathbf{v}_0\gr^\mathbf{w}\!A$-generators $K'$ of 
$\gr^{\mathbf{w}[\mathbf{s}]} (\lideal{F^\mathbf{v}_0A}{ K})$ using \Cref{423}.
\IF[Check by \Cref{410}.]{$J\subseteq \lideal{F^\mathbf{v}_0\gr^{\mathbf{w}}A}{K'}$}
\RETURN \texttt{true}.
\ENDIF
\RETURN \texttt{false}.
\end{algorithmic}
\end{algorithmbis}


Finally we obtain an algorithm to compute $F^\mathbf{w}[\mathbf{s}]_\bullet V$:


\begin{algorithmbis}[Given a $\mathbf{w}$-weight $\mathbf{v}$ on $A$ and an $F_0^\mathbf{v} A$-submodule $V$ of a finitely presented $A$-module with shift vector $\mathbf{s}$, this algorithm computes $F^\mathbf{w}\bracket{s}_\bullet V$ if this filtration has a finite set of generators.]
\label[algorithm]{431}
\begin{algorithmic}[1]
\REQUIRE {
Two weight vectors $\mathbf{v},\mathbf{w}\in \ZZ^n$ on $A$ such that $\mathbf{v}$ is a $\mathbf{w}$-weight and such that \Cref{884} and \Cref{905} are satisfied, a finite set $E$, an $A$-module $M:=A^E/L$ with $L=\lideal{A}{L'}$, a submodule $V=\lideal{F^\mathbf{v}_0 A}{\overline{V'}}\subseteq M$ with $L',V'\subseteq A^E$ finite and a shift vector $\mathbf{s}\in \ZZ^{E}$.
}
\ENSURE A finite set $G\subseteq A^E$ and $\mathbf{t}\in \ZZ^G$ such that $$F^\mathbf{w}[\mathbf{s}]_\bullet V=\sum_{g\in G} F^\mathbf{w}_{\bullet-\mathbf{t}_g} F^\mathbf{v}_0A \cdot \ol g=\sum_{g\in G} F^\mathbf{w}_{\bullet-\deg_{\mathbf{w}[\mathbf{s}]}(g)} F^\mathbf{v}_0A \cdot \ol g.$$
\STATE Set $k=\deg_{\mathbf{w}[\mathbf{s}]}(\tilde{V'})$.
\STATE Initialize an empty set $G\subseteq A^E$ and a dynamic vector $\mathbf{t}\in \ZZ^G$.
\WHILE[Test by \Cref{73}.]{$F^\mathbf{w}[\mathbf{s}]_\bullet V\neq \sum_{g\in G} F^\mathbf{w}_{\bullet-\mathbf{t}_g} F^\mathbf{v}_0 A\cdot {g}$}
\STATE Compute a finite set $G'\subseteq A^E$ and $\mathbf{t}'\in \ZZ^{G'}$ with
$F^\mathbf{w}[\mathbf{s}]_{k'} V =\sum_{g\in G'} F^\mathbf{w}_{k'-\mathbf{t}_g'} F^\mathbf{v}_0 A\cdot \ol g=\sum_{g\in G'} F^\mathbf{w}_{k'-\deg_{\mathbf{w}[\mathbf{s}]}(g)} F^\mathbf{v}_0 A\cdot \ol g$ for $k'\leq k$ using \Cref{444}.
\STATE Replace $G$ by $G'$ and $\mathbf{t}$ by $\mathbf{t}'$.
\STATE Increase $k$.
\ENDWHILE
\RETURN $G$, $\mathbf{t}$.
\end{algorithmic}
\end{algorithmbis}

\bibliographystyle{elsarticle-harv}
\bibliography{gbmhm}

\begin{thebibliography}{12}
\expandafter\ifx\csname natexlab\endcsname\relax\def\natexlab#1{#1}\fi
\expandafter\ifx\csname url\endcsname\relax
  \def\url#1{\texttt{#1}}\fi
\expandafter\ifx\csname urlprefix\endcsname\relax\def\urlprefix{URL }\fi

\bibitem[{Bergman(1978)}]{Bergman}
Bergman, G.~M., 1978. The diamond lemma for ring theory. Adv. in Math. 29~(2),
  178--218.
\newline\urlprefix\url{http://dx.doi.org/10.1016/0001-8708(78)90010-5}

\bibitem[{Bruns and Gubeladze(2009)}]{Bruns}
Bruns, W., Gubeladze, J., 2009. Polytopes, rings, and {$K$}-theory. Springer
  Monographs in Mathematics. Springer, Dordrecht.
\newline\urlprefix\url{http://dx.doi.org/10.1007/b105283}

\bibitem[{Bruns and Ichim(2010)}]{BrunsComp}
Bruns, W., Ichim, B., 2010. Normaliz: algorithms for affine monoids and
  rational cones. J. Algebra 324~(5), 1098--1113.
\newline\urlprefix\url{http://dx.doi.org/10.1016/j.jalgebra.2010.01.031}

\bibitem[{Bueso et~al.(2003)Bueso, G\'omez-Torrecillas, and
  Verschoren}]{BuesoBook}
Bueso, J., G\'omez-Torrecillas, J., Verschoren, A., 2003. Algorithmic methods
  in non-commutative algebra. Vol.~17 of Mathematical Modelling: Theory and
  Applications. Kluwer Academic Publishers, Dordrecht, applications to quantum
  groups.
\newline\urlprefix\url{http://dx.doi.org/10.1007/978-94-017-0285-0}

\bibitem[{Bueso et~al.(2001)Bueso, G\'{o}mez-Torrecillas, and Lobillo}]{Noeth}
Bueso, J.~L., G\'{o}mez-Torrecillas, J., Lobillo, F.~J., 2001. Computing the
  {G}elfand-{K}irillov dimension. {II}. In: Ring theory and algebraic geometry
  ({L}e\'{o}n, 1999). Vol. 221 of Lecture Notes in Pure and Appl. Math. Dekker,
  New York, pp. 33--57.

\bibitem[{G\'omez-Torrecillas and Lobillo(2000)}]{MultiFilt}
G\'omez-Torrecillas, J., Lobillo, F.~J., 2000. Global homological dimension of
  multifiltered rings and quantized enveloping algebras. J. Algebra 225~(2),
  522--533.
\newline\urlprefix\url{http://dx.doi.org/10.1006/jabr.1999.8101}

\bibitem[{Greuel and Pfister(2008)}]{SingularBook}
Greuel, G.-M., Pfister, G., 2008. A Singular introduction to commutative
  algebra, extended Edition. Springer, Berlin, with contributions by Olaf
  Bachmann, Christoph Lossen and Hans Sch\"onemann, With 1 CD-ROM (Windows,
  Macintosh and UNIX).

\bibitem[{Koch(2003)}]{Koch}
Koch, R., 2003. Affine {M}onoids, {H}ilbert {B}ases and {H}ilbert {F}unctions.
  Ph.D. thesis, Universität Osnabrück.

\bibitem[{Oaku(1996)}]{OakuAffine}
Oaku, T., 1996. Gr\"obner bases for {$D$}-modules on a non-singular affine
  algebraic variety. Tohoku Math. J. (2) 48~(4), 575--600.
\newline\urlprefix\url{http://dx.doi.org/10.2748/tmj/1178225300}

\bibitem[{Oaku and Takayama(2001)}]{OT2001}
Oaku, T., Takayama, N., 2001. Algorithms for {$D$}-modules---restriction,
  tensor product, localization, and local cohomology groups. J. Pure Appl.
  Algebra 156~(2-3), 267--308.
\newline\urlprefix\url{http://dx.doi.org/10.1016/S0022-4049(00)00004-9}

\bibitem[{Rottner(2018)}]{Rottner}
Rottner, C., 2018. Algorithmic {M}ethods for {M}ixed {H}odge {M}odules. Ph.D.
  thesis, Universität Kaiserslautern.
\newline\urlprefix\url{https://nbn-resolving.org/urn:nbn:de:hbz:386-kluedo-53651}

\bibitem[{Saito(1990)}]{Saito2}
Saito, M., 1990. Mixed {H}odge modules. Publ. Res. Inst. Math. Sci. 26~(2),
  221--333.
\newline\urlprefix\url{http://dx.doi.org/10.2977/prims/1195171082}

\end{thebibliography}
\end{document}